\documentclass{amsart}
\usepackage[latin1]{inputenc}
\usepackage[T1]{fontenc}

\usepackage[english]{babel}
\usepackage{csquotes}

\usepackage{color} 
\usepackage[colorlinks=true,linktoc=all,linkcolor=blue,citecolor=red,filecolor=pink,urlcolor=blue]{hyperref}
\usepackage[hyperref=true,style=alphabetic,citestyle=alphabetic,backend=bibtex,maxnames=100,doi=true,isbn=false,url=false,giveninits=true]{biblatex}
\bibliography{biblio}

\usepackage{wasysym}

\usepackage{enumitem}
\usepackage{graphicx}
\graphicspath{{pictures/}}

\usepackage[mathcal,mathscr]{euscript}

\usepackage{tikz}
\usetikzlibrary{arrows}

\usepackage{todonotes}

\usepackage{amssymb,amsmath,latexsym}
\theoremstyle{plain}                    
\newtheorem{theorem}{Theorem}[section]
\newtheorem{lemma}[theorem]{Lemma}
\newtheorem{proposition}[theorem]{Proposition}
\newtheorem{corollary}[theorem]{Corollary}
\theoremstyle{definition}

\newtheorem{example}[theorem]{Example}
\theoremstyle{remark}
\newtheorem{remark}[theorem]{Remark}

\numberwithin{equation}{section}

\newcommand{\zz}{\mathbb Z}

\newcommand{\qq}{\mathbb Q}
\newcommand{\rr}{\mathbb R}

\newcommand{\hh}{\mathbb H}

\newcommand{\mink}{\rr^{n,1}}

\newcommand{\so}{\mathrm{SO}_0(n,1)}
\newcommand{\oofull}{\mathrm{O}(n,1)}
\newcommand{\sofull}{\mathrm{SO}(n,1)}

\newcommand{\stab}[2]{\operatorname{Stab}_{#1}({#2})}

\newcommand{\cat}[1]{\operatorname{CAT}(#1)}

\newcommand{\barsub}[1]{\mathcal B (#1)}
\newcommand{\gromov}[1]{\mathcal G (#1)}
\newcommand{\igq}{B_n}

\newcommand{\hyperbolize}[1]{\mathcal H(#1)}

\newcommand{\hm}{M_\Gamma} 
\newcommand{\hc}{X_\Gamma} 
\newcommand{\hcalt}[1]{{#1}_\Gamma} 

\newcommand{\uchc}{\widetilde{\hc}} 
\newcommand{\ichc}{\hc '} 

\newcommand{\hq}{\square^n_\Gamma} 
\newcommand{\hqk}{\square^k_\Gamma} 
\newcommand{\uchq}{\widetilde{\hq}} 

\newcommand{\hg}{\Gamma_X} 
\newcommand{\hgq}{\Gamma_{\square^n}} 

\newcommand{\cd}{g} 
\newcommand{\cdX}{\cd_X} 
\newcommand{\cdXi}{\cdX '} 
\newcommand{\cdXu}{\widetilde \cdX} 

\newcommand{\fold}{f} 
\newcommand{\foldc}{\fold_{\Gamma}} 
\newcommand{\foldic}{\fold_{\Gamma} '} 
\newcommand{\foldutocell}{\widetilde{\foldc}} 
\newcommand{\foldutosquare}{\widetilde \fold} 

\newcommand{\piuchc}{\pi}
\newcommand{\pifromichc}{\piuchc '} 
\newcommand{\pitoichc}{\piuchc ''}
\newcommand{\pisquare}{\pi_\square}

\newcommand{\dcc}[1]{\mathcal C(#1)} 
\newcommand{\dccx}{\mathcal C (\uchc)} 

\newcommand{\mirrors}{\mathcal M} 
\newcommand{\components}{\mathcal C} 
\newcommand{\edgespace}{E_{M,C}^\varepsilon}
\newcommand{\edgespacecpt}[1]{E_{M,C_{#1}}^\varepsilon}

\newcommand{\dccm}{\mathcal C (\uchc,\mirrors)} 

\newcommand{\height}[1]{\operatorname{h}(#1)}
\newcommand{\length}[1]{\ell(#1)}
\newcommand{\lk}[1]{\operatorname{lk}\left(#1\right)}
\newcommand{\lku}[1]{\operatorname{lk}^\uparrow\left(#1\right)}
\newcommand{\lkd}[1]{\operatorname{lk}_\downarrow\left(#1\right)}

\newcommand{\shadow}[1]{\zeta_{#1}}
\newcommand{\mirrorcomplexity}[1]{\mathfrak m (#1)}
\newcommand{\mirrorcomplexitymirror}[2]{\mathfrak m (#1,#2)}

\newcommand{\tangent}[2]{\operatorname{T}_{#1}(#2)}

\makeatletter
\@namedef{subjclassname@2020}{\textup{2020} Mathematics Subject Classification}
\makeatother

\begin{document}
\title[Special cubulation of strict hyperbolization]{Special cubulation of strict hyperbolization}
\author{Jean-Fran\c{c}ois Lafont}
\address{Department of Mathematics - The Ohio State University, 231 West 18th Avenue, Columbus, OH 43210-1174, USA}
\email{jlafont@math.ohio-state.edu}
\author{Lorenzo Ruffoni}
\address{Department of Mathematics - Tufts University, 177 College Avenue, Medford, MA 02155, USA}
\email{lorenzo.ruffoni2@gmail.com}
\date{\today}
 \subjclass[2020]{20F67, 53C23, 20E26, 57Q05}
 \keywords{Strict hyperbolization; hyperbolization of polyhedra; negatively curved manifolds; special cube complexes; CAT(0) cube complexes; foldable complexes;  residually finite group; virtual algebraic fibering.}

\begin{abstract}
We prove that  the Gromov hyperbolic groups obtained by the strict hyperbolization procedure of Charney and Davis are virtually compact special, hence linear and residually finite. 
Our strategy consists in constructing an  action of a hyperbolized group on a certain dual $\cat 0$ cubical complex.
As a result, all the common applications of strict hyperbolization are shown to provide manifolds with virtually compact special fundamental group.
In particular, we obtain examples of closed negatively curved Riemannian manifolds whose fundamental groups are linear and virtually algebraically fiber.
\end{abstract}

\maketitle
\tableofcontents

\section{Introduction}\label{sec:intro}
\addtocontents{toc}{\protect\setcounter{tocdepth}{2}}

Closed aspherical manifolds occupy a central place in manifold topology. For this class of manifolds, the Borel Conjecture predicts that two such manifolds are homeomorphic if and only if they have isomorphic fundamental groups -- in other words, that the topology
is entirely encoded in the fundamental group. A challenging
problem is the question of examples. The fundamental group will always satisfy Poincar\'e Duality over $\zz$ (i.e.\ they are $PD_n$ groups), and the Wall Conjecture predicts that conversely, any $PD_n$ group is the fundamental group of an aspherical manifold. Classically, there were two sources of examples of aspherical manifolds: they either arose from Lie theory, as quotients of contractible Lie groups by discrete subgroups, or from differential geometry, as non-positively curved manifolds.

In the late 1970's, Gromov introduced two metric versions of non-positive curvature, $\cat 0$ spaces and Gromov hyperbolic spaces. Simply connected, complete, locally $\cat 0$ spaces are automatically contractible. In dimensions $\geq 4$, manifolds that support locally $\cat 0$ metrics form a new source of aspherical manifolds. Moreover, it is easy to produce such manifolds, through a process known as hyperbolization. 
This was originally outlined by Gromov in \cite{G87}, and subsequently developed by Davis and Januszkiewicz in \cite{DJ91}. 
Hyperbolization is a functorial procedure, which inputs a simplicial complex, and outputs a locally $\cat 0$ space.
In a later refinement, Charney and Davis in \cite{CD95} developed a strict hyperbolization procedure, where the output is locally $\cat{-1}$, i.e.\ admits a metric of negative curvature (as opposed to just  non-positive curvature). The hyperbolization procedures have been used to produce examples of aspherical manifolds with various unexpected properties. 
In this work we show that one can construct hyperbolizations that have some additional algebraic regularity.

\begin{theorem}\label{mainthm1}
Given a dimension $n>0$, there exists a strict hyperbolization procedure $\mathcal H$ with the following property. 
Let $K$ be any $n$-dimensional simplicial complex, which is  compact, homogeneous, and without boundary. Then the resulting hyperbolized space $\hyperbolize K$ has fundamental group $G=\pi_1(\hyperbolize K)$ which acts cocompactly on a $\cat 0$ cubical complex by cubical isometries.
\end{theorem}

Most of the paper is concerned with the proof of this result (see \S\ref{sec:strict hyperbolization} and \S \ref{sec:dual cubical complex}).
The cubical complex in this statement is $n$-dimensional, but it is not locally compact and the action is not proper. 
Nevertheless, since the hyperbolization procedure is strict (i.e.\ $\hyperbolize K$ is locally $\cat{-1}$), the fundamental group $G$ is Gromov hyperbolic.
Therefore the work of Agol, Haglund--Wise, and Groves--Manning about special cube complexes (see \cite{AG13,HW08,GM18}) can be used to extract information about $G$.
A cubical complex is \textit{special} if it admits a local isometry to the Salvetti complex of a right-angled Artin group (RAAG) (see \cite{HW08}).
A group is \textit{virtually compact special} if it has a finite index subgroup which is the fundamental group of a compact special cubical complex.

\begin{theorem}\label{mainthm2}
Given a dimension $n>0$, there exists a strict hyperbolization procedure $\mathcal H$ with the following property. 
Let $K$ be any $n$-dimensional compact simplicial complex.
Then the resulting hyperbolized space $\hyperbolize K$ has fundamental group $G=\pi_1(\hyperbolize K)$ which 
is  Gromov hyperbolic and virtually compact special.
In particular, $G$ enjoys the following properties.
\begin{enumerate}
    \item \label{item:first} $G$ virtually embeds in a right--angled Artin group (RAAG) (see \cite{HW08}).
    
    \item $G$ is linear over $\zz$ (see \cite{DJ00,HW08}), hence is residually finite.
    
    \item $G$ has separable quasiconvex subgroups (see \cite{HW08}).
    
    \item $G$ is virtually residually finite rationally solvable (RFRS) (see \cite{AG08}).
    

    \item \label{item:property T} $G$ has the Haagerup property, hence does not have property (T) (see \cite{NR97,CMV04}). 

    \item $G$ satisfies the strong Atiyah conjecture (see \cite{SC14}).
    
    \item $G$ is virtually bi-orderable (see \cite{DK92}).

    
    \item   $G$ virtually embeds in the mapping class group of a closed  surface, in a braid group, and in the group of diffeomorphisms of $\rr$ (see \cite{KO12,KK15,BKK16}).
    
    \item  $G$  admits a proper  affine action on $\rr^n$ for some $n\geq 1$ (see \cite{DGK20}).

    \item \label{item:last} $G$ admits  Anosov representations (see \cite{DFWZ23}).
    
\end{enumerate}
\end{theorem}

This is achieved in \S \ref{sec:special cubulation} via a study of the cube stabilizers for the action from Theorem~\ref{mainthm1} and using a criterion for improper actions from \cite{GM18}. The special cubical complex in Theorem~\ref{mainthm2} comes from a geometric action on a $\cat 0$ cubical complex different from the one in Theorem~\ref{mainthm1}; its dimension is in general larger than $n$ and not easy to bound.
We note that the fact that hyperbolized groups do not have property (T) was already observed by Belegradek in \cite{BE07} without using cubical methods.

The use of strict hyperbolization (as opposed to non--strict hyperbolization procedures) is crucial here. Indeed, there are closed aspherical manifolds whose fundamental group is not Gromov hyperbolic and not residually finite (see \cite{ME90,BE06}).
A well-known question by Gromov asks whether all Gromov hyperbolic groups are residually finite. 
Theorem~\ref{mainthm2} implies that the strict hyperbolization procedure introduced by Charney and Davis in \cite{CD95} does not provide counterexamples to this question.

In \cite{LRGM23} we  obtained results analogous to Theorems~\ref{mainthm1} and \ref{mainthm2} for the relative strict hyperbolization procedure that was considered in \cite{DJW01,BE07} to construct aspherical manifolds with boundary and relatively hyperbolic groups.



\subsection{The main arguments}
The hyperbolization procedure in Theorems~\ref{mainthm1} and \ref{mainthm2} is the composition of two hyperbolization procedures. 
The first one is Gromov's cylinder construction, which turns the simplicial complex $K$ into a non--positively curved cubical complex $\gromov K$ (see \S\ref{sec:Gromov construction}, and \cite[\S3.4.A]{G87}).
The second one is the strict hyperbolization procedure of Charney and Davis, which turns a non-positively curved cubical complex $X$ into a locally $\cat{-1}$  piecewise hyperbolic polyhedron $\hc$ (see \S\ref{sec:strict hyperbolization}, and \cite{CD95}). 
Here $\Gamma$ is a certain uniform arithmetic lattice of simple type in $\so = \operatorname{Isom}^+(\hh^n)$, which needs to be chosen to define the strict hyperbolization procedure.
The hyperbolization procedure in Theorem~\ref{mainthm1} is then given by $\hyperbolize K=\left( \gromov {K} \right)_\Gamma$, i.e.\ by the composition
$$K \mapsto X= \gromov K \mapsto \hc =\left( \gromov {K} \right)_\Gamma , $$
and in this paper we are mostly concerned with the study of the second part, i.e.\ the strict hyperbolization of a cubical complex.
Sections \S\ref{sec:strict hyperbolization} and \S\ref{sec:dual cubical complex} lead to the following argument.

\begin{proof}[Proof of Theorem~\ref{mainthm1}]
Let $K$ be an $n$-dimensional  simplicial complex, which is compact, homogeneous, and without boundary.
Then the cubical complex $X=\gromov K$ is an $n$-dimensional  cubical complex, which is compact, homogeneous, and without boundary.
Moreover, up to a barycentric subdivision of $K$, the cubical complex $X=\gromov K$ can be assumed to be foldable, i.e.\ to admit a combinatorial map $f:X\to \square^n$ to the standard cube which is injective on each cube (see Proposition~\ref{prop:gromov_cylinder}).

Foldability provides a collection of subspaces of $X$ that we call mirrors. 
A mirror is defined as a connected component of the full preimage in $X$ of a codimension-1 face of the standard cube $\square^n$ under the folding $f:X\to \square^n$ (see \S\ref{subsec:mirrors convexity}).
Mirrors of $X$ give rise to nice locally convex codimension-1 subspaces of the hyperbolized complex $\hc$, which we still call mirrors.
Lifting the collection of mirrors of $\hc$ to the universal cover $\uchc$ of $\hc$ provides a stratification of $\uchc$: a point is in the $k$-stratum if it is contained in $n-k$ mirrors (where $n=\dim K =\dim X =\dim \hc$). 

We construct a dual cubical complex $\dccx$ in which vertices are given by cells in this stratification, and edges correspond to codimension-1 inclusion of cells (see \S\ref{sec:dual cubical complex}).
The complex $\dccx$ comes with a natural height function on its vertices, recording the dimension of the corresponding cell. 
In particular, the link of each vertex splits into an ascending sublink and a descending sublink. 
The former is flag because the cubical complex $X=\gromov K$ is non--positively curved, and the latter is flag because of a Helly property satisfied by collections of pairwise orthogonal hyperplanes in $\hh^n$ (see Lemma~\ref{lem:helly_hyperbolic_space}).
It follows that links of vertices in $\dccx$ are flag (see Proposition~\ref{prop:flag_links}), hence $\dccx$ is a non--positively curved cubical complex.
Moreover, the separation properties of the collection of mirrors (see \S\ref{subsec:separation}) imply that $\dccx$ is simply--connected, hence $\cat 0$ (see Theorem~\ref{thm:dual cubical complex is CAT(0)}).

Finally, note that the action of $G=\pi_1(\hc)=\pi_1(\hyperbolize K)$ on $\uchc$ by deck transformations induces an action of $G$ on $\dccx$, as desired (see Lemma~\ref{lem:dual action is cocompact}).
\end{proof}

The reader should note that the dimension of the dual cubical complex $\dccx$ in Theorem~\ref{mainthm1} is the same as the dimension of the input simplicial complex $K$.
Moreover, the Charney-Davis strict hyperbolization procedure from \cite{CD95} relies on the careful choice of a suitable arithmetic lattice $\Gamma$. Such a lattice can be chosen with a certain flexibility, and the proof of Theorem~\ref{mainthm1} works for any choice of $\Gamma$ for which the strict-hyperbolization procedure is defined.

The action from Theorem~\ref{mainthm1} is further studied in \S\ref{sec:special cubulation}.

\begin{proof}[Proof of Theorem~\ref{mainthm2}]
First of all, let us prove the theorem under the hypothesis and in the setting of Theorem~\ref{mainthm1}, i.e.\ with the additional assumption that $K$ is homogeneous and without boundary. 
Then the Gromov hyperbolic group $G=\pi_1(\hc)=\pi_1(\hyperbolize K)$ acts on the dual $\cat 0$ cubical complex $ \dccx$ cocompactly and by cubical isometries.

The action is not proper, but the cube stabilizers can be identified with suitable cell stabilizers for the action of $G$  by deck transformations on the universal cover $\uchc$ of $\mathcal H(K)$ (see \S\ref{subsec:action}).
These stabilizers are quasiconvex subgroups both of $G$ and of $\Gamma$ (see \S\ref{subsec:quasiconvex and vcs}).
Arithmetic lattices like $\Gamma$ are known to be virtually compact special by \cite{HW12}. In particular, we obtain that cell stabilizers for the action of $G$ on $\dccx$ are virtually compact special.
It then follows from \cite[Theorem D]{GM18} that $G$ itself is virtually compact special (see Theorem~\ref{thm:main_vcs}).

Finally, let us prove the theorem for an arbitrary compact simplicial complex $K$, without additional assumptions.
To this end, let $K'$ be a compact and homogeneous simplicial complex in which $K$ embeds as a subcomplex. (For instance, first embed $K$ in the complete simplex on its vertex set, and then embed this simplex in a triangulation $K'$ of a sphere.)
Since Gromov's cylinder construction maps subcomplexes to locally convex subcomplexes, and the Charney-Davis hyperbolization preserves local convexity, we have that $\mathcal H(K)$ is a locally convex subspace of $\mathcal H(K')$.
It follows that $G=\pi_1(\mathcal H(K))$ is a quasiconvex subgroup of $G'=\pi_1(\mathcal H(K'))$.
But $G'$ is hyperbolic and virtually compact special by the first part of this proof, hence $G$ is virtually compact special too by Lemma~\ref{lem:quasiconvex_in_special}.
\end{proof}

For the sake of clarity: the cubical complex that witnesses the specialness of $G$ is {\bf not} the cubical complex from Theorem~\ref{mainthm1}. It is obtained via the construction in  \cite{GM18}, and its dimension is in general higher than  $n=\dim K$.
One of the benefits of working with the dual $\cat 0$ cubical complex $\dccx$ (as opposed to other available $\cat 0$ cubical complexes, such as $\widetilde X$) is that the stabilizers for the action of $G$ on $\dccx$ can be related to the stabilizers for the action on $\uchc$, which are more geometric in nature and easier to understand.



\subsection{Classical applications of hyperbolization procedures}
The interest in hyperbolization procedures is that they can be used to construct closed aspherical manifolds with various interesting properties.
As a result of our Theorem~\ref{mainthm2}, many applications of the strict hyperbolization procedure  introduced by Charney and Davis in \cite{CD95} can be obtained with additional algebraic features (e.g.\ the properties \eqref{item:first}-\eqref{item:last} listed in Theorem~\ref{mainthm2}). 
We now collect some of these applications.

\subsubsection{Riemannian hyperbolization}\label{subsec:ontaneda}
The strict hyperbolization procedure introduced by Charney and Davis in \cite{CD95} outputs a space with a  metric which is locally $\cat{-1}$ and piecewise hyperbolic: the space is obtained by gluing together copies of the hyperbolizing cube $\hq$ (see \S\ref{sec:strict hyperbolization}).
When the cell complex $X$ used in the hyperbolization procedure is  homeomorphic to a smooth manifold,  the hyperbolized complex $\hc$ is homeomorphic to a manifold too, but the locally $\cat{-1}$ metric can a priori have singularities where the boundaries of different copies of the hyperbolizing cube $\hq$ meet.
It was recently shown by Ontaneda in \cite{O20} that the construction can be tweaked in such a way that the manifold $\hc$ supports a smooth Riemannian metric with strictly negative sectional curvatures (possibly with respect to a different smooth structure).

This was used in \cite[Corollary 5]{O20} to construct examples in any dimension $n\geq 4$ of closed Riemannian $n$--manifolds of pinched negative curvature which are ``new'' in the sense that they are not homeomorphic to any of the previously known examples of Riemannian manifold of negative curvature, such as closed real hyperbolic manifolds (or more generally locally symmetric spaces of rank $1$), or the Gromov-Thurston branched covers in \cite{GT87}, or the examples of Mostow-Siu in \cite{MS80} or Deraux in \cite{DE05}.
These manifolds are also distinct from the recent examples constructed by Stover--Toledo in \cite{ST21a,ST21b}, as the latter are K\"ahler, while the result of strict  hyperbolization cannot be K\"ahler by \cite[Theorem 1.8]{BE07}.
Our construction does not require the smoothness provided by Ontaneda's work, but it is compatible with it, so we get the following. 
\begin{corollary}\label{cor:ontaneda_special_manifolds}
For any $\varepsilon >0$ and $n\geq 4$ there are closed Riemannian $n$-manifolds with the following properties:
\begin{itemize}
    \item they have sectional curvatures in the interval $[-1-\varepsilon, -1]$;
    \item they are not homeomorphic to a locally symmetric space of rank 1, or one of the manifolds constructed by Gromov--Thurston, Mostow--Siu, Deraux, or Stover--Toledo;
    \item their fundamental groups are Gromov hyperbolic and virtually compact special (in particular, they satisfy properties  \eqref{item:first}-\eqref{item:last} in Theorem~\ref{mainthm2}).
\end{itemize} 
\end{corollary}

\begin{remark}\label{rem:ontaneda_special_groups}
Thanks to the solution of the Borel Conjecture for closed aspherical $n$-manifolds with Gromov hyperbolic fundamental group in dimension $n\geq 5$ (see Bartels-L\"uck in \cite{BL12}), the fundamental groups of these manifolds provide examples of Gromov hyperbolic groups that are not isomorphic to lattices in $\sofull$ or the other real simple Lie groups of rank $1$.
While it is not a priori clear from their construction whether these groups are linear, they actually turn out to be virtually compact special, hence linear over $\zz$ and residually finite.
We note that Giralt proved in \cite{GI17} that the fundamental groups of the Gromov-Thurston manifolds are also virtually compact special. 
\end{remark}

Similarly, other applications obtained by Ontaneda in \cite{O20} can be taken to have additional algebraic features. For example, we have the following versions of Corollary 2  in \cite{O20}.

\begin{corollary}
Let $\varepsilon>0$. The cohomology ring of any finite CW--complex embeds in the cohomology ring of a closed Riemannian manifold which has sectional curvatures in $[-1-\varepsilon,-1]$ and whose fundamental group is Gromov hyperbolic and virtually compact special (hence satisfies properties  \eqref{item:first}-\eqref{item:last} in Theorem~\ref{mainthm2}). In particular, it 
can be embedded into the cohomology ring of a Poincar\'e Duality subgroup of $\operatorname{SL}_N(\zz)$ (for $N$ large).
\end{corollary}

For another application in the spirit of Corollary~\ref{cor:ontaneda_special_manifolds} see \cite{RU23}, where a relative version of the Charney-Davis hyperbolization procedure is used to construct closed manifolds with hyperbolic fundamental group that do not admit any real projective or flat conformal structures, in any dimension at least $5$.
It follows from the present work (or from \cite{LRGM23}) that the fundamental groups of these manifolds are virtually compact special too.

\subsubsection{Pathological aspherical manifolds}
Davis and Januszkiewicz used the hyperbolization procedures to construct aspherical manifolds exhibiting a variety of pathological behavior (see \cite{DJ91}). 
As a consequence of our Theorem~\ref{mainthm2},
these examples can now be constructed to have the added property that their fundamental groups are virtually compact special, hence satisfy properties \eqref{item:first}-\eqref{item:last} from Theorem~\ref{mainthm2}.
For the convenience of the reader, we collate some of their examples.

\begin{corollary}\label{cor:pathological}
It is possible to construct (topological) manifolds of the following types which are piecewise hyperbolic and locally $\cat {-1}$.
\begin{itemize}

\item A closed $4$-manifold which is not homotopy equivalent to any PL $4$-manifold (see \cite[\S 5a]{DJ91}).

\item For $n=4k,k\geq 2$, a closed  $n$-manifold which is not homotopy equivalent to any smooth manifold (see \cite[Example 5.2]{BLW10}).

\item For $n\geq 5$, a closed  $n$-manifold whose universal cover is not homeomorphic to $\rr^n$ (see \cite[\S 5b]{DJ91}).

\item For $n\geq 5$, a closed  $n$-manifold whose universal cover is homeomorphic to $\rr^n$, but whose boundary at infinity is not homeomorphic to $S^{n-1}$ (see \cite[\S 5c]{DJ91}).

\end{itemize}
Moreover, in all these examples, the fundamental groups of the manifolds are Gromov hyperbolic and virtually compact special (in particular, they satisfy properties \eqref{item:first}-\eqref{item:last} from Theorem~\ref{mainthm2}).
\end{corollary}

\begin{remark}
Concerning the first example in Corollary~\ref{cor:pathological}, taking products with tori yields examples in all dimensions $n\geq 4$ of closed aspherical $n$-manifolds not homotopy equivalent
to any PL $n$-manifold. These manifolds will have fundamental group which is linear over $\zz$, but when $n\geq 5$ will only support a locally $\cat 0$ metric due to the product structure. It would be interesting to produce examples in dimensions $n\geq 5$ which support locally $\cat {-1}$ metrics. 
\end{remark}

\subsubsection{Representing cobordism classes}\label{subsec:cobordism}
As another application, we can obtain representatives for cobordism classes that are both topologically and algebraically nice.

\begin{corollary}\label{cor:cobordism}
Let $M$ be an arbitrary closed smooth manifold. Then $M$ is cobordant to an aspherical manifold $M'$, where $\pi_1(M')$ is a Gromov hyperbolic and virtually compact special (in particular, it satisfies properties \eqref{item:first}-\eqref{item:last} from Theorem~\ref{mainthm2}).
\end{corollary}

Following an idea of Gromov (see \cite{G87,P91}), one lets $K$ be the cone over a smooth triangulation $\tau$ of $M$.
Then we apply the strict hyperbolization $\hyperbolize K$, and note that since hyperbolization preserves links, the point $p\in \hyperbolize K$ corresponding to the cone point will have link a copy of $\tau$. 
Thus, removing a small neighborhood of $p$ leaves us with a cobordism $W$ between $M$ and $M':=\mathcal H(\tau)$. 
Our Theorem~\ref{mainthm2} then applies to $M'$. 
Note that $\pi_1(W)$ itself contains $\pi_1(M)$, hence might not be linear (for instance, if $\pi_1(M)$ is a
non-linear group). 
On the other hand, if $\pi_1(M)$ is residually finite, then $\pi_1(M')$ is also residually finite by our Theorem~\ref{mainthm2}, and we showed in \cite[Corollary 5.9]{LRGM23} that in this case there is a cobordism between $M$ and $M'$ that has residually finite fundamental group. 

\begin{remark}
Thom's work showed that oriented cobordism classes are rationally represented by products of even dimensional complex projective spaces (see \cite[Section 17]{MS74}).
So every smooth oriented closed
manifold has a multiple which is cobordant to a non-negatively curved Riemannian manifold. 
In analogy, combining Davis--Januszkiewicz--Weinberger \cite{DJW01}, Charney--Davis \cite{CD95}, 
and Ontaneda \cite{O20}, one obtains that every smooth oriented closed manifold is cobordant to a strictly negatively curved Riemannian manifold. 

In dimensions $\geq 5$, the Borel Conjecture is known to hold for aspherical manifolds with Gromov hyperbolic groups (see Bartels--L\"uck \cite{BL12}). As such, 
the topological manifold $M':=\mathcal H(\tau)$ is completely determined, up to homeomorphism, by its fundamental group. So the discussion above in principle 
reduces the study of cobordism classes of manifolds of dimension $n\geq 5$, to the study of the corresponding $\pi_1(M')$. Our corollary further
reduces it to the linear case.
\end{remark}

\begin{remark}
More generally, Corollary~\ref{cor:cobordism} works for a PL manifold, or even for a triangulable topological manifold.
Note that in all dimension $n\geq 4$ there exist closed topological manifolds that are not triangulable (see \cite{FR82,MA16}).
Moreover, in all dimension $n\geq 6$ such manifolds can be chosen to be aspherical by \cite{DFL14}, and have virtually compact special (hence residually finite) fundamental group by \cite[Theorem 5.12]{LRGM23}.
\end{remark}

\subsubsection{Prescribing the Gromov boundary}
 The groups obtained by strict hyperbolization are Gromov hyperbolic groups, so it is natural to ask what their Gromov boundary looks like.
 For example, the groups obtained by Riemannian hyperbolization in \cite{O20} (see \S\ref{subsec:ontaneda}) are fundamental groups of smooth Riemannian manifolds of negative curvature, hence their Gromov boundaries are spheres of the appropriate dimensions.
 
\begin{corollary}
Let $n\geq 1$, and let $M$ be a closed connected orientable PL $n$-manifold that bounds a compact orientable PL $(n+1)$-manifold.
Then there exists a Gromov hyperbolic group $G$ such that
\begin{itemize}
    \item the Gromov boundary of $G$ is homeomorphic to the tree of manifolds $\mathcal X (M)$;
    \item $G$ is virtually compact special (hence satisfies \eqref{item:first}-\eqref{item:last} in Theorem \ref{mainthm2}).
\end{itemize}
\end{corollary}

The groups in this statement are the ones obtained by \'Swi\k{a}tkowski in \cite{SW20} via strict hyperbolization of certain pseudomanifolds in which the link of a point is either a sphere or a copy of the manifold $M$.
The tree of manifolds $\mathcal X (M)$ is a  compact metrizable space which is obtained, roughly speaking, as a certain limit of connected sums of copies of $M$.


\subsubsection{Manifolds with exotic symmetries}
The hyperbolization procedures satisfy a certain functorial property: automorphisms of the simplicial complex $K$ induce isometries of the hyperbolized complex 
$\mathcal H(K)$. This has been used by various authors to produce closed manifolds with interesting symmetries. 

For example, if $G$ is a Gromov hyperbolic group which is a Poincar\'e Duality group over $\mathbb Z$, an easy application of Smith theory shows that the fixed subgroup $G^\sigma$ of an involution $\sigma\in Aut(G)$ is still a Poincar\'e Duality group, but over $\mathbb Z_2$. 
Farrell--Lafont in \cite{FL04} used an exotic symmetry produced via strict hyperbolization, to give examples whose fixed subgroups are {\bf not} Poincar\'e Duality over $\mathbb Z$. Our results now show that these examples can also be chosen to satisfy properties \eqref{item:first}-\eqref{item:last} in Theorem \ref{mainthm2}.

For another application, recall that in their seminal paper  \cite{BC00}, Baum--Connes defined a trace map $tr: K_0(C_r^*G)\rightarrow \mathbb R$, where $C_r^*G$ is the reduced $C^*$-algebra of the discrete
group $G$. They also formulated the {\it trace conjecture}, which predicted that when $G$ is a group with torsion, the image of the trace map is contained in the additive
{\it subgroup} of $\mathbb Q$ generated by $1/n$, where $n$ ranges over the order of finite subgroups of $G$. 
A counterexample to this conjecture was constructed by Roy \cite{RO99}, 
using the Davis--Januszkiewicz (non-strict) hyperbolization procedure. She constructed a group $G$ whose only finite subgroups are isomorphic to $\mathbb Z_3$, 
and an element in $K_0(C_r^*G)$ whose trace equals $-1105/9$. Nevertheless, there is always the possibility
that the original Baum--Connes trace conjecture might hold for certain restricted classes of groups. 
The computations carried out by Roy (\cite{RO99}, pgs. 210-213) apply verbatim 
if one instead uses the Charney--Davis strict hyperbolization, so our results have the following consequence.

\begin{corollary}\label{cor:trace_conjecture}
There exists a Gromov hyperbolic group $G$ whose only finite subgroups are isomorphic to $\zz_3$, but where the
image of the trace map contains $-1105/9$. Moreover, this group satisfies properties \eqref{item:first}-\eqref{item:last} in Theorem \ref{mainthm2}. In particular, the original Baum--Connes trace conjecture does not hold for the classes of groups \eqref{item:first}-\eqref{item:last} in Theorem \ref{mainthm2}.
\end{corollary}

\begin{remark}
L\"uck formulated a refinement of the original Baum--Connes trace conjecture: the image of the trace
map is contained in the {\it subring} $\zz[1/|Fin(G)|]$, obtained from $\zz$ by inverting all the orders of finite subgroups of $G$. L\"uck 
showed that this refined Trace Conjecture holds for any group that satisfies the Baum--Connes Conjecture (see \cite{LU02}). In the subsequent
literature, this refined version is what is commonly referred to as the Trace Conjecture. For Gromov hyperbolic groups, the Baum--Connes Conjecture was established by Lafforgue (see \cite{LA02}). Thus the group appearing in our Corollary~\ref{cor:trace_conjecture} satisfies the refined trace conjecture.
\end{remark}

\subsection{Virtual algebraic fibering}
In this section we present new applications of a more algebraic flavor.
We say a group $G$ \textit{algebraically fibers} if it admits a surjective homomorphism to $\zz$ with finitely generated kernel. 
We say it \textit{virtually algebraically fibers} if it has a finite index subgroup that algebraically fibers.
Agol introduced the notion of   \textit{residually finite rationally solvable} (or RFRS) group  in \cite{AG08} as a major ingredient in the solution of the Virtual Haken Conjecture and Virtual Fibering Conjecture.
Kielak proved in \cite[Thoerem 5.3]{K20} that a finitely generated virtually RFRS group virtually algebraically fibers if and only if its first $L^2$--Betti number vanishes.
Fisher has extended this result in \cite[Theorem A]{FI21} to relate the vanishing of higher $L^2$--Betti numbers of $G$ to higher finiteness properties of the kernel of a virtual algebraic fibration. 

All of the groups constructed in this paper via strict hyperbolization are virtually compact special, hence  virtually RFRS, see \cite[Corollary 2.3]{AG08}.
In some cases, it is possible to prove vanishing of many $L^2$--Betti numbers (for instance for all the examples obtained by Ontaneda in \cite{O20}, provided the curvatures are sufficiently pinched; see below for details). 
Hence, we get several new examples of virtually compact special Gromov hyperbolic groups that admit a virtual algebraic fibration, whose kernel has good algebraic finiteness properties.
On the other hand, these groups can often be seen to be incoherent, and in some cases it is possible to see that the kernel of a virtual algebraic fibration is itself a witness to incoherence (i.e.\ is finitely generated but not finitely presented).

Before providing the details for our case, we note that similar arguments also work for arithmetic hyperbolic manifolds of simple type and for Gromov--Thurston manifolds. These are known to be virtually specially cubulated (hence RFRS) by \cite{HW12} and \cite{GI17} respectively.

\subsubsection{Kernels with good algebraic finiteness properties}

We start by constructing Gromov hyperbolic groups that virtually algebraically fiber, and are not isomorphic to groups that were previously known to have this property.

\begin{corollary}
For all $n\geq 4$ there is a closed Riemannian $n$-manifold $M$ with negative sectional curvatures and such that
\begin{itemize}
    
    \item $\pi_1(M)$ virtually algebraically fibers;
    
    \item  $\pi_1(M)$ is Gromov hyperbolic and  virtually compact special (hence satisfies \eqref{item:first}-\eqref{item:last} in Theorem \ref{mainthm2});
    
    \item $\pi_1(M)$ is  not isomorphic to a uniform lattice in $\sofull$ (or other real simple Lie group of rank $1$), or to the fundamental group of a Gromov-Thurston, Mostow-Siu, Deraux, or Stover--Toledo manifold.
\end{itemize}
\end{corollary}

The manifolds in this statement are the ones constructed by Ontaneda in \cite{O20} (see Corollary~\ref{cor:ontaneda_special_manifolds} above).
As a result of our Theorem~\ref{mainthm2} the fundamental group of such a manifold $M$ is virtually compact special, and in particular it is virtually RFRS.
Moreover, $M$ can be chosen to have sectional curvatures pinched in the interval $[-1-\varepsilon,-1]$ for an arbitrarily small $\varepsilon>0$.
By a result of Donnelly and Xavier (see \cite[\S 4]{DX84}, and also \cite[Theorem 2.3]{JX00}), if the curvatures of $M$ are sufficiently pinched (i.e.\ $\varepsilon$ is sufficiently small with respect to the dimension $n$), then $M$ does not have any non-zero $L^2$--harmonic $p$-forms, for $p$ in a certain range. 
In particular,  $b^{(2)}_1(\pi_1(M))=0$.
By \cite{K20} we see that if $\varepsilon$ is small enough then $\pi_1(M)$ virtually algebraically fibers. 

Furthermore, one can pinch the curvatures even more to force the vanishing of the $L^2$--Betti numbers for $p=0,1,\dots,\lfloor{\frac n2} \rfloor -1$. 
In particular, using results from \cite{FI21} one can obtain examples in which $\pi_1(M)$ virtually algebraically fibers with kernel of type $FP_{\lfloor{\frac n2} \rfloor -1}(\qq)$.
Also note that in the even dimensional case $M$ will then satisfy the (weak) Hopf conjecture, i.e.\ $(-1)^{\frac n2}\chi(M) \geq 0 $.



    
    

\subsubsection{Kernels witness incoherence in dimension 4}
We have discussed how to use strict hyperbolization to obtain Gromov hyperbolic groups that virtually algebraically fiber with kernel of type $FP_{\lfloor{\frac n2} \rfloor}(\qq)$. 
On the other hand, these kernels should not be expected to have better finiteness properties.
Indeed, in the context of the previous paragraph, we can show that in dimension $n=4$ these kernels are not finitely presented (i.e.\ not of type $F_2$).

To see this, notice that Chern--Weil theory implies that the Euler characteristic of a closed negatively curved $4$-manifold is strictly positive (see \cite{C55}). 
This prevents the kernel of an algebraic fibration of $\pi_1(M)$ from being finitely presented, as we now describe.
We thank Genevieve Walsh for sharing the following argument with us. (This appears in \cite{KVW21}.)

\begin{lemma} \label{Genevieve}
Let $M$ be a closed aspherical $4$-manifold such that $\chi(M) \neq 0$.  If $\pi_1(M)$ virtually algebraically fibers, then $\pi_1(M)$ is incoherent (the kernel is not finitely presented). 
\end{lemma} 
\begin{proof} 
Suppose $\pi_1(M)$ virtually algebraically fibers, and let $G$ be the finite index subgroup of $\pi_1(M)$ which surjects to $\zz$ with finitely generated kernel $K$. 
Notice that $G$ is a $PD_4$ group with  $\chi(G) \neq 0$ (since Euler characteristic is multiplicative by index), and that $\zz$ is a $PD_1$ group. 
Assume by contradiction that $K$ is finitely presented (i.e.\ type $F_2$).
Then $K$ is in particular of type $FP_2$, and it follows from \cite[Corollary 1.1]{HK07} that $K$ is a $PD_3$ group.
In particular $K$ has finite homological type (and the same is true for $\zz$).
So, by the properties of Euler characteristics on short exact sequences (see \cite[Chapter IX, 7.3(d)]{BR82}) we can conclude that $\chi(G)=\chi(K)\chi(\zz)=0$. This contradicts the fact that $\chi(G)\neq 0$.


\end{proof} 

An alternative argument for this Lemma, under the additional assumption that $\pi_1(M)$ is virtually RFRS, was shared with us by Kevin Schreve. 
\begin{proof}
In the same set up, if by contradiction $K$ is finitely presented, then it is in particular of type $\operatorname{FP}_2(\qq)$. 
So, since  $G$ is also virtually RFRS, by \cite{K20,FI21} we get that $b^{(2)}_1(G)=b^{(2)}_2(G)=0$.
But $G$ is a $PD_4$ group, so by duality this implies that all $L^2$-Betti numbers vanish. This gives $\chi(G)=0$, which is again absurd.
\end{proof}

As a result we obtain the following statement. The manifold in it  is once again one of the manifolds obtained by Ontaneda, with curvatures sufficiently pinched.
\begin{corollary}
There exists a closed $4$-dimensional Riemannian manifold $M$ with negative sectional curvatures and such that
\begin{itemize}
    \item  $\pi_1(M)$ is incoherent (it virtually algebraically fibers with non finitely presented kernel);
    
    \item $\pi_1(M)$ is Gromov hyperbolic and   virtually compact special (hence satisfies \eqref{item:first}-\eqref{item:last} in Theorem \ref{mainthm2});
    
    \item $\pi_1(M)$ is  not isomorphic to a uniform lattice in $\sofull$  (or other real simple Lie group of rank $1$), or to the fundamental group of a Gromov-Thurston, Mostow-Siu, or Stover--Toledo  manifold.
\end{itemize}
\end{corollary}

\begin{remark}
The situation in dimension 4 is quite different from that in dimension 5.
Indeed, Italiano, Martelli, Migliorini in \cite{IMM21} obtained a 5-dimensional   cusped hyperbolic manifold that fibers over the circle. 
Its fundamental group algebraically fibers, with  kernel of finite type (in particular finitely presented).
The hyperbolic groups obtained by suitable Dehn filling on these examples were shown to fiber with kernel of finite type.
Moreover,  recent work of Groves and Manning shows that some of these groups are virtually compact special (see \cite{GM22}).
\end{remark}

\begin{remark}
When $n\geq 5$, the groups obtained by strict hyperbolization done with a sufficiently large piece (as in Ontaneda) contain subgroups isomorphic to uniform arithmetic lattices in $\operatorname{SO}(4,1)$.
The incoherence of these subgroups (see \cite{KPV08,AG08,K13}) gives incoherence of the hyperbolized groups, but these subgroups are not fibers themselves (for instance because they are quasiconvex).
Notice that this approach does not work  in dimension $n=4$, as uniform lattices in $\operatorname{SO}(3,1)$ are coherent.
\end{remark}

\bigskip

\addtocontents{toc}{\protect\setcounter{tocdepth}{1}}

\subsection*{Structure of the paper}
This paper is structured as follows. 
In \S\ref{sec:intro} we presented the motivation, the context, the statements, and the major applications of our results.
In \S\ref{sec:cell complexes and hyperbolization procedures}, we provide  combinatorial and metric background about cell complexes and  hyperbolization procedures. 
\S\ref{sec:strict hyperbolization} is devoted to a description of Charney--Davis strict hyperbolization procedure for a cubical complex $X$. 
In particular, we study the geometry of the universal cover $\uchc$ of the hyperbolized complex $\hc$ in terms of a certain collection of convex subspaces called mirrors. This provides a graph of spaces decomposition of $\hc$.
In \S\ref{sec:dual cubical complex} we construct and study a dual $\cat 0$ cubical complex $\dccx$. 
Finally in \S\ref{sec:special cubulation} we study the action of the hyperbolized group $\hg=\pi_1(\hc)$ on this dual cubical complex $\dccx$, and prove that $\hg$ is virtually compact special.

\subsection*{Common terminology and notation}
The numbers in parentheses refer to the section(s) in which each item is introduced or discussed.

\begin{itemize}
    \item The hyperbolizing lattice  $\Gamma$ (\S\ref{sec:CD hyperbolizing cell}) and the cubical complex $X$.
    \item The hyperbolized complex $\hc$ (\S\ref{sec:CD hyperbolized complex}), and its universal cover $\uchc$ (\S\ref{sec:universal_cover}).
    \item The hyperbolized cube $\hq$ (\S\ref{sec:CD hyperbolizing cell}), and its universal cover $\uchq$ (\S\ref{sec:universal_cover}).
    \item The hyperbolized groups   $\hgq = \pi_1(\hq)$ (\S\ref{sec:CD hyperbolizing cell}) and  $\hg = \pi_1(\hc)$ (\S\ref{sec:CD hyperbolized complex}).
    \item The folding map $f:X\to \square^n$ of a foldable complex (\S\ref{sec:foldablecomplexes}), and the induced map $\foldc:\hc\to \hq$ on the hyperbolized complex (\S\ref{sec:CD hyperbolized complex},\S\ref{sec:universal_cover}).
    \item The Charney-Davis map $\cd:\hq\to \square^n$, and the induced map $\cdX:\hc\to X$ on the hyperbolized complex (\S\ref{sec:CD hyperbolized complex}).
    \item The dual cubical complex $\dccx$ (\S\ref{sec:dual cubical complex}).
     \item A face  $F\subseteq \square^n$.
    \item A cube  $C\subseteq X$ (if $X$ is a cubical complex) or $Q \subseteq \dccx$.
    \item A cell  $\sigma \subseteq \hc$, $\uchc$ (\S\ref{subsec:stratification}), $\hh^n$, $\hq$, $\uchq$ (\S\ref{sec:universal_cover}). 
    \item A tile $\tau$ in $\uchc$ (\S\ref{sec:universal_cover}), and the dual tile $\dcc \tau$ in $\dccx$ (\S\ref{subsec:efficiency}).
    \item A mirror $M$ in $\uchc$ (\S\ref{subsec:mirrors convexity}), and the dual mirror $\dcc M$ in $\dccx$ (\S\ref{subsec:mirror complexity}).
    \item An edge--path $p$ in $\dccx$, its length $\length p$, its height $\height p$ (\S\ref{subsec:cubes and links}, \S\ref{subsec:efficiency}), the number of $(p,M)$--crossings $\mirrorcomplexitymirror{p}{M}$ with respect to a mirror $M$, and its total mirror complexity $\mirrorcomplexity{p}$ (\S\ref{subsec:mirror complexity}).
\end{itemize}

\subsection*{Acknowledgements}
We thank Igor Belegradek, Dick Canary, Mike Davis, Jingyin Huang, Ben McReynolds, Nick Miller for useful conversations, and in particular Genevieve Walsh and Kevin Schreve for pointing out the arguments for Lemma~\ref{Genevieve}. 
J.-F. Lafont was partly supported by the NSF Grant number DMS-1812028 and DMS-2109683.
L. Ruffoni acknowledges support by the AMS, the Simons Foundation, the INDAM group GNSAGA, and thanks the Department of Mathematics at The Ohio State University for the hospitality during some visits in which this work was conducted.
We thank the referee for their very careful reading and their suggestions.


\section{Cell complexes and hyperbolization procedures}\label{sec:cell complexes and hyperbolization procedures}
	\addtocontents{toc}{\protect\setcounter{tocdepth}{2}}

We collect in this section some background material used in our constructions. In \S \ref{sec:cell complexes} we review the basics about cell complexes, and in \S \ref{sec:foldablecomplexes} we focus on foldable complexes, i.e.\ complexes that can be folded down to a standard simplex or cube. In \S \ref{sec:hyperbolization template} we introduce a general template for the study of hyperbolization procedures for foldable complexes. In \S \ref{sec:Gromov construction} we review a specific hyperbolization procedure due to Gromov.

\subsection{Combinatorial and metric geometry of cell complexes}\label{sec:cell complexes}

In this section we collect background material about cell complexes, mainly to fix notation and terminology; for a detailed treatment see \cite[\S I.7, \S II.5]{BH99}. Let us denote by $\mathbb M^n_k$ the simply connected Riemannian manifold of dimension $n$ and  constant sectional curvature $k$: for instance $\mathbb M^n_1=\mathbb S^n$ is the round sphere, $\mathbb M^n_0=\mathbb E^n$ is the Euclidean space, and $\mathbb M^n_{-1}=\hh^n$ is the hyperbolic space. An isometrically embedded copy of $\mathbb M^d_k$ inside $\mathbb M^n_k$ will be called a $d$--\textit{plane}, or a \textit{hyperplane} if $d=n-1$.

\subsubsection{Cells}
A \textit{cell} in $\mathbb M^n_k$ is defined to be the convex hull of a finite set of points; if $k>0$ we are going to also require that it is contained in an open ball of radius $\frac{\pi}{2\sqrt k}$. 
The \textit{dimension} of a cell $C$ is the smallest $d$ such that $C$ is contained in a $d$--plane. A cell of dimension $d$ will also be called a $d$--\textit{cell}.
The \textit{interior} of $C$ is its interior inside this $d$--plane.
A \textit{face} $F$ of $C$ is a subspace of the form $F=H\cap C$ where $H$ is a hyperplane such that $C$ lives in one of the two closed half--spaces bounded by $H$, and $H\cap C\neq \varnothing$. A face is itself a cell, and we call \textit{vertices} and \textit{edges} of $C$ the faces of dimension $0$ and $1$ respectively.

\subsubsection{Cell complexes}\label{subsec:cell complexes}
An $\mathbb M^n_k$-\textit{cell complex} is a topological space $X$ obtained by gluing together cells from $\mathbb M^n_k$ by isometries of their faces, in such a way that each cell embeds in $X$ and the intersection of any two cells is either empty or a cell. 
Notice that this definition is slightly more restrictive than the one in \cite[Definition I.7.37]{BH99} (which allows one to glue two faces of the same cell), and the one  in \cite{CD95} (in which cells are allowed to intersect in a proper union of faces). Both conditions can be satisfied by performing a cellular subdivision. On the other hand, we do not require cell complexes to be locally compact at this stage, i.e.\ a vertex can be contained in infinitely many cells. 

We call an $\mathbb M^n_k$-cell complex simply  a \textit{cell complex} when we do not need to keep track of $\mathbb M^n_k$.
For instance we will denote by $\triangle^n$ the standard $n$-simplex and by $\square^n=[0,1]^n$ the standard $n$-cube; these are cells in $\mathbb M^n_0$. A \textit{simplicial complex} is a  cell complex obtained by gluing simplices, and a \textit{cubical complex} is a cell complex obtained by gluing cubes.

The \textit{dimension} of a cell complex is the maximum dimension of its cells. We say that an $n$-dimensional cell complex is \textit{homogeneous} if every cell is contained in a cell of dimension $n$, and that it is \textit{without boundary} if every $(n-1)$-cell is contained in at least two different $n$-cells. 
For all $k=0,\dots,n$, the $k$-\textit{skeleton} of $X$ is the subspace consisting of all the cells of dimension at most $k$, and will be denoted by $X^{(k)}$.
A \textit{subcomplex} of $X$ is a closed subspace $Y\subseteq X$ which is a union of cells of $X$.
If $X$ and $Y$ are cell complexes, a continuous function $f:X\to Y$ is a \textit{combinatorial map} if for every cell $C$ of $X$ we have that $f$ is a homeomorphism from $C$ to a cell $f(C)$ of $Y$.

Given a cell complex $X$, its \textit{barycentric subdivision} $\barsub X$ is the simplicial complex whose $k$-simplices correspond to strictly ascending sequences of faces $F_0\subset \dots\subset F_k$ of $X$. 
There exists a natural (non--combinatorial) homeomorphism $X\to \barsub X$.
We refer to \cite[\S I.7.44-48]{BH99} for more details. 
Similarly, if $X$ is a cubical complex, then its \textit{cubical subdivision} is the cubical complex obtained by subdividing each $n$--cube along midcubes into $2^n$ cubes. 

\begin{remark}\label{rem:not a cell complex}
By definition, a cell is compact, it has finitely many faces, and it can be realized as the intersection of finitely many closed half-spaces (see \cite[\S I.7]{BH99}).
We want to warn the reader that one of the main object under investigation in this paper (see \S\ref{subsec:stratification}) is obtained by gluing together certain ``generalized cells'', i.e.\  subsets of $\hh^n$ which are given by the intersection of an infinite but locally finite  collection of closed half-spaces. These subsets are convex but not compact, so the resulting space is not strictly speaking a cell complex. However, some of the usual tools for the study of cell complexes can be applied in this context (e.g.\ links). We will highlight the subtleties in the construction whenever relevant.
\end{remark}

\subsubsection{Links}\label{subsec:def links}
Let $X$ be a cell complex. We define the link of points and cells as follows (see \cite[\S I.7]{BH99} for more details). 
Let $p$ be a point of an $n$--cell $C\subseteq \mathbb M_k^n$. Then we define the \textit{link} $\lk{p,C}$ to be the space of unit vectors in the tangent space at $p$ inside $C$. Measuring the angle between vectors endows $\lk{p,C}$ with a natural length metric, which makes it isometric to  an $(n-1)$--cell in $\mathbb S^{n-1}$.
The \textit{link} $\lk{p,X}$ of $p$ in $X$ is then defined by gluing together the links $\lk{p,C_i}$, where $C_i$ ranges over the cells of $X$ containing $p$. This is naturally an $\mathbb M_1^{n-1}$--cell complex.
If $Y$ is a sufficiently regular subspace of $X$ containing $p$ (e.g.\ a subcomplex), then the \textit{link} $\lk{p,Y}$ is defined analogously, restricting to vectors along $Y$.

Let $F$ be a $d$--face of an $n$--cell $C\subseteq \mathbb M_k^n$. Then we define the \textit{link} $\lk{F,C}$ to be the subspace of unit vectors in the tangent space  at an interior point of $F$, which are pointing into $C$ and are orthogonal to $F$. As before, this is naturally an $(n-d-1)$--cell in $\mathbb S^{n-1}$. 
The \textit{link} $\lk{C,X}$ of a $d$--cell $C\subseteq X$ is then defined by gluing together the links $\lk{C,C_i}$, where $C_i$ ranges over the cells of $X$ containing $C$. 
It is naturally an $\mathbb M_1^{n-d-1}$--cell complex. 
Finally, if $Y\subseteq X$ is a subcomplex of $X$ containing $C$, the \textit{link} $\lk{C,Y}$ of $C$ in $Y$ is defined analogously, by restricting to the cells of $Y$ that contain $C$.
Observe that if $X$ is a simplicial or cubical complex, then the link of a $d$--cell $C$ is a simplicial complex in which vertices are given by the $(d+1)$--cells  containing $C$, and in which $m+1$ vertices span an $m$--simplex if and only if the corresponding cells are contained in a $(d+m+1)$--cell.

\subsubsection{Spaces and complexes of bounded curvature}\label{subsec:non-positive curvature}
We will consider the usual notions of curvature for metric spaces, such as being locally $\cat k$ or Gromov hyperbolic (see \cite[\S II.1, \S III.H.1]{BH99} for more details). 
In particular, we will say a space is \textit{non-positively curved} if it is locally $\cat 0$, and \textit{negatively curved} if it is locally $\cat k$ for some $k<0$. 
Note that if $k<0$ then a $\cat{k}$ space is in particular Gromov hyperbolic (see \cite[Proposition III.H.1.2]{BH99}), and that if $k\leq 0$ then a $\cat k$ space is uniquely geodesic (see \cite[Proposition II.1.4]{BH99}). Whenever $x,y$ are points in a uniquely geodesic space, we denote by $[x,y]$ the unique geodesic between them.

Let $X$ be an $\mathbb M^n_k$-cell complex. Each cell can be naturally endowed with a metric from $\mathbb M^n_k$, and these can be glued together to make $X$ into a complete geodesic metric space, as soon as there are only finitely many isometry classes of cells in $X$ (see \cite[Theorem I.7.50]{BH99}). When equipped with this metric, $X$ is said to be a cell complex of piecewise constant curvature $k$; we say it is \textit{piecewise} \textit{spherical}, \textit{Euclidean}, or \textit{hyperbolic} when $k=1,0,-1$ respectively. If not otherwise specified, a simplicial or cubical complex is always endowed with its standard piecewise Euclidean metric.

It is natural to ask for conditions under which a complex of piecewise constant curvature is a space of bounded curvature, namely a locally $\cat k$ space. 
For cubical complexes this is completely controlled by the links of vertices.
In a cubical complex,  cells are isometric to the standard Euclidean cube $\square^n=[0,1]^n$, so the link of a vertex is a piecewise spherical simplicial complex, in which all edges have length $\frac{\pi}{2}$.
The following is known as Gromov's link condition (see \cite[Theorems II.5.18, II.5.20]{BH99}).
A simplicial complex is \textit{flag} if any $k+1$ pairwise adjacent vertices span a $k$--simplex.

\begin{lemma}\label{lem:gromov link condition}
Let $X$ be a cubical complex. Then the following are equivalent.
\begin{enumerate}
    \item $X$ is non--positively curved (i.e.\ locally $\cat 0$).
    \item The link of each vertex is a flag simplicial complex.
    \item The link of each vertex is a $\cat 1$ simplicial complex.
\end{enumerate}
\end{lemma}



\subsection{Foldable complexes}\label{sec:foldablecomplexes}
Here we consider the notion of foldability for simplicial and cubical complexes that we will require later. 
The first definition is essentially from \cite[\S1]{BS99} (but see also \cite[Definition 7.2]{CD95}, and \cite{XIE04} for a more recent discussion).

A simplicial (respectively cubical) $n$-dimensional complex $X$ is \textit{foldable} if it admits a combinatorial map $f:X\to \triangle^n$ (respectively $f:X\to \square^n$) such that its restriction to each cell of $X$ is injective. Such a map will be called a \textit{folding} for $X$.
Notice that in a foldable complex the cells are necessarily embedded. This is the main reason why we have incorporated this condition in the definition of cell complex in \S \ref{sec:cell complexes}.

Foldability has some immediate consequences. If $X$ is foldable, and $p:Y\to X$ is a combinatorial map which is injective on each cell, then $Y$ is foldable too. 
In particular any covering of a foldable complex is foldable. Moreover if $X$ is foldable, then the links of cells of codimension $2$ are bipartite graphs. We collect below some examples in the cubical case; analogous ones can be constructed for the simplicial case.

\begin{figure}[h]
    \centering
    \includegraphics[width=.75\textwidth]{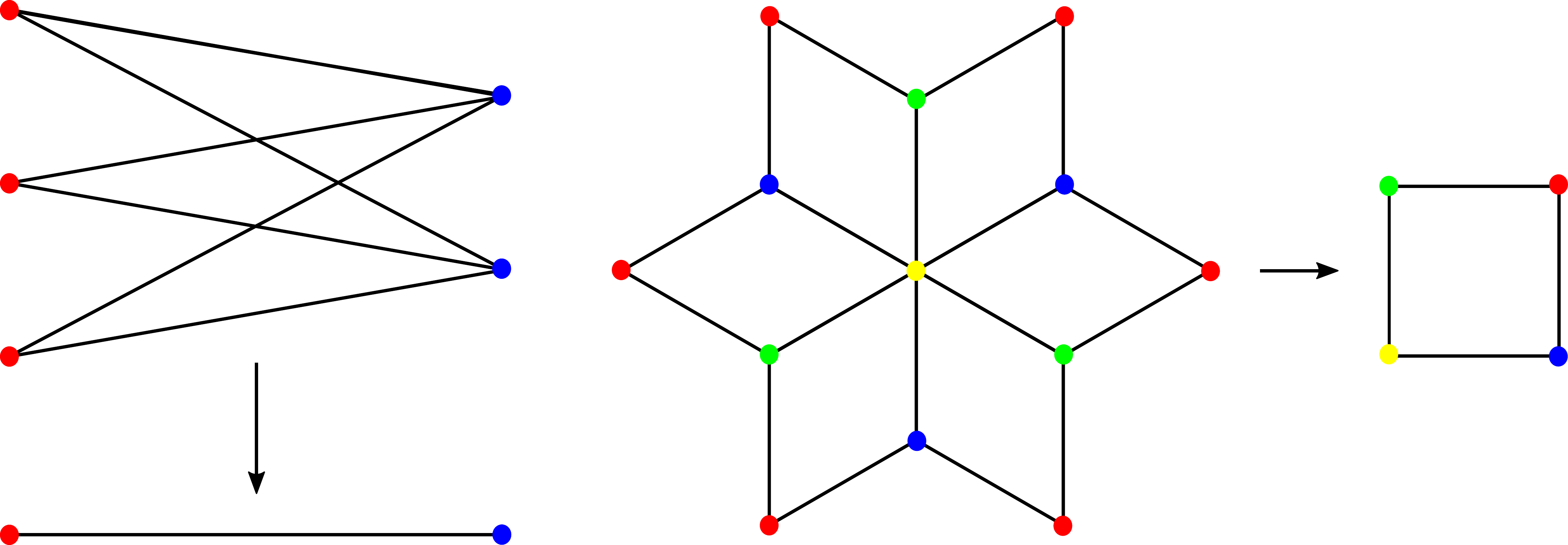}
    \caption{Foldable cubical complexes}
    \label{fig:foldable cubical complexes}
\end{figure}

\begin{example}[Foldable and not foldable cubical complexes]\label{ex:foldable_complexes}
\

\begin{enumerate}
    \item A graph is foldable if and only if it is bipartite (Figure~\ref{fig:foldable cubical complexes}, left).
    
    \item The rose $R_m$ consisting of $m$   squares  with a  vertex in common is foldable if and only if $m$ is even (Figure~\ref{fig:foldable cubical complexes}, right).
    
    \item Foldability of $X$ implies that links of codimension 2 cells are bipartite. 
    However, foldability is not completely determined by this property. For example,  let $X$ be the cubical complex obtained by taking the product $\partial \Delta^2 \times \rr$, where $\rr$ is endowed with the standard cell structure induced by $\zz$. Then the links of vertices are cycles of length 4 (hence they are bipartite), but $X$ is not foldable; notice that the universal cover of $X$ identifies the square complex defined by $\zz^2$ in $\rr^2$, which is foldable (compare \cite[Lemma 1.2]{BS99}).

\end{enumerate}

 \end{example}

A main source of foldability comes from barycentric subdivisions; the following is well-known (see \cite[Lemma 2.1]{BS99}), we include a proof for completeness (see left of Figure~\ref{fig:foldable subdivision} for an example).

\begin{lemma}\label{lem:barsub_folds}
If $X$ is a cell complex, then $\barsub X$ is a foldable simplicial complex.  
\end{lemma}
\proof
Let $X$ have dimension $n$, and  consider the simplex spanned by $\{0,\dots,n\}$; this is just the standard simplex $\triangle^n$. Then we can define a map $f:\barsub X\to \triangle^n$ by sending a vertex of $\barsub X$ to the number which is equal to the dimension of the corresponding cell in $X$.
\endproof

\begin{figure}[h]
    \centering
    \includegraphics[width=.75\textwidth]{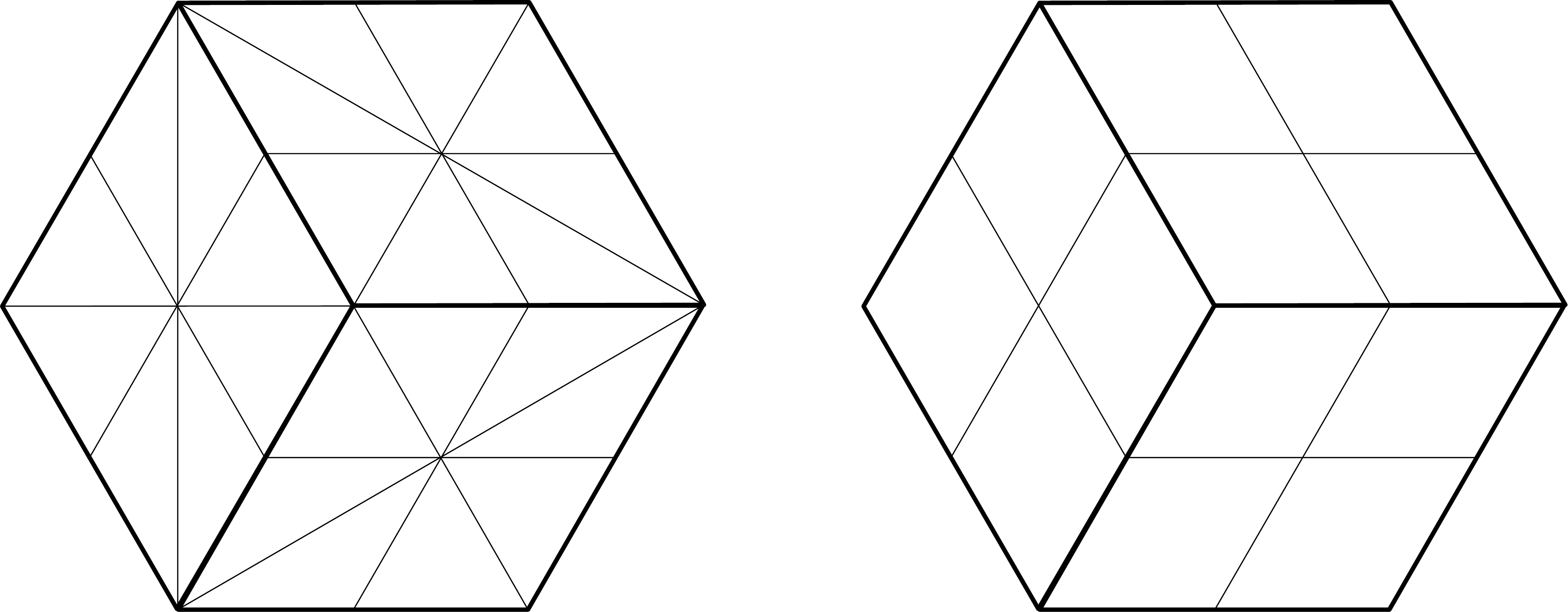}
    \caption{The barycentric subdivision of the rose of $3$ squares is a foldable simplicial complex, but its cubical subdivision is not a foldable cubical complex.}
    \label{fig:foldable subdivision}
\end{figure}

On the other hand, if $X$ is a non--foldable cubical complex of dimension at least $2$, then its cubical subdivision is still non--foldable (see Figure~\ref{fig:foldable subdivision}, right).  In \S \ref{sec:Gromov construction} we will review Gromov's construction and show that it can be used to turn any cubical complex into a foldable one (mildly changing the topology).


\subsection{Hyperbolization procedures}\label{sec:hyperbolization template}
In this section we set a framework for the study of certain constructions, which take a cell complex as input and return a non-positively curved space as output. 
The resulting space is in particular always aspherical, so the topology of the original complex is altered.
What is interesting is that this can happen in a controlled way that allows to preserve some features of the original complex. 
Constructions of this type are generally known as hyperbolization procedures (or asphericalization procedures). They were first introduced by Gromov (see \cite[\S 3.4.A]{G87}), and then popularized by several authors (see \cite{DJ91,P91,CD95,DFL14,O20}).

All the hyperbolization procedures we will consider in this paper are obtained by different incarnations of the same abstract construction, which we now review briefly, referring the reader to \cite{WI63} or \cite[\S 1]{DJ91} for more details. The naive idea is to fix some topological space $S$ and then replace every top-dimensional cell of a complex $X$ by a copy of $S$. For this gluing to be well--defined, it is common to assume that both $X$ and $S$ come equipped with chosen maps to a reference space.

For concreteness let us consider the following set up. Let us denote by $\sigma^n$ the standard simplex $\triangle^n$ or the standard cube $\square^n$, and let us fix a topological space $S$, equipped with a continuous map $g:S\to \sigma^n$, and a foldable simplicial or cubical complex $X$, equipped with a  fixed folding $f:X\to \sigma^n$. 
One then considers the fibered product $\mathcal H_S(X)=\{(x,s)\in X\times S\ | \ f(x)=g(s) \}$, i.e.\ the space obtained via the pullback square in Figure~\ref{fig:template}.

\begin{figure}
    \centering
    \begin{tikzpicture}
\node at (-1,1) (A) {$\mathcal H_S(X)$};
\node at (-1,-1) (B) {$X$};
\node at (1,-1) (C) {$\sigma^n$};
\node at (1,1) (D) {$S$};
\draw [dashed,->] (A) edge (B)   (A) edge (D) ;
\draw [->]  (B) edge (C) (D) edge (C);
\node at (0,-1.25) {$f$};
\node at (1.25,0) {$g$};
\node at (0,1.25) {$f_S$};
\node at (-1.4,0) {$g_X$};
\end{tikzpicture}
    \caption{A template for hyperbolization procedures}
    \label{fig:template}
\end{figure}

Note that the construction endows $\mathcal{H}_S(X)$ with natural continuous maps $g_X:\mathcal{H}_S(X)\to X$ and $f_S:\mathcal{H}_S(X)\to S$, which are just the restrictions of the projections onto the factors of $X\times S$, and which make the diagram commute.
Properties of the pair $(S,g)$ will result in properties of the space $\mathcal H_S(X)$, and the art of hyperbolization consists in crafting a pair $(S,g)$ which yields some interesting properties on $\mathcal H_S(X)$. For a trivial example, consider the case $S$ consists of a single point. Then $\mathcal H_S(X)$ is just the discrete set $f^{-1}(g(S))$. 

The following lemma identifies a mild condition under which the space $\mathcal H_S(X)$ looks like a collection of copies of $S$  (compare the remark on page 321 of \cite{WI63}).
We explicitly remark that we do not assume $S$ to be compact until part \eqref{item:hypcell compact} of this lemma. This will be relevant in \S\ref{sec:universal_cover} for the study of a certain combinatorial decomposition of a space into non--compact pieces.

\begin{lemma}\label{lem:cell_is_hyperbolizing_cell}
Let $g:S\to \sigma^n$ be surjective, and let $C\subseteq X$ be an $n$--cell. Then the following hold.
\begin{enumerate}
    \item \label{item:hypcell homeo} The map $f_S$ restricts to a homeomorphism $\varphi: g_X^{-1}(C) \to S$.
    
    \item \label{item:hypcell retraction} The map $\varphi^{-1} \circ f_S:\mathcal H_S(X) \to g_X^{-1}(C)$ is a retraction.
    
    \item \label{item:hypcell compact} If $X$ and $S$ are compact, then $\mathcal H_S(X)$ is compact too.

\end{enumerate}

\end{lemma}
\proof
Let us denote by $f_C$ the restriction of the folding map $f$ to $C$. Note that $f_C:C\to \sigma^n$ is a homeomorphism.
To prove \eqref{item:hypcell homeo}, let $\varphi:g_X^{-1}(C)\to S$ be the restriction of $f_S$ to $g_X^{-1}(C)$. Then $\varphi$ is continuous, because it is just the restriction of the projection $X\times S\to S$. 
Injectivity and surjectivity of $\varphi$ follow respectively from those of $f_C$. 
To conclude, we construct an explicit continuous inverse. Consider the map $\lambda:S\to C$, $\lambda(s)=f_C^{-1}(g(s))$. Notice it is well--defined (because $g$ is surjective), and continuous.
Then the map $\psi:S\to g_X^{-1}(C)\subseteq X\times S$, $\psi(s)=(\lambda(s),s)$ provides a continuous inverse to $\varphi$. 

Now \eqref{item:hypcell retraction} follows from   \eqref{item:hypcell homeo}, as every element of $g_X^{-1}(C)$ is of the form $(\lambda(s),s)$. Finally, to prove \eqref{item:hypcell compact}, observe that if $X$ is compact, then it consists of finitely many $n$--cells. As a result of the previous argument, $\mathcal H_S(X)$ is covered by finitely many copies of the compact space $S$, hence it is compact.
\endproof

Depending on the applications in which they are interested, authors differ on what additional geometric conditions they require on the association $X\to \mathcal H_S(X)$, hence they start with different spaces $(S,g)$. We refer the reader to \cite{DJ91} for a very general treatment of how properties of $(S,g)$ imply properties of $\mathcal H_S(X)$.
Some commonly required conditions are the following

\begin{enumerate}

\item \label{item:hypdef_hyperbolicity} (Hyperbolicity): $\mathcal H_S(X)$ admits a non-positively curved metric.

\item \label{item:hypdef_functoriality} (Functoriality): if $Z\subseteq X$ is a locally convex subcomplex, then $\mathcal H_S(Z)\subseteq \mathcal H_S(X)$ is a locally convex subspace.

\item \label{item:hypdef_local} (Local structure): if $C\subseteq X$ is an $n$--cell, then $\mathcal H_S(C)$ is an $n$-manifold with boundary and corners, and $\lk{\mathcal H_S(C),\mathcal H_S(X)} \cong \lk{C,X}$. In particular, if $X$ is a manifold, then $\mathcal H_S(X)$ is a manifold too.

\item \label{item:hypdef_homology} (Homology surjectivity): the map $g_X:\mathcal H_S(X)\to X$ induces a surjection on homology.

\end{enumerate}

The association $X\to \mathcal H_S(X)$ is then called the \textit{hyperbolization procedure} defined by $(S,g)$.
We call $S$ the \textit{hyperbolizing cell}, and $\mathcal H_S(X)$ the \textit{hyperbolized complex}.
Despite the name (established in the literature), the output $\mathcal H_S(X)$ of a  hyperbolization procedure is a metric space which  a priori is just non-positively curved. A \textit{strict hyperbolization} is one for which $\mathcal H_S(X)$ is negatively curved.
In this paper we will consider a (non--strict) hyperbolization for simplicial complexes due to Gromov (see \S\ref{sec:Gromov construction}), and a strict hyperbolization for cubical complexes due to Charney and Davis (see \S\ref{sec:strict hyperbolization}).


\begin{remark}\label{rem:faces--create topology} 
If $(S,g)$ is a given hyperbolizing cell, $g:S\to \sigma^n$ is surjective, and $F\subseteq \sigma^n$ is a closed face of the $n$--cell $\sigma^n$, then the subspace $g^{-1}(F)$ will be called a \textit{face} of $S$. The dimension of a face of $S$ is defined to be simply the dimension of the corresponding face of $\sigma^n$.
Note that a face of $S$ does not need to be connected.
When this happens, $\mathcal H_S(X)$ may fail to be simply connected, even if both $X$ and $S$ are. 
For some interesting examples, see \cite[1b.1]{DJ91}, or consider the elementary one in Figure~\ref{fig:hyperbolization_topology}. 
Despite their non-trivial role in the construction, most of the times the maps $f$ and $g$ are omitted from the notation.
\end{remark}

\begin{figure}[ht]
    \centering
    \includegraphics[width=.8\textwidth]{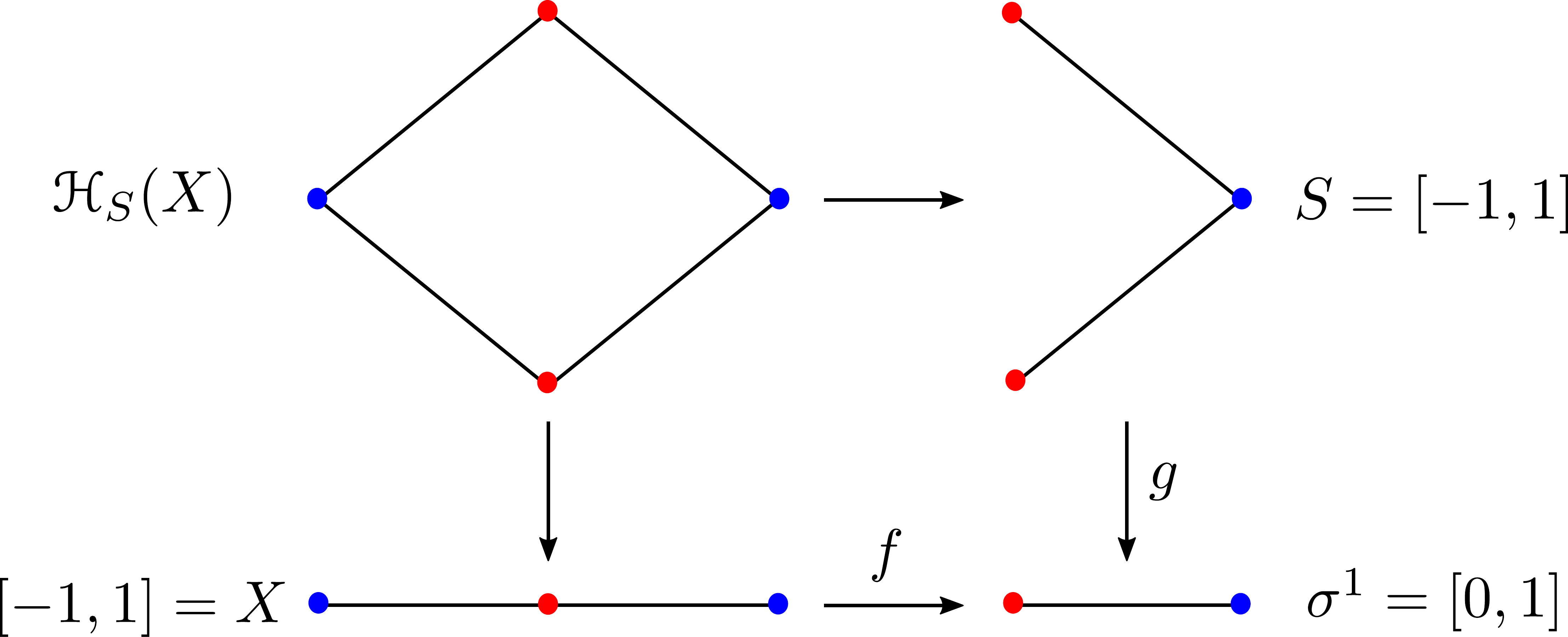}
    \caption{Example for Remark~\ref{rem:faces--create topology}: a face of the hyperbolizing cell $S$ is not connected and $\mathcal H_S(X)$ is not simply connected. The maps $f$ and $g$ here are defined by the vertex coloring. }
    \label{fig:hyperbolization_topology}
\end{figure}


\subsection{Gromov's cylinder construction}\label{sec:Gromov construction}

In this section we review a construction, due to Gromov, which turns a simplicial complex $K$ into a foldable cubical complex $\mathcal G(K)$ having non-positive curvature (see \cite[\S 3.4.A]{G87} for the original source, or \cite[\S 4c]{DJ91}, \cite[\S 4]{P91}, and references therein, for expository accounts). 

The construction uses induction on dimension and pullback simultaneously, following this scheme.
For each dimension $n\geq 1$ we will first define $\gromov {\triangle^n}$ and a map $g:\gromov {\triangle^n}\to \triangle^n$, then for any foldable $n$--dimensional simplicial complex $K$, with a folding $f:K\to \triangle^n$, we will define $\gromov K$ via the pullback square (compare \S \ref{sec:hyperbolization template})
\begin{center}
\begin{tikzpicture}
\node at (-1,1) (A) {$\gromov K$};
\node at (-1,-1) (B) {$K$};
\node at (1,-1) (C) {$\triangle^n$};
\node at (1,1) (D) {$\gromov {\triangle^n}$};
\draw [dashed,->] (A) edge (B)   (A) edge (D) ;
\draw [->]  (B) edge (C) (D) edge (C);
\node at (0,-1.25) {$f$};
\node at (1.25,0) {$g$};
\node at (0,1.25) {$f_{\gromov{\triangle^n}}$};
\node at (-1.4,0) {$g_K$};
\end{tikzpicture}
\end{center}
Finally, for a general $K$ (not necessarily foldable), we will define $\gromov K=\gromov{\barsub K}$ (recall that the barycentric subdivision is always foldable by Lemma~\ref{lem:barsub_folds}). Note that in any case the construction equips $\gromov K$ with a natural map to $\triangle^n$.

For $n=1$ we set $\gromov{\triangle^1}=\triangle^1$, and we define $g:\gromov {\triangle^1}\to \mathcal \triangle^1$ to be just the identity. 
By the pullback construction this defines $\gromov K$ and a map  $g:\gromov K\to \triangle^1$ for all simplicial graphs $K$. Concretely, when $K$ is a simplicial graph, then $\gromov K= K$ if $K$ is bipartite, and $\gromov K=\barsub K$ otherwise; the folding to $\triangle^1$ is induced by the bipartition.

\begin{figure}[h]
\centering
\def\svgwidth{\columnwidth}
\begingroup%
  \makeatletter%
  \providecommand\color[2][]{%
    \errmessage{(Inkscape) Color is used for the text in Inkscape, but the package 'color.sty' is not loaded}%
    \renewcommand\color[2][]{}%
  }%
  \providecommand\transparent[1]{%
    \errmessage{(Inkscape) Transparency is used (non-zero) for the text in Inkscape, but the package 'transparent.sty' is not loaded}%
    \renewcommand\transparent[1]{}%
  }%
  \providecommand\rotatebox[2]{#2}%
  \newcommand*\fsize{\dimexpr\f@size pt\relax}%
  \newcommand*\lineheight[1]{\fontsize{\fsize}{#1\fsize}\selectfont}%
  \ifx\svgwidth\undefined%
    \setlength{\unitlength}{1782.73433655bp}%
    \ifx\svgscale\undefined%
      \relax%
    \else%
      \setlength{\unitlength}{\unitlength * \real{\svgscale}}%
    \fi%
  \else%
    \setlength{\unitlength}{\svgwidth}%
  \fi%
  \global\let\svgwidth\undefined%
  \global\let\svgscale\undefined%
  \makeatother%
  \begin{picture}(1,0.47847341)%
    \lineheight{1}%
    \setlength\tabcolsep{0pt}%
    \put(0,0){\includegraphics[width=\unitlength,page=1]{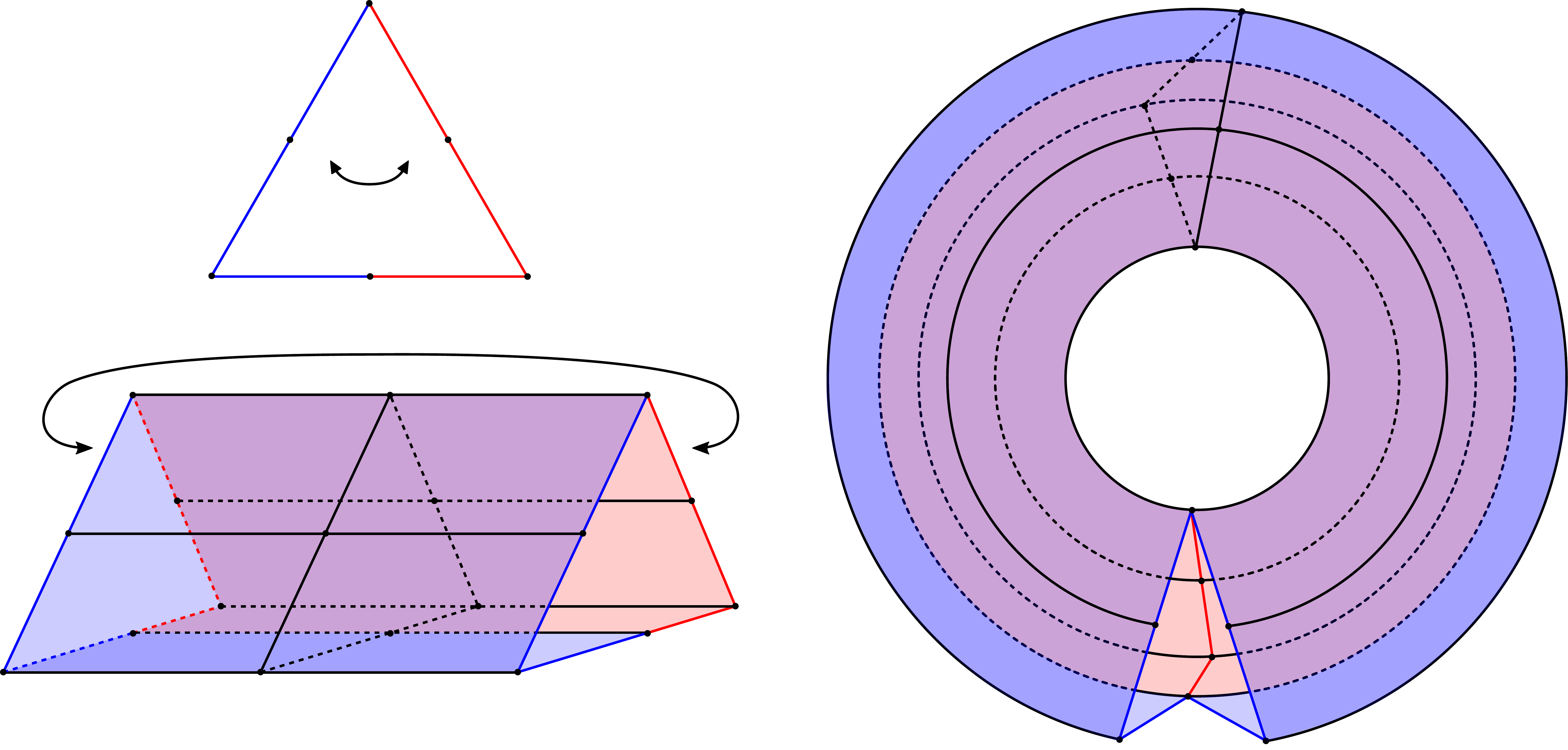}}%
    \put(0.27227254,0.45229863){\color[rgb]{0,0,0}\makebox(0,0)[lt]{\lineheight{1.25}\smash{\begin{tabular}[t]{l}$\gromov{\partial \triangle^2}$\end{tabular}}}}%
    \put(0.12323041,0.00347571){\color[rgb]{0,0,0}\makebox(0,0)[lt]{\lineheight{1.25}\smash{\begin{tabular}[t]{l}$\gromov{\partial \triangle^2} \times [-1,1]$\end{tabular}}}}%
    \put(0.73122818,0.229715){\color[rgb]{0,0,0}\makebox(0,0)[lt]{\lineheight{1.25}\smash{\begin{tabular}[t]{l}$\gromov{\triangle^2}$\end{tabular}}}}%
  \end{picture}%
\endgroup%

    \caption{Gromov's cylinder construction}
    \label{fig:gromov}
\end{figure}

Let us now assume by induction that for any simplicial complex $K$ of dimension at most $n-1$ the space $\mathcal G(K)$ is defined, and is endowed with a map to $\triangle^{n-1}$. 
In order to define $\gromov {\triangle^n}$, consider a reflection of $\partial \triangle^n$, and induce a reflection on  $\gromov{\partial \triangle^n}$. Let $U,V$ be the two closed half-spaces exchanged by the reflection, and define
$$
\gromov {\triangle^n}=\gromov {\partial \triangle^n}\times [-1,1] / \sim
$$
where $(u,t)\sim (u',t')$ if and only if $|t|=|t'|=1$ and $u=u'\in U$.   
Notice that taking a further quotient which identifies also points on $V$, one would get a map $\gromov {\triangle^n}\to \gromov {\partial \triangle^n}\times S^1$, and we can think of $\gromov {\triangle^n}$ as being obtained from $\gromov {\partial \triangle^n}\times S^1$ by cutting a slit in it along a half--fiber (see Figure~\ref{fig:gromov}).

By induction, $\gromov {\partial \triangle^n}$ is well--defined, and it comes with a map $\gromov {\partial \triangle^n}\to \triangle^{n-1}$. 
Notice that the boundary of $\gromov {\triangle^n}$ consists of two copies of $V$, glued along a subspace identifiable with $U\cap V$. In other words, $\partial \gromov {\triangle^n}$ can be naturally identified with $\gromov {\partial \triangle^n}$, hence $\partial \gromov {\triangle^n}$ comes with a map to $\partial \triangle^n$.
This map can be extended  to a map  $\gromov {\triangle^n}\to  \triangle^n$  as follows: take a regular neighborhood $N\cong \partial \gromov {\triangle^n} \times [0,1]$ inside $\gromov {\triangle^n}$, and identify $\triangle ^n$ with the cone over $\partial \triangle ^n$. Then extend the map over $N$ along the cone direction, and collapse the complement of $N$ to the cone point. This completes the construction of $\gromov{\triangle^n}$ and a map $g:\gromov{\triangle^n}\to \triangle ^n$. Arguing as above (i.e.\ with the template from \S \ref{sec:hyperbolization template}), this also defines $\gromov K$ for any simplicial complex $K$.

\begin{proposition}\label{prop:gromov_cylinder}
If $K$ is a simplicial complex, then $\gromov K$ is a foldable cubical complex of non--positive curvature. 
Moreover if $K$ is homogeneous (respectively without boundary,  locally compact, or compact), then $\gromov K$ is also homogeneous (respectively without boundary,  locally finite, or compact).
\end{proposition}
\proof
First we show that $\gromov K$ admits the structure of a cubical complex, starting with the case $K=\triangle ^n$.
This is clear for $\gromov{\triangle^1}=[0,1]=\square^1$.
Then, arguing by induction, $\gromov{\triangle^n}$ inherits a cubical structure from the one of $\gromov{\partial \triangle^n}\times [-1,1]$. Here we think of $[-1,1]$ as being given the standard cubical structures as a union of two unit intervals, and we give $\gromov{\partial \triangle^n}\times [-1,1]$ the standard cubical structure coming from the fact that $\square^{n-1}\times \square^1=\square^n$.
Since $\gromov K$ is in general defined via the pullback construction (see \S \ref{sec:hyperbolization template}), it inherits a natural cubical structure from $\gromov{\triangle^n}$.

We now prove that the cubical complex $\gromov K$ has the desired properties.
Foldability is proven in \cite[Lemma 7.5]{CD95}. Non--positive curvature is proven in \cite[Proposition 4c.2(3)]{DJ91}. For the other properties we argue as follows. 
Note that for each $n$ the hyperbolizing cell $\gromov{\triangle^n}$ is homogeneous, has a single boundary component, and satisfies $\partial \gromov{\triangle ^n}=g^{-1}(\partial \triangle ^n)$. So, if $K$ is homogeneous then $\gromov K$ is homogeneous, and if $K$ is without boundary, the same holds for $\gromov K$. It is proved in \cite[Lemma 1e.1 and \S 4c]{DJ91} that Gromov's construction preserves the local structure (e.g.\ links). This implies that if $K$ is locally finite, then so is $\gromov K$. In particular, by \eqref{item:hypcell compact} in Lemma~\ref{lem:cell_is_hyperbolizing_cell}, if $K$ is compact, then so is $\gromov K$, because $\gromov{\triangle^n}$ is compact.
\endproof

We have defined Gromov's construction for simplicial complexes. 
Given any cell complex $X$ we can first take its barycentric subdivision $\barsub X$ (which is a simplicial complex), and then apply Gromov's construction to it.  

\begin{corollary}\label{cor:use_gromov}
If $X$ is any cell complex, then $\gromov{\barsub X}$ is a foldable cubical complex of non--positive curvature.  Moreover if $X$ is homogeneous (respectively without boundary,  locally compact, or compact), then $\gromov{\barsub X}$ is also homogeneous (respectively without boundary,  locally compact, or compact).
\end{corollary}
\proof
We know $K=\barsub X$ is a (foldable) simplicial complex (by Lemma~\ref{lem:barsub_folds}), homeomorphic to $X$. Then the statements follow from Proposition~\ref{prop:gromov_cylinder}.
\endproof

Gromov's construction is known to satisfy even more properties, namely conditions (1)-(6) in \cite{CD95} and (1), (2'), (3)-(5) in \cite{DJ91}. Some of these are versions of conditions \eqref{item:hypdef_hyperbolicity}--\eqref{item:hypdef_homology} from \S \ref{sec:hyperbolization template}, while others deal with preservation of stable tangent bundles and rational Pontryagin classes, when they are defined.
This is needed in the applications of the hyperbolization procedure to construct examples of closed aspherical manifolds with various prescribed features or pathologies (as in \cite{DJ91,O20}).
 

 \section{Strict hyperbolization}\label{sec:strict hyperbolization}

The hyperbolization procedure introduced by Charney and Davis in \cite{CD95} is defined for cubical complexes, and fits in the framework outlined in \S \ref{sec:hyperbolization template}, in the sense that it is determined by the choice of a hyperbolizing cell. Differently from Gromov's cylinder construction (described in \S \ref{sec:Gromov construction}), this procedure is not defined by induction. Rather, for each dimension $n>0$ a hyperbolizing cell is defined independently, and defines a hyperbolization procedure for $n$--dimensional cubical complexes.

While the original construction is a bit more general than the version we use here, we find it convenient to impose some mild restrictions on the cubical complex in order to simplify the presentation.
From now on assume $X$ is \textit{admissible}, i.e.\ it satisfies the following conditions (see \S~\ref{sec:cell complexes and hyperbolization procedures} for definitions):
\begin{enumerate}
    \item cubical; 
    \item homogeneous, without boundary;
    \item foldable;
    \item non-positively curved; 
    \item locally compact. 
\end{enumerate}
This setting, consistent with that of \cite{XIE04}, is more general than the one in \cite{BS99}, as we do not require gallery--connectedness.
In particular, we allow $X$ to be a pseudomanifold.
On the other hand, the first two conditions are a bit more restrictive than the corresponding ones in \cite{CD95}, while the other ones are the same. 
More precisely, if $X$ is foldable, then necessarily cubes of $X$ are embedded. In \cite{CD95} they allow two cubes to meet in a proper union of faces; note that such faces have to be disjoint in each cube, because non--positive curvature guarantees that links of vertices are simplicial. In particular, up to performing cubical subdivision, one can always assume that $X$ is cubical.
Finally we remark that at this stage we are only assuming local finiteness instead of compactness of $X$. While in our main theorems (Theorems \ref{mainthm1} and \ref{mainthm2}) we assume that the complex is compact (in order to get a hyperbolic group), most of the geometric and combinatorial arguments do not need that, and in \S\ref{lem:hypcube_quasiconvex_in_lattice} we actually need to consider a certain hyperbolization of $\rr^n$.

The main contribution of this section is to define some subspaces of the space that results from strict hyperbolization on an admissible cubical complex $X$.
We call such subspaces \textit{mirrors}, and prove that their lifts to the universal cover are convex and separating (see Proposition~\ref{prop:mirrors convex} and Proposition~\ref{prop:mirrors separate} respectively).
Along the way, we also study a combinatorial decomposition of the universal cover (see \S\ref{sec:universal_cover} and \S\ref{subsec:stratification}) that will be the starting point for the construction of the dual cubical complex in \S\ref{sec:dual cubical complex}.

\subsection{The hyperbolizing cell}\label{sec:CD hyperbolizing cell}
The hyperbolizing cell used in this hyperbolization procedure is a certain hyperbolic manifold with boundary and corners, obtained by cutting a closed hyperbolic manifold along a suitable collection of pairwise orthogonal totally geodesic codimension--$1$ submanifolds. While the existence of such an object is clear in dimension $2$ (see Figure~\ref{fig:CD hyp cube}), the construction in higher dimension requires some arithmetic methods involving quadratic forms (see \S \ref{sec:construction of CD cell} below for more details).
Specifically, the construction relies on the choice of a suitable congruence subgroup $\Gamma$ of an arithmetic lattice in $\so$, so we will denote the hyperbolizing cell by $\hq$.
Here and in the following we denote by $\igq$ the group of Euclidean isometries of the standard cube $\square^n$.
Also recall from Remark~\ref{rem:faces--create topology} that a $k$--face of $\hq$ is by definition a subspace of the form $\cd^{-1}(\square^k)$, where $\square^k$ is a $k$--face of $\square^n$.

\begin{lemma}[Corollary 6.2 in \cite{CD95}]\label{lem:CD cell}
For every $n\geq 2$ there exists a compact, connected, orientable hyperbolic $n$--manifold with corners $\hq$, an isometric action of $\igq$ on $\hq$, and a $\igq$--equivariant and face--preserving map $\cd:\hq\to \square^n$, such that the following hold.
\begin{enumerate}
    \item The poset of faces of $\hq$ is $\igq$--equivariantly isomorphic to that of $\square^n$.
    
    \item Each face of $\hq$ is totally geodesic.
    
    \item The faces of $\hq$ intersect orthogonally.
    
    \item \label{item:CD 0-cell} Each $0$--dimensional face is a single point.
    
    \item \label{item:CD cell CD map} The map $\cd:\hq\to \square^n$ and its restriction to each face have degree one.
\end{enumerate}
\end{lemma}
We call $\hq$ the \textit{hyperbolizing cube}, and $\cd$ the \textit{Charney--Davis map}. We denote by $\hgq=\pi_1(\hq)$ the fundamental group of the hyperbolizing cube.

\begin{remark}\label{rmk:CD lowerdim faces hyp 1}
In this hyperbolization procedure, a $k$--face of $\hq$ is guaranteed to be connected when $k=0,n$, but may be disconnected otherwise (see Remark~\ref{rem:faces--create topology}, and the Remark after Corollary 6.2 in \cite{CD95}).
Nevertheless, by abuse of notation, we will denote by $\hqk=\cd^{-1}(\square^k)$ the $k$--face of $\hq$, even when $0<k<n$. 
Notice that $\hqk$ is a priori different from the $k$--dimensional hyperbolizing cube, i.e.\ the hyperbolizing cell that one obtains by hyperbolizing a $k$--dimensional cube with a hyperbolizing lattice $\Lambda\subseteq \operatorname{SO}_0(k,1)$ for $0<k<n$. Namely, $\square^k_\Lambda$ is always connected by construction.
Finally, with respect to \eqref{item:CD cell CD map} in Lemma~\ref{lem:CD cell}, when $\hqk$ is disconnected, there is a preferred component of $\hqk$ on which $\cd$ has degree one, while it has degree zero on the other components (see \cite[Lemma 5.9(b)]{CD95} and \S\ref{sec:construction of CD cell} for details).
\end{remark}

\subsubsection{The construction of \texorpdfstring{$\hq$}{the hyperbolizing cube}}\label{sec:construction of CD cell}

To construct the hyperbolizing cube $\hq$, Charney and Davis consider the hyperboloid model for $\hh^n$ inside Minkowski space $\mink$, i.e.\ the space $\rr^{n+1}$ equipped with a quadratic form of signature $(n,1)$. The isometry group of $\mink$ is naturally identified with the indefinite orthogonal group $\oofull$, and its connected component $\so$ is naturally identified with the group of orientation preserving isometries of $\hh^n$.
Then they show that $\so$ contains an arithmetic lattice $\Gamma$ which enjoys some key properties, from which the properties of $\hq$ in Lemma~\ref{lem:CD cell} follow. 
In particular, $\Gamma$ is a cocompact torsion--free lattice of $\so$, whose normalizer in $\oofull$ contains all the permutations of the coordinates $x_1,\dots,x_n$, and all the reflections in the coordinate hyperplanes $H_i=\{(x_1,\dots,x_n,x_{n+1}) \in \mink \ | \ x_i=0\}$ for $i=1,\dots,n$. Note that these generate a group of isometries isomorphic to $\igq$.
We will refer to the lattice constructed in \cite[\S 6]{CD95} as the \textit{hyperbolizing lattice}.

\begin{figure}[h]
\centering
\def\svgwidth{\columnwidth}
\begingroup%
  \makeatletter%
  \providecommand\color[2][]{%
    \errmessage{(Inkscape) Color is used for the text in Inkscape, but the package 'color.sty' is not loaded}%
    \renewcommand\color[2][]{}%
  }%
  \providecommand\transparent[1]{%
    \errmessage{(Inkscape) Transparency is used (non-zero) for the text in Inkscape, but the package 'transparent.sty' is not loaded}%
    \renewcommand\transparent[1]{}%
  }%
  \providecommand\rotatebox[2]{#2}%
  \newcommand*\fsize{\dimexpr\f@size pt\relax}%
  \newcommand*\lineheight[1]{\fontsize{\fsize}{#1\fsize}\selectfont}%
  \ifx\svgwidth\undefined%
    \setlength{\unitlength}{1513.77760519bp}%
    \ifx\svgscale\undefined%
      \relax%
    \else%
      \setlength{\unitlength}{\unitlength * \real{\svgscale}}%
    \fi%
  \else%
    \setlength{\unitlength}{\svgwidth}%
  \fi%
  \global\let\svgwidth\undefined%
  \global\let\svgscale\undefined%
  \makeatother%
  \begin{picture}(1,0.35199114)%
    \lineheight{1}%
    \setlength\tabcolsep{0pt}%
    \put(0,0){\includegraphics[width=\unitlength,page=1]{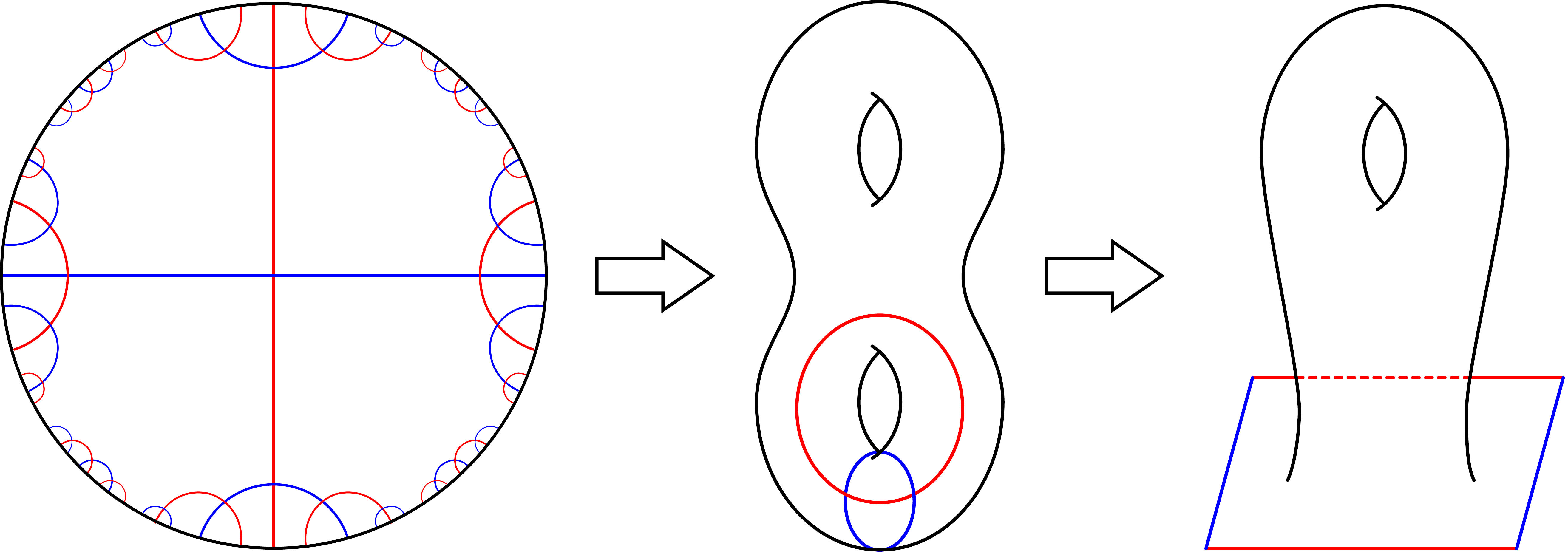}}%
    \put(0.18393379,0.26376914){\color[rgb]{1,0,0}\makebox(0,0)[lt]{\lineheight{1.25}\smash{\begin{tabular}[t]{l}$H_1$\end{tabular}}}}%
    \put(0.05454743,0.14315719){\color[rgb]{0,0,1}\makebox(0,0)[lt]{\lineheight{1.25}\smash{\begin{tabular}[t]{l}$H_2$\end{tabular}}}}%
    \put(0.54479385,0.1646905){\color[rgb]{1,0,0}\makebox(0,0)[lt]{\lineheight{1.25}\smash{\begin{tabular}[t]{l}$M_1$\end{tabular}}}}%
    \put(0.54479393,0.01083283){\color[rgb]{0,0,1}\makebox(0,0)[lt]{\lineheight{1.25}\smash{\begin{tabular}[t]{l}$M_2$\end{tabular}}}}%
    \put(0.28781602,0.33272854){\color[rgb]{0,0,0}\makebox(0,0)[lt]{\lineheight{1.25}\smash{\begin{tabular}[t]{l}$\hh^n$\end{tabular}}}}%
    \put(0.62752149,0.3296832){\color[rgb]{0,0,0}\makebox(0,0)[lt]{\lineheight{1.25}\smash{\begin{tabular}[t]{l}$\hm$\end{tabular}}}}%
    \put(0.94461055,0.3296832){\color[rgb]{0,0,0}\makebox(0,0)[lt]{\lineheight{1.25}\smash{\begin{tabular}[t]{l}$\hq$\end{tabular}}}}%
  \end{picture}%
\endgroup%

    \caption{A hyperbolizing cube}
    \label{fig:CD hyp cube}
\end{figure}

If $\Gamma$ is such a lattice, then it acts freely, properly discontinuously, and cocompactly by orientation--preserving isometries on $\hh^n$.
We can consider the closed connected oriented hyperbolic $n$--manifold  $\hm = \hh^n / \Gamma$. 
The hyperplanes $H_i$ descend to codimension--$1$ submanifolds $M_i = H_i / \stab{\Gamma}{H_i}$ which are closed, oriented, totally geodesic and pairwise orthogonal (see Figure~\ref{fig:CD hyp cube}).
Then the hyperbolizing cell $\hq$ is defined to be the metric completion of the space $\hm \setminus \cup_{i=1}^n M_i$, with respect to the length metric induced on the complement of $\cup_{i=1}^n M_i$. This is the manifold with boundary and corners obtained by cutting $\hm$ open along the submanifolds $M_1,\dots,M_n$ (see \cite[\S 5]{CD95}).
In particular, the map $\cd:\hq\to \square^n$ is induced by the collapse map $\cd_0:\hm\to (S^1)^n$ obtained by applying the Pontryagin-Thom construction to $\hm$ with respect to each of the codimension--$1$ submanifolds $M_1,\dots,M_n$.

\begin{remark}\label{rem:deeper hyp 1}
It is implicit in \cite{CD95} that a hyperbolizing lattice $\Gamma$ contains infinitely many other hyperbolizing lattices as proper subgroups. They still enjoy the properties which are relevant for the construction, and provide corresponding hyperbolizing cubes. 
As observed by Ontaneda in \cite[Lemma 2.1]{O17}, this can be used to produce  hyperbolizing cubes for which the normal injectivity radius of the faces is arbitrarily large.
\end{remark}


\subsection{The hyperbolized complex}\label{sec:CD hyperbolized complex}
Following the template of \S \ref{sec:hyperbolization template}, to define the strict hyperbolization procedure of \cite{CD95} we proceed as follows. For each dimension $n>0$, we choose the hyperbolizing cell to be the hyperbolizing cube $(\hq,\cd)$ defined in \S\ref{sec:CD hyperbolizing cell}. 
Then for  any  foldable cubical complex  $X$ of dimension $n$, we define the \textit{hyperbolized complex} to be the space $\hc$ obtained as the fiber product of the folding map $\fold :X\to \square^n$ and the Charney-Davis map  $\cd:\hq \to \square^n$, i.e.\ by the pullback square in  Figure~\ref{fig:CD_hypcomplexdiagram}.

\begin{figure}[ht]
    \centering
    \begin{tikzpicture}
\node at (-1,1) (A) {$\hc$};
\node at (-1,-1) (B) {$X$};
\node at (1,-1) (C) {$\square^n$};
\node at (1,1) (D) {$\hq$};
\draw [dashed,->] (A) edge (B)   (A) edge (D) ;
\draw [->]  (B) edge (C) (D) edge (C);
\node at (0,1.25) {$\foldc$};
\node at (0,-1.25) {$\fold$};
\node at (1.25,0) {$\cd$};
\node at (-1.4,0) {$\cdX$};
\end{tikzpicture}
    \caption{The hyperbolized complex $\hc$ as a fibered product.}
    \label{fig:CD_hypcomplexdiagram}
\end{figure}

\begin{figure}[ht]
    \centering
    \def\svgwidth{\columnwidth}
\begingroup%
  \makeatletter%
  \providecommand\color[2][]{%
    \errmessage{(Inkscape) Color is used for the text in Inkscape, but the package 'color.sty' is not loaded}%
    \renewcommand\color[2][]{}%
  }%
  \providecommand\transparent[1]{%
    \errmessage{(Inkscape) Transparency is used (non-zero) for the text in Inkscape, but the package 'transparent.sty' is not loaded}%
    \renewcommand\transparent[1]{}%
  }%
  \providecommand\rotatebox[2]{#2}%
  \newcommand*\fsize{\dimexpr\f@size pt\relax}%
  \newcommand*\lineheight[1]{\fontsize{\fsize}{#1\fsize}\selectfont}%
  \ifx\svgwidth\undefined%
    \setlength{\unitlength}{1607.84081647bp}%
    \ifx\svgscale\undefined%
      \relax%
    \else%
      \setlength{\unitlength}{\unitlength * \real{\svgscale}}%
    \fi%
  \else%
    \setlength{\unitlength}{\svgwidth}%
  \fi%
  \global\let\svgwidth\undefined%
  \global\let\svgscale\undefined%
  \makeatother%
  \begin{picture}(1,0.39529988)%
    \lineheight{1}%
    \setlength\tabcolsep{0pt}%
    \put(0,0){\includegraphics[width=\unitlength,page=1]{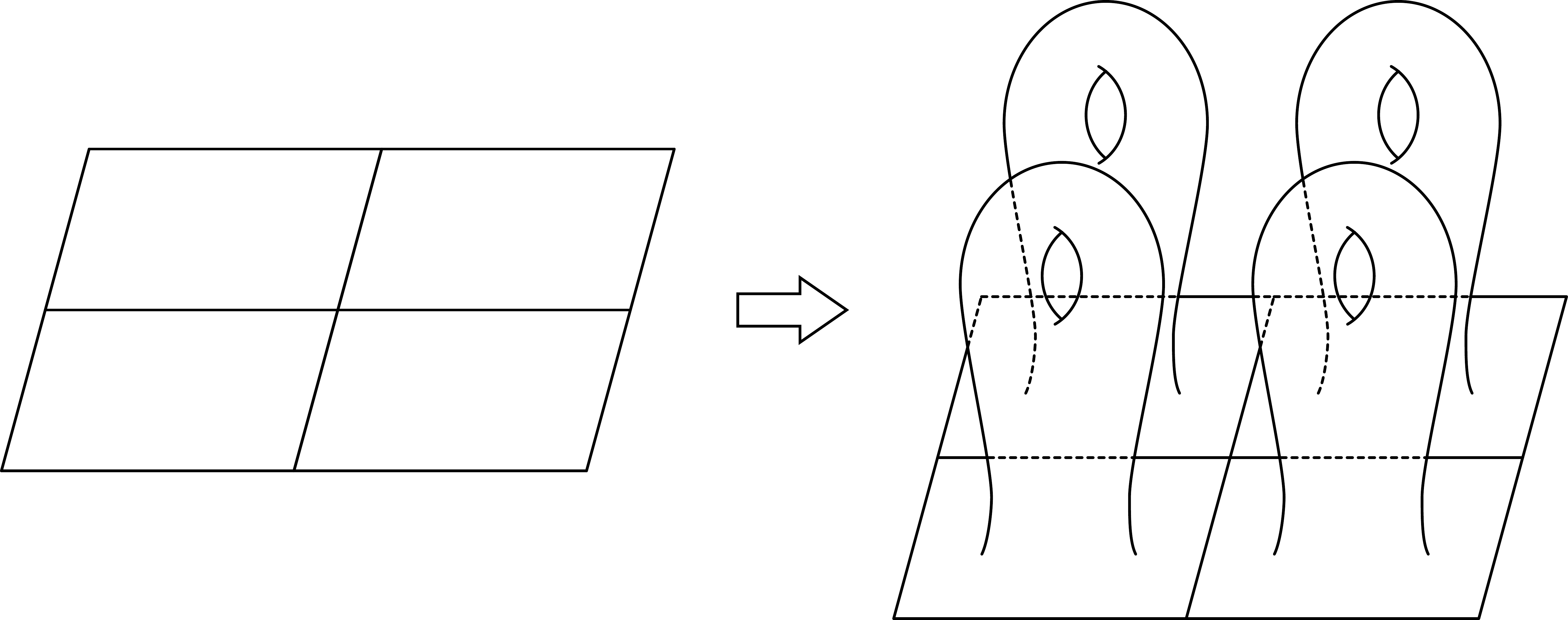}}%
    \put(0.01951348,0.10956739){\color[rgb]{0,0,0}\makebox(0,0)[lt]{\lineheight{1.25}\smash{\begin{tabular}[t]{l}$X$\end{tabular}}}}%
    \put(0.58906561,0.01300913){\color[rgb]{0,0,0}\makebox(0,0)[lt]{\lineheight{1.25}\smash{\begin{tabular}[t]{l}$\hc$\end{tabular}}}}%
  \end{picture}%
\endgroup%

    \caption{Strict hyperbolization of a square complex.}
    \label{fig:hyperbolized_complex}
\end{figure}

\begin{remark}\label{rem:CD hyp complex metric}
By \eqref{item:CD cell CD map} in Lemma~\ref{lem:CD cell} we know that $\cd$ is surjective. So, Lemma~\ref{lem:cell_is_hyperbolizing_cell} allows us to think of $\hc$ as being obtained by replacing every $n$--cube of $X$ by a hyperbolizing cube $\hq$, in the following sense (see  Figure~\ref{fig:hyperbolized_complex}). 
If $C$ is a top--dimensional cube of $X$, then its preimage $\cdX^{-1}(C)$ in $\hc$ is homeomorphic to $\hq$ (see \eqref{item:hypcell homeo} in Lemma~\ref{lem:cell_is_hyperbolizing_cell}).
Then one can endow $\hc$ with a length metric by gluing together these local metrics. 
In particular, $\foldc:\hc\to \hq$ induces an isometry $\cdX^{-1}(C) \to \hq$ for each top--dimensional cube $C\subseteq X$.
For a concrete example, if $X$ is (a suitable cubical subdivision of) the standard cubical structure on the $n$--torus, then $\hc$ is a closed hyperbolic manifold (see \cite[Lemma 3.2]{BE07} for details). Indeed, the piecewise hyperbolic metric obtained by gluing the hyperbolizing cubes together has no singularity and is in fact globally smooth and hyperbolic.
\end{remark}

We collect here some of the main properties of this construction which are relevant for our work.

\begin{proposition}[Proposition 7.1 in \cite{CD95}]\label{prop:CD complex}
For every $n\geq 2$ and every $n$--dimensional foldable cubical complex $X$, the space $\hc$ carries the structure of an $n$--dimensional piecewise hyperbolic cell complex, and is endowed with a map $\cdX:\hc\to X$, such that the following hold.
 
\begin{enumerate}
 
\item \label{item:CD faces and links} If $C\subseteq X$ is a $k$--cube, then  $\cdX^{-1}(C)\subseteq \hc$ is isometric to a $k$--face of $\hq$, and $\lk{\cdX^{-1}(C),\hc}$ is a piecewise spherical cell complex, isomorphic to $\lk{C,X}$.

\item \label{item:CD subcomplexes} If $Z\subseteq X$ is locally convex subcomplex of $X$, then $\cdX^{-1}(Z)$ is a locally convex subspace of $\hc$. 

\item \label{item:CD npc} If $X$ is locally $\cat 0$, then $\hc$ is locally $\cat{-1}$.

\item \label{item:CD hyperbolic} If $X$ is compact and locally $\cat 0$, then $\hg=\pi_1(\hc)$ is a Gromov hyperbolic group.

\end{enumerate}
\end{proposition}

\begin{remark}\label{rmk:CD lowerdim faces hyp 2}
The statement says in particular that if $C$ is a top--dimensional cube of $X$ then $\cdX^{-1}(C)$ is isometric to $\hq$ (compare Remark~\ref{rem:CD hyp complex metric}). 
On the other hand, if $C$ is a $k$--cube with $k<n$, then $\cdX^{-1}(C)$ is isometric to $\hqk=\cd^{-1}(\square^k)$, i.e.\ the hyperbolization of a lower dimensional cell, as introduced in Remark~\ref{rmk:CD lowerdim faces hyp 1}. 

If $Z\subseteq X$ is a $k$--dimensional subcomplex, the subspace $\cdX^{-1}(Z)$ can be identified with the fibered product of the maps $\fold_{|_Z}:Z\to \square^k$ and $\cd^k:\hqk\to \square^k$, respectively obtained by restricting the folding map  $\fold :X\to \square^n$ to $Z$ and the Charney--Davis map $\cd:\hq\to \square^n$ to $\hqk$ (see Figure~\ref{fig:CD_hypcomplexdiagram_lowerdim}).
Loosely speaking, hyperbolization trickles down to the lower dimensional skeletons of the complex $X$.

\begin{figure}[h]
    \centering
    \begin{tikzpicture}
\node at (-1,1) (A) {$\cdX^{-1}(Z)$};
\node at (-1,-1) (B) {$Z$};
\node at (1,-1) (C) {$\square^k$};
\node at (1,1) (D) {$\hqk$};
\draw [dashed,->] (A) edge (B)   (A) edge (D) ;
\draw [->]  (B) edge (C) (D) edge (C);
\node at (0,-1.25) {$\fold_{|_Z}$};
\node at (1.5,0) {$\cd^k$};
\end{tikzpicture}
    \caption{Hyperbolization of lower dimensional subcomplexes}
    \label{fig:CD_hypcomplexdiagram_lowerdim}
\end{figure}

\end{remark}

\begin{remark}\label{rem:deeper hyp 2}
In this construction the choice of $X$ and $\Gamma$ are essentially independent. In particular for any fixed cubical complex $X$ one can consider deeper hyperbolizations by taking deeper hyperbolizing lattices (see Remark~\ref{rem:deeper hyp 1}). While the combinatorial geometry of the hyperbolized complex, controlled by $X$, remains unchanged under different choices of the hyperbolizing lattice, its hyperbolic geometry can be quantitatively improved by an appropriate choice of the hyperbolizing lattice, as observed by Ontaneda in \cite[Lemma 2.1]{O17}.
\end{remark}

\begin{remark}[Finding codimension-1 subspaces] \label{rem:where are the hyperplanes}
The original approaches to cubulating a group $G$ relied on producing sufficiently many codimension one subgroups inside $G$ (see \cite{SA95,SA97,HW14,BW12}).

Since the copies of $\hq$ in the hyperbolized complex $\hc$ are obtained from an arithmetic hyperbolic manifold, they contain a large supply of compact totally geodesic codimension one submanifolds. It is tempting to try and use these to produce codimension one subgroups in the hyperbolized group $\hg=\pi_1(\hc)$. 
The difficulty with this approach is due to lack of control on the angles at which these totally geodesic codimension one hypersurfaces intersect the boundary of $\hq$. This makes it unclear how to extend the proposed subspace past the boundary. One could take a geodesic extension, but it would not be clear what the global behaviour of the subspace would be (see left of Figure~\ref{fig:hyperplanes_fail}).
Or one could take a geodesic reflection, but that would give rise to a kink angle (see right of Figure~\ref{fig:hyperplanes_fail}). Given that $\hq$ has fixed finite diameter, kink angles too far from right angles might prevent the subspace from even being quasiconvex.

You can try to control the kink angle, for instance by requiring the codimension one submanifold to be orthogonal to all faces of $\hq$. In this case, the extension would be a locally convex subspace of $\hc$.
Examples of orthogonal subspaces can be obtained by noting that the symmetry group  of the cube $\igq $ acts on $\hq$ (see Lemma~\ref{lem:CD cell}). Each reflection  of $\igq$ has some fixed point set, which meets the boundary orthogonally and is totally geodesic.

However, one can only find finitely many such subspaces, both in the orthogonal case and in the case of  kink angles bounded away from $0$ (see \cite{SH91} and \cite[\S 5]{FLMS21}).
This would make it quite delicate to ensure that one can find enough such subspaces to apply the standard criteria for properness of the induced cubulation (such as those in \cite{BW12,HW14}).
To address these issues, we turn to a different type of subspaces, which we call mirrors. These are defined in \S\ref{subsec:mirrors convexity} using the foldability of $X$, and enjoy properties reminiscent of those of hyperplanes in a $\cat 0$ cube complex.
For the sake of clarity, the collection of mirrors also fails to provide a proper action of $\hg=\pi_1(\hc)$ on a $\cat 0$ cubical complex in the usual way.
Nevertheless, in \S\ref{sec:dual cubical complex} we will be able to use mirrors to construct an action of $\hg$ on a $\cat 0$ cubical complex for which the cube stabilizers are manageable, and are in a certain sense already detected by the action of $\hg$ by deck transformations on the universal cover $\uchc$ (see \S~\ref{subsec:action}). 
The  reader interested in these remarks should also compare this discussion with that in  Remark~\ref{rem:wallspaces} below.
\end{remark}

\begin{figure}[ht]
    \centering
    \includegraphics[width=\textwidth]{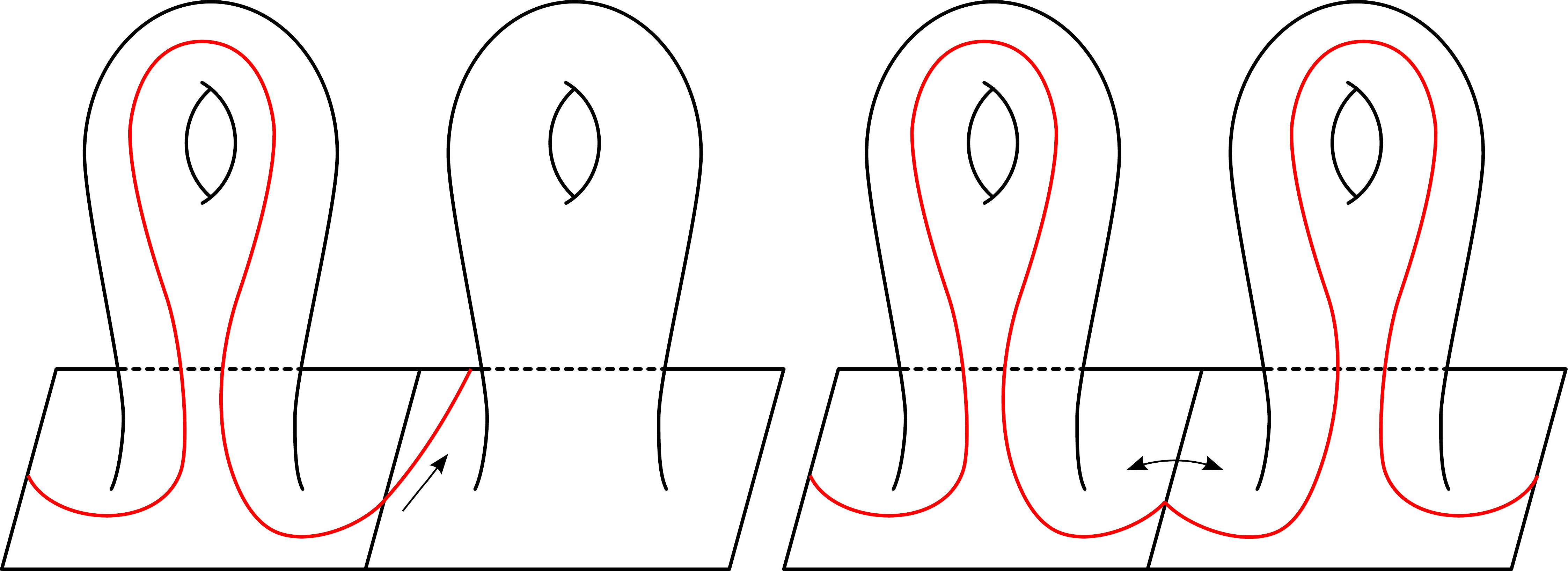}
    \caption{Failure of the attempt to create hyperplane--like subspaces. Left: geodesic extension. Right: geodesic reflection.}
    \label{fig:hyperplanes_fail}
\end{figure}


\subsection{Tiling, folding, and developing the universal cover}\label{sec:universal_cover}

Recall that we are assuming $X$ is an admissible complex, as defined at the beginning of \S\ref{sec:strict hyperbolization}.
It follows from Proposition~\ref{prop:CD complex} (see also Lemma~\ref{lem:cell_is_hyperbolizing_cell}) that the hyperbolized complex $\hc$ admits a decomposition into hyperbolized cubes, analogous to the decomposition of $X$ into cubes.
In this section we show how to obtain an analogous decomposition of the universal cover $\uchc$  of $\hc$ into pieces which are isometric to the universal cover $\uchq$ of the hyperbolizing cube. 
Let us denote by $\piuchc : \uchc\to\hc$ and $\pisquare : \uchq\to\hq$ the universal covering projections.

We start by realizing the space $\uchq$ as a convex subset of $\hh^n$. Let us consider once again the coordinate hyperplanes $H_i=\{(x_1,\dots,x_n,x_{n+1}) \in \mink \ | \ x_i=0\}$ for $i=1,\dots,n$ (introduced in \S \ref{sec:construction of CD cell}).
An \textit{open $\Gamma$--cell} is a connected component of the complement in $\hh^n$ of the collection of $\Gamma$--orbits of the hyperplanes $H_i$. A \textit{$\Gamma$--cell} is the closure of an open $\Gamma$--cell. 
Notice that all $\Gamma$--cells are convex, isometric to each other, and that $\Gamma$ permutes them transitively. 
It follows from the construction of $\hq$ in \S\ref{sec:construction of CD cell} that the universal cover $\uchq$ of $\hq$ can be isometrically identified with any $\Gamma$--cell (see Figure~\ref{fig:uchq_is_cell}).

\begin{figure}[h]
\centering
\def\svgwidth{\columnwidth}
\begingroup%
  \makeatletter%
  \providecommand\color[2][]{%
    \errmessage{(Inkscape) Color is used for the text in Inkscape, but the package 'color.sty' is not loaded}%
    \renewcommand\color[2][]{}%
  }%
  \providecommand\transparent[1]{%
    \errmessage{(Inkscape) Transparency is used (non-zero) for the text in Inkscape, but the package 'transparent.sty' is not loaded}%
    \renewcommand\transparent[1]{}%
  }%
  \providecommand\rotatebox[2]{#2}%
  \newcommand*\fsize{\dimexpr\f@size pt\relax}%
  \newcommand*\lineheight[1]{\fontsize{\fsize}{#1\fsize}\selectfont}%
  \ifx\svgwidth\undefined%
    \setlength{\unitlength}{1205.76473687bp}%
    \ifx\svgscale\undefined%
      \relax%
    \else%
      \setlength{\unitlength}{\unitlength * \real{\svgscale}}%
    \fi%
  \else%
    \setlength{\unitlength}{\svgwidth}%
  \fi%
  \global\let\svgwidth\undefined%
  \global\let\svgscale\undefined%
  \makeatother%
  \begin{picture}(1,0.46199345)%
    \lineheight{1}%
    \setlength\tabcolsep{0pt}%
    \put(0,0){\includegraphics[width=\unitlength,page=1]{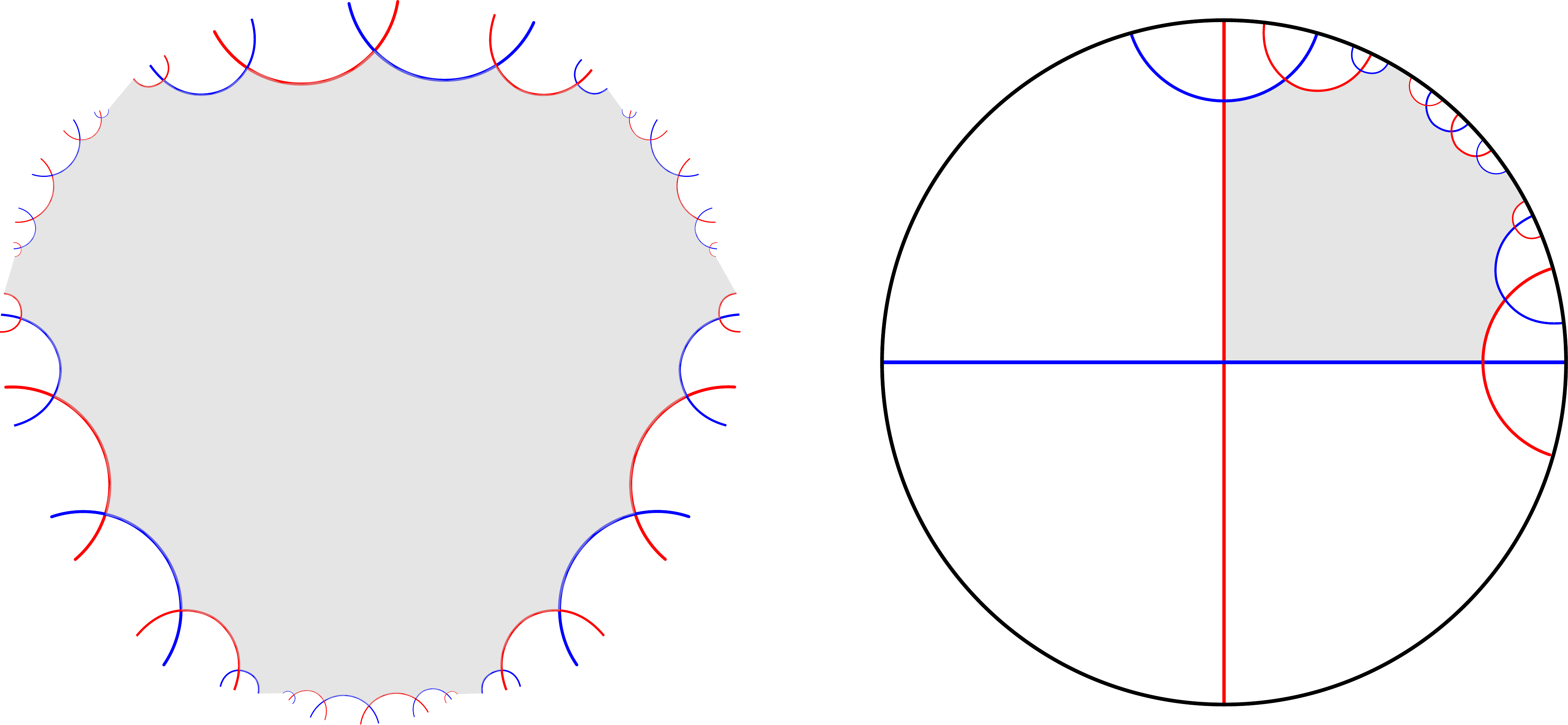}}%
    \put(0.10270339,0.28163188){\color[rgb]{0,0,0}\makebox(0,0)[lt]{\lineheight{1.25}\smash{\begin{tabular}[t]{l}$\uchq$\end{tabular}}}}%
    \put(0.80517254,0.28163188){\color[rgb]{0,0,0}\makebox(0,0)[lt]{\lineheight{1.25}\smash{\begin{tabular}[t]{l}$\uchq$\end{tabular}}}}%
    \put(0.65007382,0.10722537){\color[rgb]{0,0,0}\makebox(0,0)[lt]{\lineheight{1.25}\smash{\begin{tabular}[t]{l}$\hh^n$\end{tabular}}}}%
  \end{picture}%
\endgroup%

    \caption{The universal cover $\uchq$ of $\hq$, and its isometric embedding in $\hh^n$ as a $\Gamma$--cell}    
    \label{fig:uchq_is_cell}
\end{figure}

While it might be tempting to think that $\uchc$ is obtained via some simple fibered product construction involving $\widetilde X$ and $\uchq$, that is not the case.

\begin{remark}[What $\uchc$ is not] \label{rem:what_uchc_isnot}
Note that $\uchc \neq (\widetilde X)_\Gamma$, i.e.\ the universal cover of the hyperbolization of $X$ is not the hyperbolization of the universal cover of $X$. Indeed, $(\widetilde X)_\Gamma$ is not simply connected, because it retracts to $\hq$ by \eqref{item:hypcell retraction} in Lemma~\ref{lem:cell_is_hyperbolizing_cell}. 
Analogously, $\uchc$ is not the fiber product of $\widetilde X$ and $\uchq$ either. Indeed, note that the faces of $\uchq$ (i.e.\ the preimages of faces of $\square^n$ via the map $\cd \circ \pisquare$) are disconnected (see Figure~\ref{fig:uchq_is_cell}). This prevents the fiber product of $\widetilde X$ and $\uchq$ from being simply connected, as observed in Remark~\ref{rem:faces--create topology}.
\end{remark}

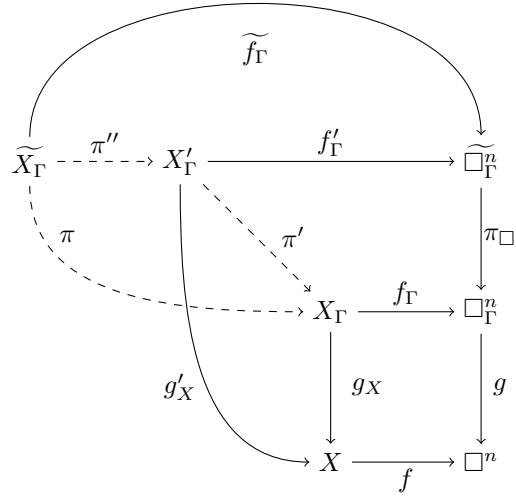
\begin{figure}[ht]
    \centering
    \begin{tikzpicture}
\node at (-1,1) (A) {$\hc$};
\node at (-1,-1) (B) {$X$};
\node at (1,-1) (C) {$\square^n$};
\node at (1,1) (D) {$\hq$};

\node at (-3,3) (F) {$\ichc$};
\node at (1,3) (G) {$\uchq$};
\node at (-5,3) (H) {$\uchc$};

\draw [->] (A) edge (B) (B) edge (C) (A) edge (D) (D) edge (C) (G) edge (D)  (F) edge (G);
\draw [->,dashed] (F) edge (A) (H) edge (F);
\draw [out=-90,in=180,->,dashed] (H) edge (A);
\draw [out=-90,in=180,->] (F) edge (B);
\draw [out=90,in=90,->] (H) edge (G);

\node at (0,1.25) {$\foldc$};
\node at (0,-1.25) {$\fold$};
\node at (1.25,0) {$\cd$};
\node at (1.25,2) {$\pisquare$};
\node at (-.5,0) {$\cdX$};
\node at (-3,0) {$\cdXi$};
\node at (-1,3.25) {$\foldic$};
\node at (-1.5,2) {$\pifromichc$};
\node at (-4,3.25) {$\pitoichc$};
\node at (-4.5,2) {$\piuchc$};
\node at (-2,4.5) {$\foldutocell$};
\end{tikzpicture}
    \caption{The hyperbolized complex $\hc$, its covering spaces, and the folding map.}
    \label{fig:CD_hypcomplexdiagram_coveringspaces}
\end{figure}

In order to address this, and get a working understanding of $\uchc$, we consider the intermediate space $\ichc$ obtained as a fibered product of $X$ and $\uchq$ along the maps $\fold : X\to \square^n$ and $\cd \circ \pisquare : \uchq\to \hq\to \square^n$ (see Figure~\ref{fig:CD_hypcomplexdiagram_coveringspaces}). By the universal property of pullbacks we have an induced map $\pifromichc: \ichc \to \hc$. 

\begin{lemma}
The map $\pifromichc: \ichc \to \hc$ is a covering map.
\end{lemma}
\proof
By the composition law for pullbacks the space $\ichc$ is actually the same as the pullback of $\foldc:\hc\to \hq$ and $\pisquare : \uchq\to \hq$. In particular, the map $\pifromichc$ is the pullback of the universal covering projection $\pisquare$ along the map $\foldc$, hence is itself a covering map. 
\endproof

In particular, $\ichc$ can be endowed with a length metric that makes $\pifromichc$ a local isometry (see \cite[Proposition I.3.25]{BH99}), and the universal cover $\uchc$ can be realized as the universal cover of this space $\ichc$, even in a metric sense.  Let $\pitoichc : \uchc \to \ichc$ denote the universal covering projection.

\begin{figure}[h]
\centering
\def\svgwidth{\columnwidth}
\begingroup%
  \makeatletter%
  \providecommand\color[2][]{%
    \errmessage{(Inkscape) Color is used for the text in Inkscape, but the package 'color.sty' is not loaded}%
    \renewcommand\color[2][]{}%
  }%
  \providecommand\transparent[1]{%
    \errmessage{(Inkscape) Transparency is used (non-zero) for the text in Inkscape, but the package 'transparent.sty' is not loaded}%
    \renewcommand\transparent[1]{}%
  }%
  \providecommand\rotatebox[2]{#2}%
  \newcommand*\fsize{\dimexpr\f@size pt\relax}%
  \newcommand*\lineheight[1]{\fontsize{\fsize}{#1\fsize}\selectfont}%
  \ifx\svgwidth\undefined%
    \setlength{\unitlength}{5283.440422bp}%
    \ifx\svgscale\undefined%
      \relax%
    \else%
      \setlength{\unitlength}{\unitlength * \real{\svgscale}}%
    \fi%
  \else%
    \setlength{\unitlength}{\svgwidth}%
  \fi%
  \global\let\svgwidth\undefined%
  \global\let\svgscale\undefined%
  \makeatother%
  \begin{picture}(1,0.29798257)%
    \lineheight{1}%
    \setlength\tabcolsep{0pt}%
    \put(0,0){\includegraphics[width=\unitlength,page=1]{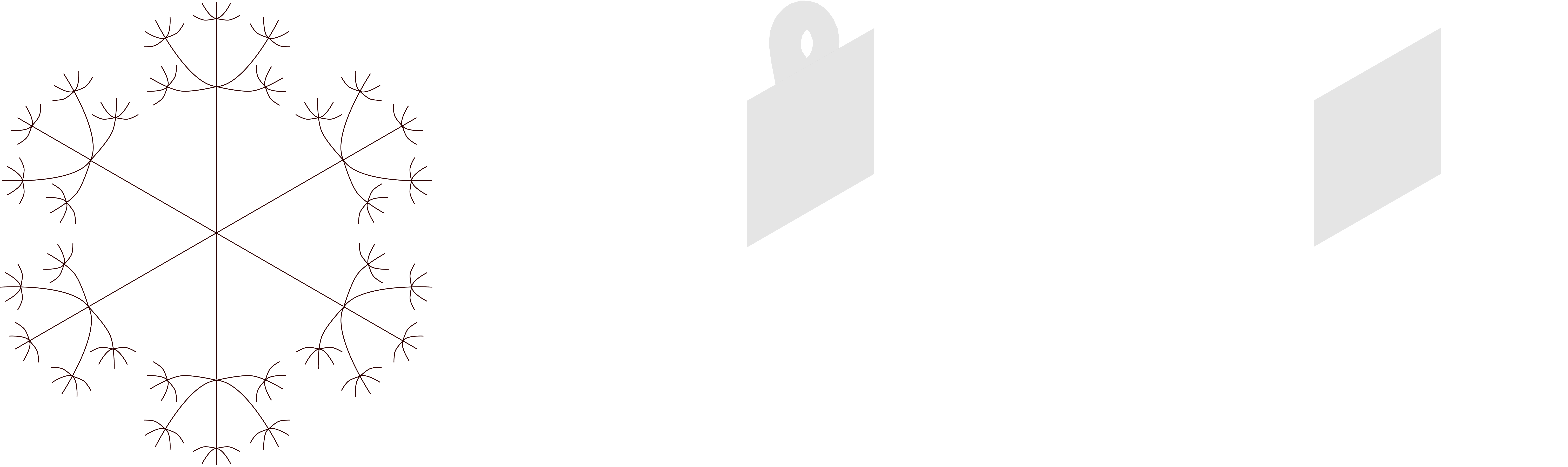}}%
    \put(0.25426357,0.26691941){\color[rgb]{0,0,0}\makebox(0,0)[lt]{\lineheight{1.25}\smash{\begin{tabular}[t]{l}$\uchc$\end{tabular}}}}%
    \put(0,0){\includegraphics[width=\unitlength,page=2]{tiles.pdf}}%
    \put(0.92726152,0.25983233){\color[rgb]{0,0,0}\makebox(0,0)[lt]{\lineheight{1.25}\smash{\begin{tabular}[t]{l}$X$\end{tabular}}}}%
    \put(0,0){\includegraphics[width=\unitlength,page=3]{tiles.pdf}}%
    \put(0.56549293,0.25953355){\color[rgb]{0,0,0}\makebox(0,0)[lt]{\lineheight{1.25}\smash{\begin{tabular}[t]{l}$\hc$\end{tabular}}}}%
    \put(0,0){\includegraphics[width=\unitlength,page=4]{tiles.pdf}}%
  \end{picture}%
\endgroup%

    \caption{Tiles in $\uchc, \hc$ and $X$.}
    \label{fig:tiles}
\end{figure}

We define a \textit{tile} of $\hc$ to be a subspace of the form $\cdX^{-1}(C)$, for $C$ a top--dimensional cube of $X$.
Recall from Remark~\ref{rem:CD hyp complex metric} that each tile of $\hc$ is isometric to $\hq$.
In complete analogy, we define a \textit{tile} in $\ichc$ and in $\uchc$ to be a connected component of the lift of a tile from $\hc$ via the covering maps $\pifromichc$ and $\piuchc=\pifromichc \circ \pitoichc$ respectively. We refer to this decomposition into tiles as the \textit{tiling} of each of these spaces (see Figure~\ref{fig:tiles}). 
Note that, since the complex $X$ is assumed to be admissible, each point of $X$ is either contained in the interior of a tile, or in the intersection of at least two tiles. 
Moreover the folding map $f$ of $X$ induces an analogous map on $\hc$ and its covering spaces, as established in the next lemma. 

\begin{lemma}\label{lem:foldingmap_hypcomplex}
The map $\foldutocell = \foldic \circ \pitoichc : \uchc \to \ichc \to \uchq$ restricts to an isometry between each tile of $\uchc $ and $\uchq$.
\end{lemma}

\proof
Recall that $\ichc$ is defined via a pullback construction, in the sense of \S \ref{sec:hyperbolization template}.
Therefore, by \eqref{item:hypcell homeo} in Lemma~\ref{lem:cell_is_hyperbolizing_cell}, the map $\foldic:\ichc \to \uchq$ restricts to a homeomorphism between each tile of $\ichc$ and $\uchq$. Since the metric on $\ichc$ is lifted from $\hc$ via $\pifromichc$, and $\foldc$ restricts to an isometry between each tile  of $\hc$ and $\hq$ (see  Remark~\ref{rem:CD hyp complex metric}), the map $\foldic$ actually gives an isometry between a tile of $\ichc$ and $\uchq$.
Since the tiles of $\ichc$ are simply connected, they lift isometrically to tiles of $\uchc$ via $\pitoichc$. In particular, $\pitoichc$ maps a tile of $\uchc$ isometrically onto a tile of $\ichc$. Therefore the map $\foldutocell = \foldic \circ \pitoichc$ maps a tile of $\uchc$ isometrically onto $\uchq$, just by composition.
\endproof

The map $\foldutocell$ from Lemma~\ref{lem:foldingmap_hypcomplex} will be called the \textit{folding map} of $\uchc$.
The composition of the folding map $\foldutocell$ with any isometric embedding $\varphi: \uchq\to C$ onto a  $\Gamma$--cell $C\subseteq \hh^n$ will be called a \textit{developing map} for $\uchc$.

\begin{remark}\label{rem:developing}
The restriction of a developing map to a tile is an isometric embedding of a tile into $\hh^n$ as a $\Gamma$--cell. Moreover if $T_1,T_2$ are two tiles of $\uchc$ meeting along a codimension--$1$ subspace $Z$, and $\varphi_1:T_1\to C_1\subseteq \hh^n$ is an isometric embedding onto a $\Gamma$--cell that maps $Z$ into some hyperplane $H$, then post--composing $\varphi_1$ with the reflection across $H$ provides an isometric embedding $\varphi_2$ of $T_2$ as a $\Gamma$--cell $C_2$ adjacent to $C_1$ along $H$. The two embeddings can be glued together to give an isometric embedding of $T_1\cup T_2$ onto the union of two  $\Gamma$--cells $C_1\cup C_2$ adjacent along $H$. This can be ``analytically continued'' by sequentially extending over adjacent tiles, to obtain a globally defined map $\uchc \to \hh^n$. However, this does not result in a global isometric embedding $\uchc \to \hh^n$ in general. This is due to the fact that links in $X$ can be very large, which gives rise to overlaps and singularities.
\end{remark}


\subsection{Mirrors: convexity}\label{subsec:mirrors convexity}
In this section we exploit foldability to define a collection of convex subcomplexes of $X$, and induce corresponding subspaces in $\hc$ and $\uchc$.
Let $Y$ be a foldable cubical complex of dimension $n$ (in the following we will consider $Y=X$ and $Y=\widetilde X$ depending on the situation).
If $f:Y\to \square^n$ is a fixed folding and $F\subseteq \square^n$ is a codimension-1 face, then we define a \textit{mirror} in $Y$ to be a connected component of $f^{-1}(F)$.


\begin{proposition}\label{prop:CAT(0)_mirrors_convex}
Let $Y$ be an admissible cubical complex. Then each mirror is a locally convex and geodesically complete subcomplex of $Y$. In particular, if $Y$ is $\cat 0$, then each mirror is convex.
\end{proposition}
\proof
For the first statement see \cite[Proposition 2.3]{XIE04} (and references therein such as \cite[Lemma 3.2(4)]{BS99}).  
In the $\cat 0$ case, local convexity implies global convexity (see for instance \cite[Theorem 1.6,1.10]{BUW12}, or \cite[Theorem 1.1]{RC16}).
\endproof

\begin{figure}[h]
\centering
\def\svgwidth{\columnwidth}
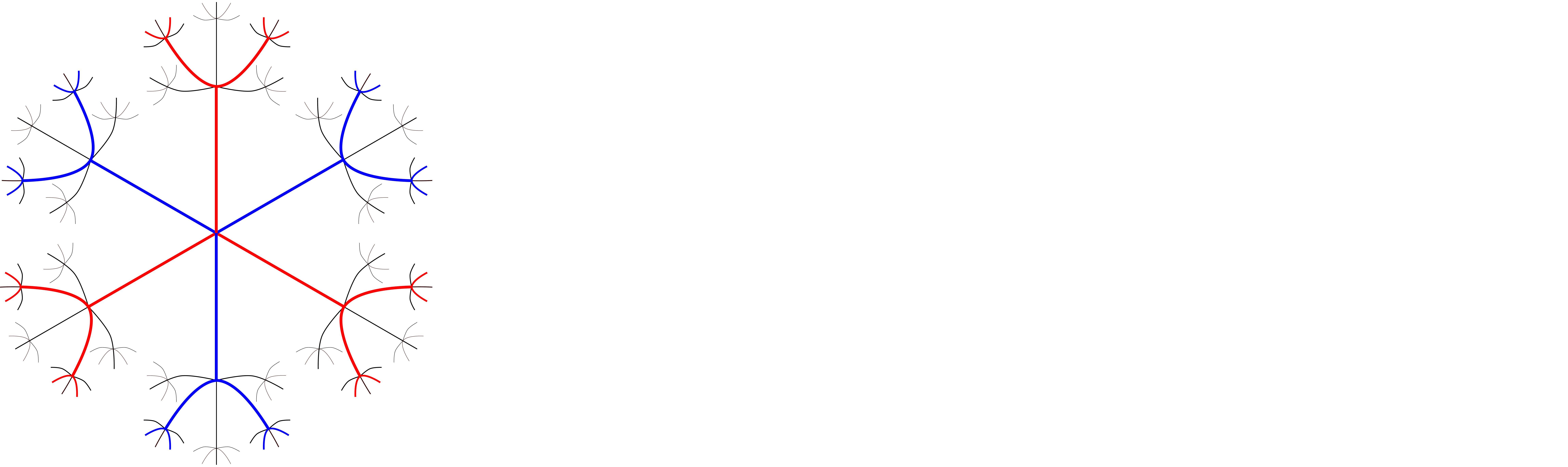
    \caption{Right to left: mirrors $M_1$, $M_2$ in $X$,  their preimages $\widehat M_1 =\cdX^{-1}(M_1), \widehat M_2 =\cdX^{-1}(M_2)$  in $\hc$, and the lifts $\widetilde M_1$, $\widetilde M_2$ to $\uchc$.}
    \label{fig:mirrors}
\end{figure}

\begin{figure}[h]
    \centering
    \begin{tikzpicture}
\node at (-1,1) (A) {$\hc$};
\node at (-1,-1) (B) {$X$};
\node at (1,-1) (C) {$\square^n$};
\node at (1,1) (D) {$\hq$};

\node at (-3,1) (H) {$\uchc$};

\draw [->] (A) edge (B) (B) edge (C) (A) edge (D) (D) edge (C) (H) edge (A) (H) edge (B);
\draw [out=-90,in=-135,->] (H) edge (C) ;

\node at (0,1.25) {$\foldc$};
\node at (0,-1.25) {$\fold$};
\node at (-2.5,-1) {$\foldutosquare$};
\node at (1.25,0) {$\cd$};
\node at (-.5,0) {$\cdX$};
\node at (-1.5,0) {$\cdXu$};
\node at (-2,1.25) {$\piuchc$};
\end{tikzpicture}
 
    \caption{The hyperbolized complex $\hc$ and the maps used to define mirrors.}
    \label{fig:CD_hypcomplexdiagram_mirrormaps}
\end{figure}
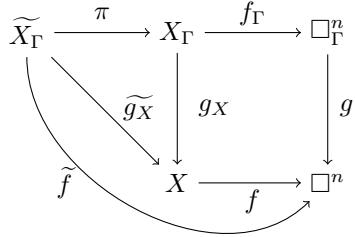

We now define a \textit{mirror} in $\uchc$ to be a connected components of $\foldutosquare^{-1}(F)$, where $F$ is a codimension-1 face of  $\square^n$ and $\foldutosquare$ is the map given by the composition
$\foldutosquare = \fold \circ \cdX \circ \piuchc : \uchc \to \hc \to X \to \square^n$ (see Figure~\ref{fig:mirrors} and
Figure~\ref{fig:CD_hypcomplexdiagram_mirrormaps}).
Equivalently, we could define it as a connected component of the full preimage of a mirror of $X$ via the map $\cdXu=\cdX \circ \piuchc$, but we find it convenient to use this definition. 
We will say that $M$ \textit{folds} to $F$, and we will denote by $\mirrors$ the collection of all mirrors in $\uchc$. 
Mirrors in $\hc$ are defined in the analogous way using the map $\fold \circ \cdX$.

\begin{proposition}\label{prop:mirrors convex}
Let $X$ be an admissible cubical complex. Then each mirror of $\uchc$ is a closed connected convex subspace of $\uchc$.
\end{proposition}
\proof
Let $M$ be a mirror of $\uchc$, and let $F\subseteq \square^n$ be the codimension-1 face to which it folds. By definition $M$ is connected and closed.
To prove convexity we argue as follows. Let $Z = \cdX(\piuchc(M)) \subseteq X$, and notice that $Z$ is a mirror of $X$ that folds to $F$. 
By Proposition~\ref{prop:CAT(0)_mirrors_convex} we know that $Z$ is locally convex in $X$. 
By \eqref{item:CD subcomplexes} in Proposition~\ref{prop:CD complex}, we also know that $\cdX^{-1}(Z)$ is locally convex in $\hc$, and therefore $M\subseteq \uchc$ is locally convex too.
By \eqref{item:CD npc} in Proposition~\ref{prop:CD complex} we also know that $\hc$ is locally $\cat{-1}$. In particular $M$ is a closed and locally convex subspace in the $\cat 0$ space $\uchc$. Therefore it is convex (again by \cite[Theorem 1.6,1.10]{BUW12}, or \cite[Theorem 1.1]{RC16}).
\endproof


\subsection{Stratification of \texorpdfstring{$\uchc$}{the universal cover of the hyperbolized complex}}\label{subsec:stratification}
In this section we use the collection $\mirrors$ of mirrors, introduced in \S\ref{subsec:mirrors convexity}, to define a stratification of $\uchc$.
The \textit{open} $k$--\textit{stratum} $\Sigma^k$ of $\uchc$ is the subspace consisting of points that fold into the interior of a $k$-face of $\square^n$ via the map
$\foldutosquare = \fold \circ \cdX \circ \piuchc : \uchc \to \hc \to X \to \square^n$, or equivalently to the interior of a $k$--cube of $X$ via the map $\cdXu=\piuchc \circ \cdX:\uchc \to \hc \to X$
(see Figure~\ref{fig:CD_hypcomplexdiagram_mirrormaps}).
Notice that $\Sigma^k$ is a locally closed subspace.
An \textit{open $k$--cell} is a connected component of $\Sigma^k$.
A  \textit{$k$--cell} is the closure of an open $k$--cell.
We say that a cell $\sigma$ \textit{folds} to the face $F=\foldutosquare(\sigma)\subseteq\square^n$ and to the cube $C=\cdXu(\sigma)\subseteq X$.
The integer $k$ will be referred to as the \textit{dimension} of a $k$--cell.
An $(n-k)$--cell is a proper subset of the intersection of $k$ mirrors. In particular $0$--cells are points, and $n$-cells are tiles (as defined in \S\ref{sec:universal_cover}). 
We call $0$--cells \textit{vertices}, and $1$--cells \textit{edges} of the stratification.

\begin{remark}[Cellular structure]\label{rem:weird cell structure}
We explicitly observe that this choice of strata does not define a stratified space structure on $\uchc$ in the sense of \cite[Definition II.12.1]{BH99}.
Moreover, the decomposition of $\uchc$ into cells does not turn it into a genuine cell complex, as defined in \S\ref{sec:cell complexes}.
Indeed, while an open $k$-cell is homeomorphic to an open disk of dimension $k$, a $k$-cell is not homeomorphic to a closed disk of dimension $k$ as soon as $k\geq 2$. Its boundary in $\uchc$ consists of an infinite union of lower--dimensional cells, so it is neither connected nor compact. For instance, an $n$--cell (i.e.\ a tile) is isometric to a $\Gamma$--cell (see Figure~\ref{fig:uchq_is_cell}).
\end{remark}

Nevertheless, we can still recover a lot of the classical behavior and tools, by observing that cells are convex and that the link of cells and points can be defined in analogy to the classical case (see \S \ref{subsec:def links}). We gather here preliminary results about this, that will be useful in the following. 
For the sake of clarity, we emphasize that in our terminology cells are closed.

\begin{lemma}\label{lem:strat_cells_convex}
Let $\sigma \subseteq \uchc$ be a cell. Then $\sigma$ is convex.
\end{lemma}
\proof
Let $C=\cdXu (\sigma)\subseteq X$ be the cube of $X$ to which $\sigma$ folds.
By \eqref{item:CD subcomplexes} in Proposition~\ref{prop:CD complex}, we know $\cdX^{-1}(C)$ is locally convex in $\hc$. Since $\sigma$ is by definition a connected component of $\piuchc^{-1}(\cdX^{-1}(C))$ and $\piuchc$ is a local isometry, we can conclude that it is a locally convex subspace of $\uchc$.
Arguing similarly to previous proofs of convexity, we can conclude that $\sigma$ is convex, because it is closed and locally convex in the $\cat 0$ space $\uchc$ (see \eqref{item:CD npc} in Proposition~\ref{prop:CD complex} and \cite[Theorem 1.6,1.10]{BUW12}, or \cite[Theorem 1.1]{RC16}).
\endproof

We now proceed to the study of links.
Consider the universal covering map $\piuchc : \uchc \to \hc$. By Proposition~\ref{prop:CD complex}, $\hc$ is a piecewise hyperbolic cell complex, so the link of points and cells in $\hc$ is well--defined (see \S \ref{subsec:def links}). Since $\piuchc$ is a local isometry, we can just identify the \textit{link} of points and cells in $\uchc$ with the links of the corresponding points and cells in $\hc$.

\begin{lemma}\label{lem:strat_links_cells}
Let $\sigma \subseteq \uchc$ be a cell. 
\begin{enumerate}
    \item \label{item:strat_links_cells_cube} Let $C=\cdXu (\sigma)\subseteq X$ be the cube  to which it folds. 
    Then $\cdXu$ induces an isomorphism between $\lk{\sigma, \uchc}$ and $\lk {C,X}$.

    \item \label{item:strat_links_cells_face}  Let $\sigma$ be contained in another cell $\tau$. Let $F=\foldutosquare(\sigma), E=\foldutosquare(\tau)\subseteq \square^n$ be the faces to which they fold. 
    Then $\foldutosquare$ induces an isomorphism between $\lk{\sigma, \tau}$ and $\lk {F,E}$. 
    
    \item \label{item:strat_links_cells_simplicial}  Let $\sigma$ be a $k$--cell. Then $\lk{\sigma, \uchc}$ is a piecewise spherical simplicial complex with vertices given by the $(k+1)$--cells  containing $\sigma$, and in which $m+1$ vertices span an $m$--simplex if and only if the corresponding $(k+1)$--cells are contained in a $(k+m+1)$--cell.

\end{enumerate}
\end{lemma}
\proof
The map  $\cdXu:\uchc \to X$  is the composition of the map $\piuchc:\uchc \to \hc$, which  preserves links because it is a covering map, and the map $\cdX:\hc \to X$, which preserves links thanks to \eqref{item:CD faces and links} in Proposition~\ref{prop:CD complex}. This proves \eqref{item:strat_links_cells_cube}.

To prove \eqref{item:strat_links_cells_face} we argue similarly. The map  $\foldutosquare:\uchc \to \square^n$  is the composition of the map $\cdXu:\uchc \to X$, which preserves links by \eqref{item:strat_links_cells_cube}, and the folding map $\fold:X\to \square^n$. By definition of folding, $\fold$ is a combinatorial isomorphism on each cube of $X$. If $B$ is the cube to which $\tau$ folds, the folding  induces an isomorphism between $\lk{C,B}$ and $\lk{F,E}$.

Finally, \eqref{item:strat_links_cells_simplicial} follows from \eqref{item:strat_links_cells_cube}, the fact that  $\cdXu:\uchc \to X$ maps cells of $\uchc$ to cubes of $X$ preserving inclusion relations, and the fact that the link of a cell in a cubical complex carries a piecewise spherical simplicial structure as described in the statement.
\endproof

\begin{lemma}\label{lem:strat_cell_intersection}
Let $\sigma_1,\sigma_2 \subseteq \uchc$ be two cells. Then either $\sigma_1\cap \sigma_2$ is empty or 
it is a cell. 

\end{lemma}
\proof

Let  $\sigma_1\cap \sigma_2$ be non empty. If it contains either a single vertex, or a single edge, then we are done. 
So let us assume that it contains at least two edges.
Also note that, since cells are convex by Lemma~\ref{lem:strat_cells_convex}, the intersection  $\sigma_1 \cap \sigma_2$ is convex. Therefore, if there are several edges then they cannot all be disjoint.

Let $v\in \sigma_1\cap \sigma_2$ be a vertex, and let $e,e'$ be two edges of $\sigma_1\cap \sigma_2$ meeting at $v$.
Note that by Lemma~\ref{lem:strat_links_cells} links in $\uchc$ are isomorphic to the corresponding links in $X$.
In particular, $e$ and $e'$ are edges of the cell $\sigma_1$ meeting at a vertex, so there is a $2$--face $\tau_1\subseteq \sigma_1$ containing both $e$ and $e'$. Analogously, we get a $2$--face $\tau_2\subseteq \sigma_2$ with the same property.
Since these links are simplicial, necessarily we have $\tau_1=\tau_2$, otherwise we would see a bigon in the link of $v$.
In particular, $\tau_1=\tau_2 \subseteq \sigma_1\cap\sigma_2$.
This shows that any two edges of $\sigma_1\cap \sigma_2$ meeting at $v$  are adjacent in $\lk{v,\sigma_1\cap \sigma_2}$.
By Lemma~\ref{lem:strat_links_cells} we know that $\lk{v,\uchc}\cong \lk{\cdXu(v),X}$, and this is a flag simplicial complex because $X$ is non--positively curved (see Lemma~\ref{lem:gromov link condition}).
The same holds for  $\lk{v,\sigma_1\cap \sigma_2}$ because $\sigma_1\cap \sigma_2$ is convex in $\uchc$.
In particular, all the edges of $\sigma_1\cap \sigma_2$ that contain $v$ are actually contained in a unique cell of minimal dimension in $\sigma_1\cap \sigma_2$; we denote this cell by $\sigma_v$.
Now, if $v,w$ are adjacent vertices of $\sigma_1\cap \sigma_2$, then by uniqueness we have $\sigma_v=\sigma_w$.
Finally, by connectedness of $\sigma_1\cap \sigma_2$, it follows that all vertices of $\sigma_1\cap \sigma_2$ are contained in a single cell.
\endproof

\begin{lemma}\label{lem:min_max_cells}
Let $\{\sigma_j \ | \ j\in J\}$ be a collection of cells of $\uchc$. Then the following statements hold.
\begin{enumerate}
    \item \label{item:lower_cell} If $\sigma=\bigcap\limits_{j\in J}  \sigma_j$ is not empty, then $\sigma$ is the unique cell of maximal dimension contained in  $\sigma_j$ for all $j\in J$.
    
    \item \label{item:upper_cell} If $\bigcup\limits_{j\in J} \sigma_j$ is contained in a single cell, then there exists a unique cell $\sigma$ of minimal dimension containing $\sigma_j$ for all $j\in J$.
\end{enumerate}
\end{lemma}
We refer to the cell in \eqref{item:lower_cell} (respectively \eqref{item:upper_cell}) of Lemma~\ref{lem:min_max_cells} as the \textit{lower cell} (respectively \textit{upper cell}) of the collection $\{\sigma_j \ | \ j\in J\}$.
\proof
Since $X$ is finite--dimensional and locally compact, if $J$ is infinite, then $\sigma=\bigcap\limits_{j\in J}  \sigma_j$ is empty. So let us assume that $J$ is finite.
By Lemma~\ref{lem:strat_cell_intersection} the intersection of finitely many cells is either empty or made of a single cell. This proves \eqref{item:lower_cell}.
To prove \eqref{item:upper_cell}, assume by contradiction that there are two different cells of minimal dimension $\sigma,\sigma'$ containing each $\sigma_j$. Then $\sigma\cap \sigma'$ is a proper union of cells, against Lemma~\ref{lem:strat_cell_intersection}.
\endproof

\begin{lemma}\label{lem:strat_cells_in_tile_fold_disjoint}
Let $\tau \subseteq \uchc$ be a cell. Let $\sigma_1,\sigma_2\subseteq \tau$ be cells of lower dimension, and let $F_1,F_2\subseteq \square^n$ be the faces to which they fold.
If $F_1=F_2$, then $\sigma_1, \sigma _2$ are either disjoint or equal.
\end{lemma}
\proof
Assume that $\sigma_1, \sigma _2$ are not disjoint, and let $v\in \sigma_1\cap \sigma_2$ be a vertex.
Let $E\subseteq \square^n$ be the face to which $\tau$ folds.
We have that $\foldutosquare(\sigma_1)=F_1=F_2=\foldutosquare(\sigma_2)$, and by \eqref{item:strat_links_cells_face} in Lemma~\ref{lem:strat_links_cells} (with $\sigma=v$) the map $\foldutosquare$ induces an isomorphism between $\lk{v,\tau}$ and  $\lk{\foldutosquare(v)),E}$.
Therefore, $\sigma_1=\sigma_2$.
\endproof

\begin{lemma}\label{lem:mirrors_in_tile}
Let $M$ be a mirror and let $\tau$ be a $k$--cell of $\uchc$ not entirely contained in $M$. If $M \cap \tau\neq \varnothing$, then $M\cap \tau$ is a ($k-1$)--cell.
\end{lemma}
\proof
First we show that $M\cap \tau$ is a union of ($k-1$)--cells. Then we show that the union actually consists of a single cell.

Let $\foldutosquare = \fold \circ \cdX \circ \piuchc : \uchc \to \hc \to X \to \square^n$ be the map that folds $\uchc$ to $\square^n$, and let $F\subseteq \square^n$ be the codimension--1 face to which $M$ folds (i.e.\ $F=\foldutosquare(M)$).
Similarly, let $E\subseteq \square^n$ be the $k$--face to which $\tau$ folds (i.e.\ $E=\foldutosquare(\tau)$).
Since $\tau\not\subseteq M$ we have $E\not\subseteq F$, and therefore $E\cap F$ is a $(k-1)$--face of $\square^n$.
Let $p\in M\cap \tau$ be an arbitrary point. 
By Lemma~\ref{lem:strat_cells_in_tile_fold_disjoint}, among the ($k-1$)--cells of $\tau$ that contain $p$, there is exactly one that folds to $E\cap F$; denote it by $\sigma_p$. 
Clearly $\sigma_p\subseteq \tau$. Moreover, since $\sigma_p$ and $M$ are non--disjoint, both fold into $F$, and $M$ is a mirror,  we also have $\sigma_p \subseteq M$. Therefore  we have $M\cap \tau = \bigcup_{p\in M\cap \tau} \sigma_p$, i.e.\ $M\cap \tau$ is a union of ($k-1$)--cells that fold to $E\cap F$.

To see that $M\cap \tau$ actually consist of only one cell, assume by contradiction that $M \cap \tau$ contains two distinct  ($k-1$)--cells $\sigma_1, \sigma_2$. 
Let $p_i\in \sigma_i$ and let $\gamma=[p_1,p_2]$ be the unique geodesic between them. Since $M$ and $\tau$ are both convex (by Proposition~\ref{prop:mirrors convex} and Lemma~\ref{lem:strat_cells_convex} respectively), we have that $\gamma\subset \tau \cap M$. Hence we find a path of cells in the boundary of $\tau$ that all fold to $E\cap F$.
But this is absurd because different boundary cells of $\tau$ folding to the same face of $\square^n$ are necessarily disjoint, again by Lemma~\ref{lem:strat_cells_in_tile_fold_disjoint}.
\endproof


\subsection{Graph of spaces decomposition for \texorpdfstring{$\uchc$}{the universal cover of the hyperbolized complex}}\label{subsec:graph of spaces}
Our goal in \S\ref{subsec:separation} will be to prove that mirrors in $\uchc$ enjoy a strong separation property. Our strategy will be to exploit a certain  graph of spaces decomposition for $\uchc$ (in the sense of \cite{SW79}), which we introduce in this section, using the foldability of $X$ (see \cite{BS99,XIE04} for analogous constructions).

Recall from \S\ref{subsec:mirrors convexity} that $\uchc$ is equipped with a collection  $\mirrors$ of closed convex subspaces called mirrors.
For each $i=1,\dots,n$, let $\mirrors_i$ be the collection of mirrors of $\uchc$ that fold to one of the two parallel $i^{th}$ faces of $\square^n=[0,1]^n$, i.e.\ $\{ x_i=0\} $ and $\{ x_i=1\} $. 
Notice that by construction any two elements of $\mirrors_i$ are disjoint, and even have disjoint $\varepsilon$--neighborhoods for $\varepsilon$ sufficiently small (because $\Gamma$ is cocompact).



    

Let $\components_i$ be the collection of connected components of $\uchc \setminus \cup_{M\in \mirrors_i} M$.
For each mirror $M\in \mirrors_i$ and for each component $C\in \components_i$, consider the following \textit{equidistant space}, obtained by pushing the mirror $M$ into the component $C$ (see Figure~\ref{fig:edgespaces}).
\[
\edgespace=\{ x\in C \ | \ d(x,M)=\varepsilon\}.
\]

\begin{figure}[h]
\centering
\def\svgwidth{\columnwidth}
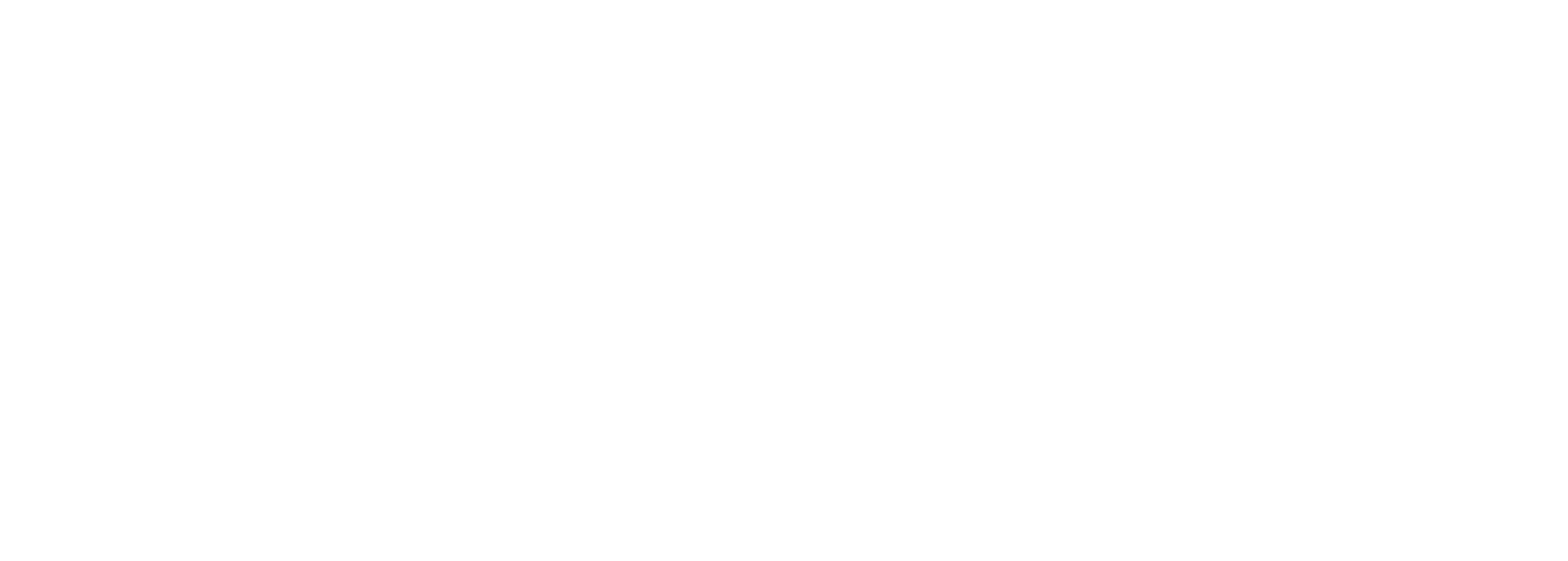
    \caption{Some examples of equidistant spaces in dimension $2$ and $3$. Left: an equidistant space relative to a mirror $M$ in the vicinity of the intersection with two other mirrors. Here $\dim X=3$, and all mirrors are locally Euclidean. Right: three equidistant spaces relative to the same mirror $M$ but three different complementary components, in the vicinity of the intersection with another mirror. Here $\dim X=2$, and the mirrors branch, i.e.\ are not locally Euclidean.}
    \label{fig:edgespaces}
\end{figure}

Notice that while we know $M$ is convex by Proposition~\ref{prop:mirrors convex}, it is not clear whether $C$ is convex. A priori, $C$ could meet $M$ on more than one side, i.e.\ the closure of $C$ could contain a piece of $M$ in its interior. We will see this is not the case by considering a suitable graph of spaces decomposition of $\uchc$. Our first step is to show that $\edgespace$ is simply connected; in the process, we actually show it is a $\cat k$--space for some $k\in (-1,0)$. The idea for this can be summarized as follows: inside each tile, $\edgespace$ looks like an equidistant hypersurface from a hyperplane in $\hh^n$, and this is a non--positively curved hypersurface in $\hh^n$ (see Figure~\ref{fig:equidistant}). 
Then contribution from different tiles come together in a way that does not introduce any positive curvature along mirrors. We start from a preliminary lemma from classical hyperbolic geometry. Recall that a hyperplane in $\hh^n$ is a totally geodesic copy of $\hh^{n-1}$.

\begin{figure}[h]
\centering
\def\svgwidth{.7\columnwidth}
\begingroup%
  \makeatletter%
  \providecommand\color[2][]{%
    \errmessage{(Inkscape) Color is used for the text in Inkscape, but the package 'color.sty' is not loaded}%
    \renewcommand\color[2][]{}%
  }%
  \providecommand\transparent[1]{%
    \errmessage{(Inkscape) Transparency is used (non-zero) for the text in Inkscape, but the package 'transparent.sty' is not loaded}%
    \renewcommand\transparent[1]{}%
  }%
  \providecommand\rotatebox[2]{#2}%
  \newcommand*\fsize{\dimexpr\f@size pt\relax}%
  \newcommand*\lineheight[1]{\fontsize{\fsize}{#1\fsize}\selectfont}%
  \ifx\svgwidth\undefined%
    \setlength{\unitlength}{1502.63625738bp}%
    \ifx\svgscale\undefined%
      \relax%
    \else%
      \setlength{\unitlength}{\unitlength * \real{\svgscale}}%
    \fi%
  \else%
    \setlength{\unitlength}{\svgwidth}%
  \fi%
  \global\let\svgwidth\undefined%
  \global\let\svgscale\undefined%
  \makeatother%
  \begin{picture}(1,0.65978605)%
    \lineheight{1}%
    \setlength\tabcolsep{0pt}%
    \put(0,0){\includegraphics[width=\unitlength,page=1]{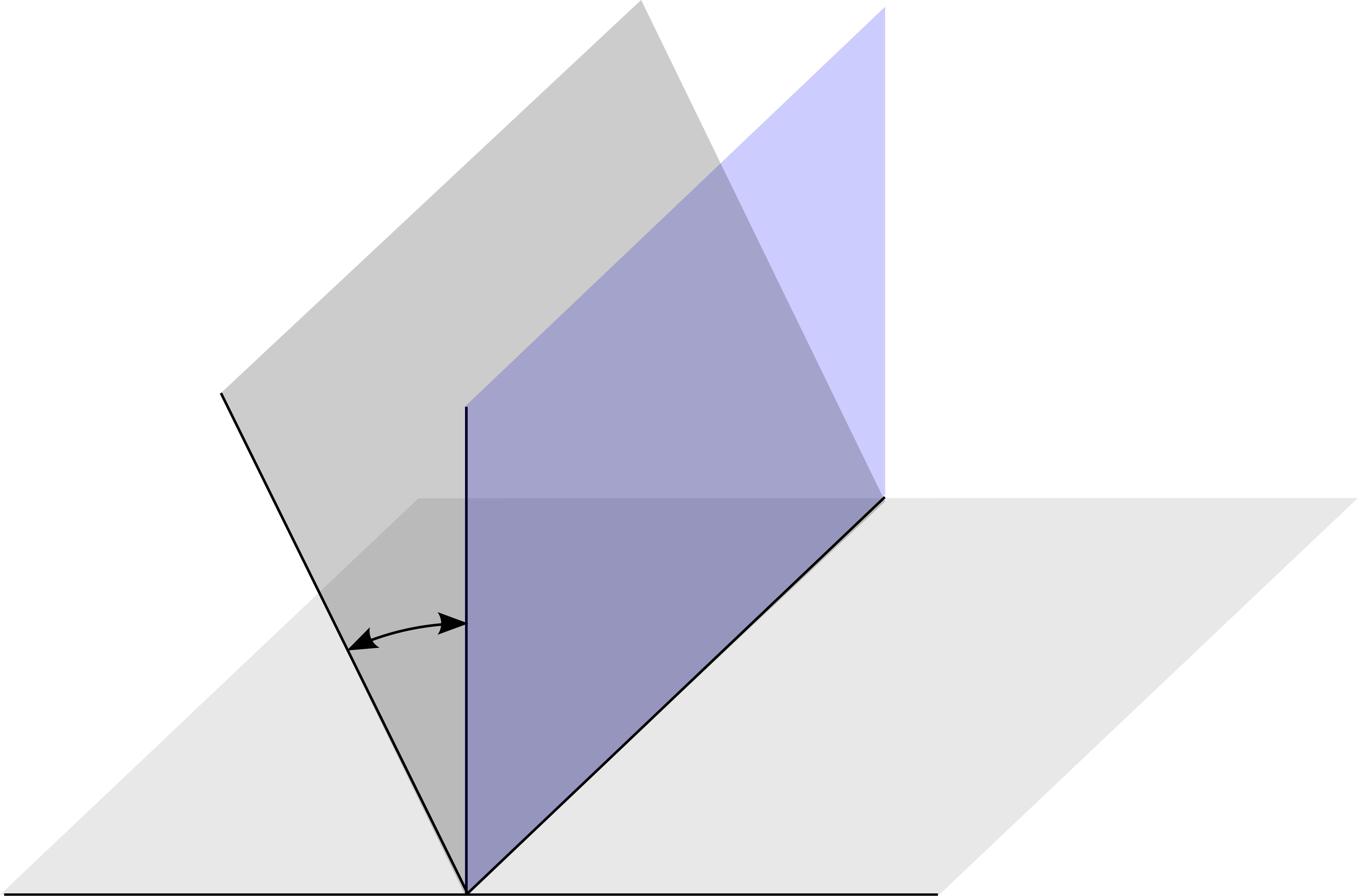}}%
    \put(0.60457986,0.57938587){\color[rgb]{0,0,1}\makebox(0,0)[lt]{\lineheight{1.25}\smash{\begin{tabular}[t]{l}$V$\end{tabular}}}}%
    \put(0.43296904,0.58010498){\color[rgb]{0,0,0}\makebox(0,0)[lt]{\lineheight{1.25}\smash{\begin{tabular}[t]{l}$S_V^\varepsilon$\end{tabular}}}}%
    \put(0.27251454,0.22055787){\color[rgb]{0,0,0}\makebox(0,0)[lt]{\lineheight{1.25}\smash{\begin{tabular}[t]{l}$\varepsilon$\end{tabular}}}}%
  \end{picture}%
\endgroup%

    \caption{Equidistant surface from a hyperplane.}
    \label{fig:equidistant}
\end{figure}

\begin{lemma}\label{lem:equidistant_surfaces}
Let $V\subseteq \hh^n$ be a hyperplane, and let $\pi_V:\hh^n\to V$ be the nearest point projection to $V$. Let $\varepsilon>0$ and $S^\varepsilon_V=\{ x\in \hh^n \ | \ d(x,V)=\varepsilon\}$. Then the following hold.
\begin{enumerate}
    \item $S^\varepsilon_V$ is a smooth $(n-1)$--dimensional submanifold of $\hh^n$.
    
    \item For each $p\in S^\varepsilon_V$, the geodesic $[p,\pi_V(p)]$ is orthogonal to $V$ and $S^\varepsilon_V$.
    
    \item For every other hyperplane $W$, if $V\cap W \neq \varnothing$, then $S^\varepsilon_V \cap W \neq \varnothing$. 
    
    \item \label{item:eqsurf_orthogonal} For every  hyperplane $W$,  $W$ is orthogonal to $S^\varepsilon_V$ if and only if $W$ is orthogonal to $V$.
    
    \item \label{item:eqsurf_conformal} $\pi_V:S^\varepsilon_V\to V$ is a $\cosh^2 (\varepsilon)$--conformal diffeomorphism.
    
    \item \label{item:eqsurf_curvature} The induced metric on $S^\varepsilon_V$ has constant sectional curvature $\frac{-1}{\cosh^2 (\varepsilon)}$.
\end{enumerate}
\end{lemma}
\proof
The first five statements can be proved by explicit computations in the upper half--space model of $\hh^n$, normalizing so that $V$ is a vertical hyperplane (see Figure~\ref{fig:equidistant}).
The computation for dimension $n=3$ is carried out in detail in \cite[IV.5, page 58]{F89}, and readily generalizes to higher dimensions. 
Finally, \eqref{item:eqsurf_curvature} follows from \eqref{item:eqsurf_conformal} and the general formula for the behavior of the sectional curvatures under rescaling.
\endproof


For the next lemma, recall from \S\ref{sec:universal_cover} that tiles are closed by definition, and that a developing map is an isometric embedding of a tile into $\hh^n$ as a $\Gamma$--cell.

\begin{lemma}\label{lem:edge_spaces_combinatorial}
Let $M\in \mirrors_i$ and $C\in \components_i$.
Then for $\varepsilon >0$ small enough the following hold.
\begin{enumerate}

    \item \label{item:edge_space_meets_mirrors} For every mirror $N\in \mirrors$, if $\edgespace \cap N \neq \varnothing$ then $M\cap N \neq \varnothing$ and $C\cap N \neq \varnothing$. 
    
    \item \label{item:edge_space_meets_tiles}  For every tile $\tau$, if $\edgespace \cap \tau \neq \varnothing$ then $M \cap \tau  \neq \varnothing$ and $C \cap \tau \neq \varnothing$. 
    
    \item \label{item:edge_space_in_tile}  For every tile $\tau$ such that $\edgespace \cap \tau \neq \varnothing$, and any developing map $ \varphi:\tau\to \hh^n$,  $\varphi$ induces an isometry between $\edgespace\cap \tau$ and $S^\varepsilon_{V}\cap \varphi(\tau)$, where $V$ is the hyperplane containing $\varphi(M\cap \tau)$.      

    \item \label{item:edge_space_orthogonal_mirrors} For every mirror $N\in \mirrors$, if $\edgespace \cap N \neq \varnothing$ then $\edgespace$ is orthogonal to $N$.
    
\end{enumerate}
\end{lemma}
\proof
To prove \eqref{item:edge_space_meets_mirrors} note that if $\edgespace \cap N \neq \varnothing$, then in particular $C\cap N \neq \varnothing$. Since $\Gamma$ is cocompact, there is a uniform lower bound $D>0$ on the distance between disjoint mirrors. But $\edgespace \cap N \neq \varnothing$ means that $N$ comes $\varepsilon$ close to $M$. By choosing $\varepsilon < D$ we can force $N$ to actually intersect $M$.

The proof of \eqref{item:edge_space_meets_tiles} is analogous to that of \eqref{item:edge_space_meets_mirrors}. 
Suppose $\edgespace \cap \tau \neq \varnothing$. Then clearly $C\cap \tau \neq \varnothing$. Moreover, a point in $\edgespace \cap \tau$ witnesses that $d(M,\tau)<\varepsilon$, and by choosing $\varepsilon$ small enough we can ensure that this forces an intersection, again by cocompactness of $\Gamma$. 

Now we consider \eqref{item:edge_space_in_tile}.
Suppose that $\edgespace \cap \tau \neq \varnothing$.
Then by \eqref{item:edge_space_meets_tiles} we know that $M \cap \tau  \neq \varnothing$ and $C \cap \tau \neq \varnothing$. In particular $M$ appears as an ($n-1$)--cell in the boundary of $\tau$ thanks to Lemma~\ref{lem:mirrors_in_tile}. If we pick a developing map $\varphi$ for $\tau$, then $\varphi(\tau)$ is a $\Gamma$--cell, and $\varphi(M)$ is some hyperplane $V$ on its boundary (see Lemma~\ref{lem:foldingmap_hypcomplex} and Remark~\ref{rem:developing}).
Then the statement follows from the fact that $\varphi$ is an isometric embedding of $\tau$ into $\hh^n$.

Finally, to prove \eqref{item:edge_space_orthogonal_mirrors}, suppose that $\edgespace \cap N \neq \varnothing$. Then by \eqref{item:edge_space_meets_mirrors} we know that $N\cap M\neq \varnothing$. In particular by construction $N$ is orthogonal to $M$.
Then the statement follows from \eqref{item:edge_space_in_tile}, together with \eqref{item:eqsurf_orthogonal} in Lemma~\ref{lem:equidistant_surfaces}.
\endproof


Next, our goal is to prove that equidistant spaces are negatively curved. In order to do this, we will study the geometry of links of points in $\uchc$, along various subspaces (we refer the reader to \S \ref{subsec:stratification} for definitions). 
Recall that the link  of a point in $\uchc$ is identified to the link of its projection to $\hc$. 

\begin{remark}\label{rem:weird cell structure mirrors and edgespaces}
All the subspaces of $\uchc$ considered here (such as a mirror $M$, and the induced space $\edgespace$) carry a natural locally finite cellular structure induced by that of $\uchc$. 
Even if they are not genuine cell complexes (as in Remark~\ref{rem:weird cell structure}), their projections to $\hc$ are, and links can be defined in analogy to the classical case.
\end{remark}


For a mirror $M$ we denote by $\pi_M:\uchc \to M$ the nearest point projection.
This is well-defined because $\uchc$ is $\cat 0$ and $M$ is convex by Proposition~\ref{prop:mirrors convex}.

\begin{lemma}\label{lem:edge_spaces_link_projection}   
Let $M\in \mirrors_i$, $C\in \components_i$, $p\in \edgespace$.
Then for $\varepsilon >0$ small enough the following holds.
Let $\tau_1,\dots, \tau_m$ be the collection of tiles containing $p$, and let $T=\tau_1\cup\dots\cup \tau_m$.
Then the following hold.
\begin{enumerate}
    \item \label{item:link_mirrortile_cat1} $\lk{\pi_M(p),M\cap T}$ is $\cat 1$.
    \item \label{item:link_edgespace_to_mirrortile} $\pi_M:\edgespace \to M$ induces an isometry $\lambda_p:\lk{p,\edgespace} \to \lk{\pi_M(p), M \cap T}$.
    \item \label{item:link_edgespace_cat1} $\lk{p,\edgespace}$ is $\cat 1$.
\end{enumerate}
\end{lemma}
\proof
Of course, \eqref{item:link_edgespace_cat1} follows from \eqref{item:link_mirrortile_cat1} and \eqref{item:link_edgespace_to_mirrortile}.
For convenience, let us denote $L=\lk{ \pi_M(p),\uchc}$, $L_T=\lk{ \pi_M(p),T}$, and $L_{M\cap T}=\lk{\pi_M(p), M\cap T}$.
We have $L_{M\cap T}\subseteq L_T \subseteq L$.
Equip $L_{M\cap T}$ and $L_T$ with the induced length metric.
Let $ \overrightarrow p \in L_T$ be the direction at $\pi_M(p)$ pointing to $p$ (see Figure~\ref{fig:links}). 

\begin{figure}[ht]
\centering
\def\svgwidth{\columnwidth}
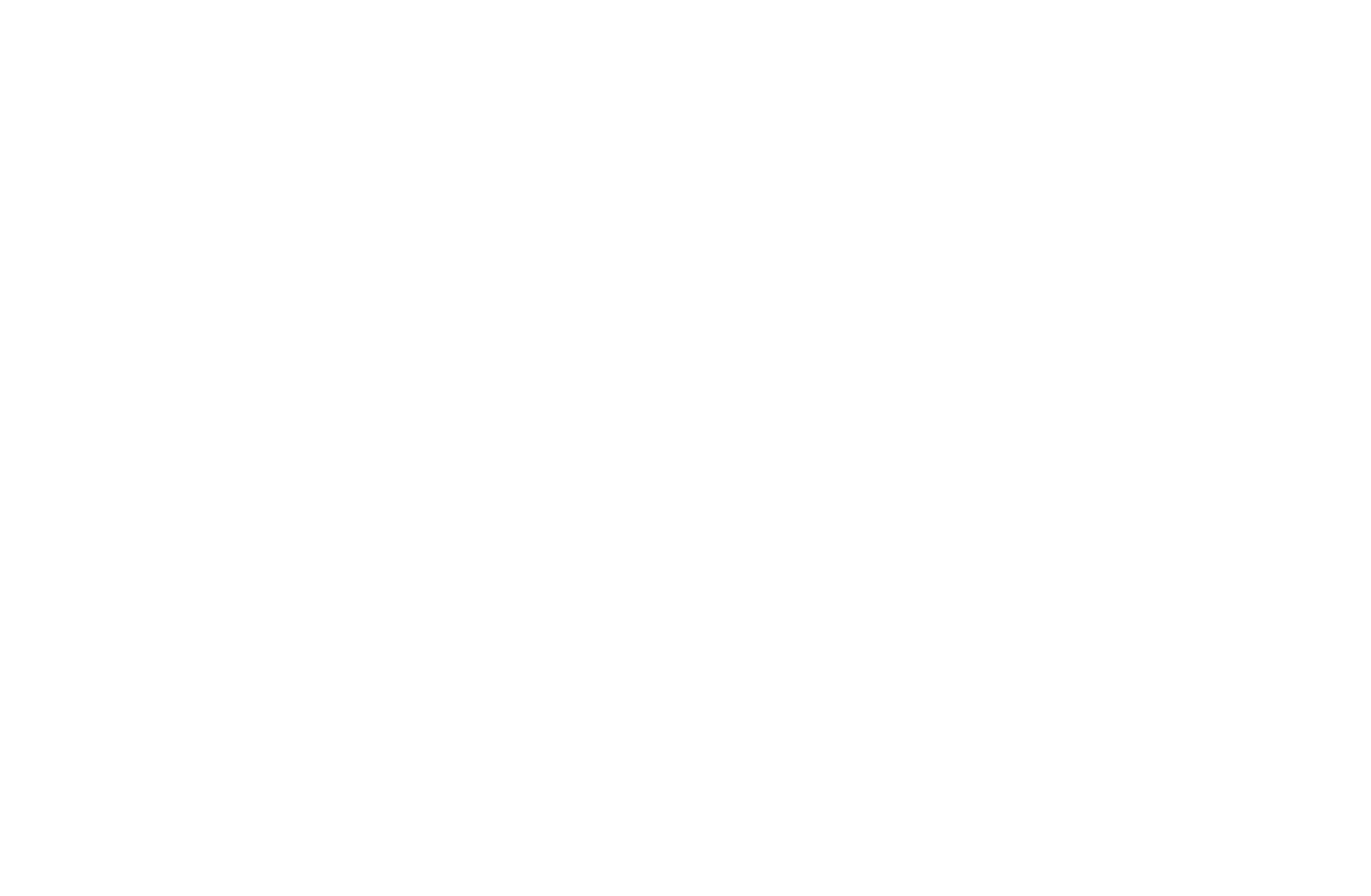
\caption{Links of points along various subspaces in the proof of Lemma~\ref{lem:edge_spaces_link_projection}. Here $p$ is contained in four tiles and sits on the intersection of two mirrors $M',M''$. The vertical projection is the nearest point projection $\pi_M:\edgespace \to M$.}
\label{fig:links}
\end{figure}

We start by proving \eqref{item:link_mirrortile_cat1}.
Since $\uchc$ is negatively curved, $L$ is $\cat 1$.
  In particular, balls of radius at most $\pi/2$ are $\pi$-convex and $\cat 1$.
Since $\uchc$ is piecewise hyperbolic, $L$ is piecewise spherical.
Moreover, all the mirrors containing $p$ intersect $M$ orthogonally by construction.
Therefore, $L$  has a natural structure of all-right spherical complex in which $\overrightarrow p$ is a vertex (possibly up to subdivision if $\pi_M(p)$ is not a vertex).
In particular, we have natural identifications $L_T=\overline{B\left( \overrightarrow p, \frac{\pi}{2} \right)}$ and $L_{M\cap T}=\partial B\left( \overrightarrow p, \frac{\pi}{2} \right)$.

Let $C_1 (Y)$ denote the spherical cone over a space $Y$, and denote the cone point by $0$.
Since $L$ is an  all-right spherical complex, we have a natural isometry
$$\varphi: C_1\left( \partial B\left( \overrightarrow p, \frac{\pi}{2} \right) \right) \to \overline{B\left( \overrightarrow p, \frac{\pi}{2} \right)}$$
defined as follows: $\varphi(0)=\overrightarrow p$, and for each $\overrightarrow q \in \partial B\left( \overrightarrow p, \frac{\pi}{2} \right)$ and $0<t\leq \frac{\pi}{2}$ let $\varphi(t,\overrightarrow q)$ be the point at distance $t$ from $\overrightarrow p$ along the geodesic $[\overrightarrow p, \overrightarrow q]$.    
As a result, $C_1(L_{M\cap T})= C_1\left( \partial B\left( \overrightarrow p, \frac{\pi}{2} \right) \right)$ is $\cat 1$. By Berestovskii's Theorem (see \cite[II.3.14]{BH99}) we conclude that $L_{M\cap T}$ is $\cat 1$ as desired.

\bigskip

To prove \eqref{item:link_edgespace_to_mirrortile} we argue as follows.
By \eqref{item:edge_space_in_tile} in Lemma~\ref{lem:edge_spaces_combinatorial} and \eqref{item:eqsurf_conformal} in Lemma~\ref{lem:equidistant_surfaces} we know that within each tile $\tau_k$ the projection $\pi_M$ is a conformal diffeomorphism, so it induces an isometry  $\lambda_p^{\tau_k}:\lk{p,\edgespace\cap \tau_k}\to L_k=\lk{\pi_M(p), M\cap \tau_k)}$.
This is enough in the case $m=1$, i.e.\ when $p$ is contained in a single tile. When $m\geq 2$, by gluing together the maps $\lambda_p^{\tau_k}$, we obtain a map
$\lambda_p:\lk{p,\edgespace} \to L_1\cup\dots\cup L_m=L_{M\cap T}$. 
Notice that shooting geodesic rays from $\pi_M(p)$ into $T$ along directions in $L_T$ provides an isometry
$$ \psi: \lk{\overrightarrow p, L_T} \to \lk{p,\edgespace}$$
Combining this with the natural isometry 
$$ r: \lk{\overrightarrow p, B\left( \overrightarrow p, \frac{\pi}{2} \right)}\to \partial B\left( \overrightarrow p, \frac{\pi}{2} \right)$$
and using the aforementioned identifications, we obtain the desired isometry
$$\lk{p,\edgespace} \overset{\psi^{-1}}{\to} \lk{\overrightarrow p, L_T} = \lk{\overrightarrow p, \overline{B\left( \overrightarrow p, \frac{\pi}{2} \right)} } \overset{r}{\to}  \partial B\left( \overrightarrow p, \frac{\pi}{2} \right) = L_{M\cap T}.$$
\endproof

\begin{remark}\label{rem:link mirror branch}
Note that, in the notation of Lemma~\ref{lem:edge_spaces_link_projection}, $L_{M\cap T}=\lk{\pi_M(p), M\cap T}$ is a closed subspace of $\lk{\pi_M(p), M}$ which is possibly proper. Indeed, $\pi_M(p)$ might live on a lower dimensional cell, where $M$ might branch off away from $T$, as in Figure~\ref{fig:mirror_branch}.
However, all the branches make an angle of at least $\pi$ with each other, because $M$ is convex.
\end{remark}

\begin{figure}[h]
\centering
\def\svgwidth{.6\columnwidth}
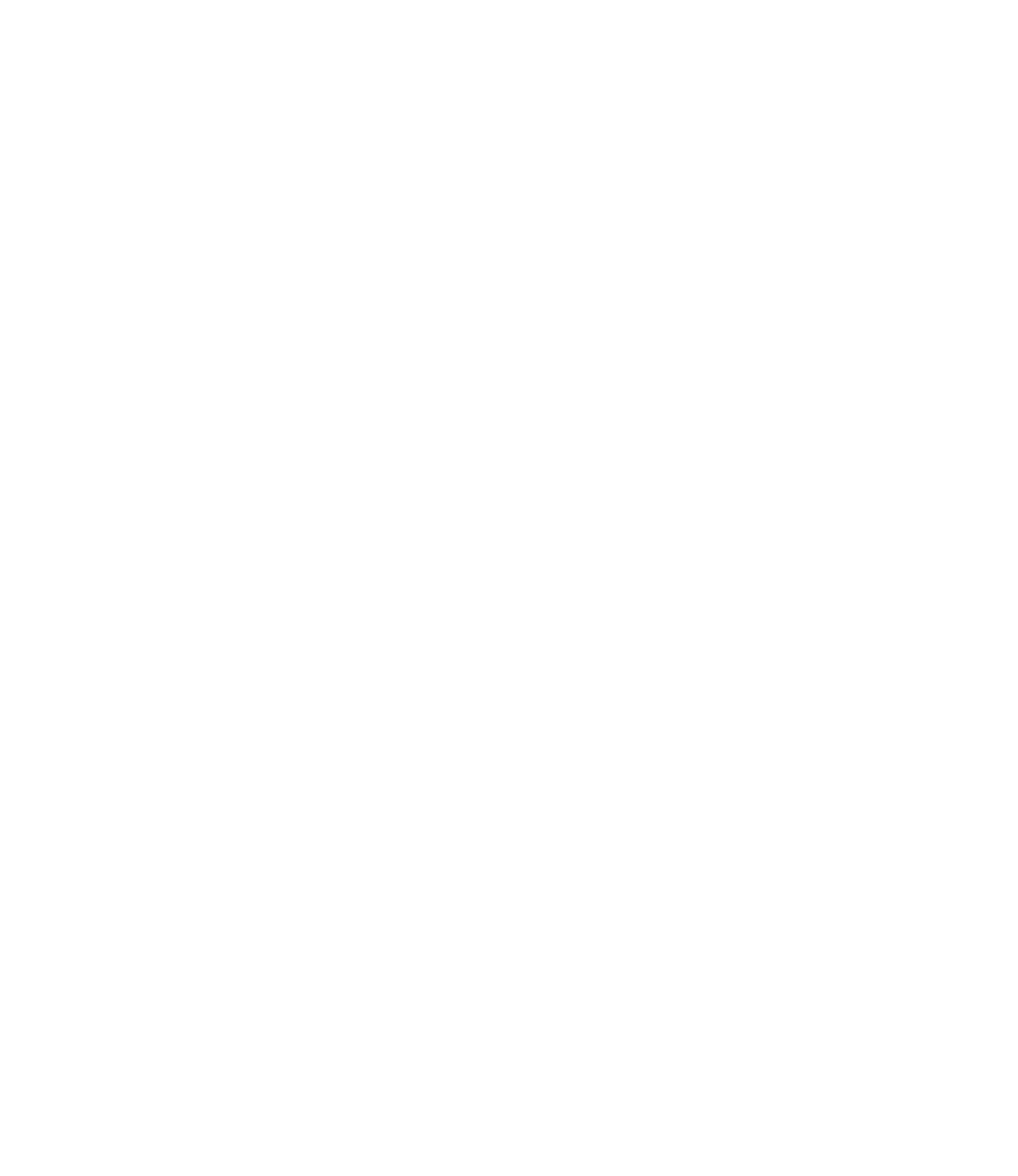
\caption{A mirror $M$ branching away from $T$, the  union of tiles containing $p$ (other mirrors not displayed).}
    \label{fig:mirror_branch}
\end{figure}


\begin{lemma}\label{lem:edge_spaces_npc}
Let $M\in \mirrors_i$ and $C\in \components_i$.
Then for $\varepsilon >0$ small enough there is $k\in (-1,0)$ such that the following hold.
\begin{enumerate}
    \item \label{item:edge_space_npc} The metric induced on $\edgespace$ is locally $\cat k$.  
    
    \item \label{item:edge_space_geodesics}  The nearest point projection $\pi_M:\edgespace \to M$ maps non--constant local geodesics to non--constant local geodesics.
    
    \item \label{item:edge_space_cat-1} The metric induced on $\edgespace$ is $\cat k$. 
    
\end{enumerate}
\end{lemma}
\proof
To prove \eqref{item:edge_space_npc} we argue as follows. By \eqref{item:edge_space_in_tile} in Lemma~\ref{lem:edge_spaces_combinatorial}, we know that, away from the intersection with mirrors, $\edgespace $ is locally isometric (via a developing map) to an equidistant hypersurface in $\hh^n$. Such a hypersurface is a manifold of negative curvature $k\in (-1,0)$ by \eqref{item:eqsurf_conformal} in Lemma~\ref{lem:equidistant_surfaces}. 
By Remark~\ref{rem:weird cell structure mirrors and edgespaces}, $\edgespace$ is essentially a cell complex, so by \cite[Theorem II.5.2]{BH99} $\edgespace$ is locally $\cat k$ if and only if the link of every vertex is a $\cat 1$ space. 
This condition is verified by \eqref{item:link_edgespace_cat1} in Lemma~\ref{lem:edge_spaces_link_projection}.


Now we consider \eqref{item:edge_space_geodesics}. 
By  \eqref{item:edge_space_in_tile} in Lemma~\ref{lem:edge_spaces_combinatorial} and \eqref{item:eqsurf_conformal} in Lemma~\ref{lem:equidistant_surfaces}, we know that in the interior of each tile $\pi_M$ is a conformal diffeomorphism with constant conformal factor. Therefore it sends a local geodesic on $\edgespace$ to a piecewise local geodesic on $M$, possibly broken at points where two or more tiles meet. 
To take care of those possibly singular points, we invoke \eqref{item:link_edgespace_to_mirrortile} in  Lemma~\ref{lem:edge_spaces_link_projection}, which guarantees that $\pi_M$ induces an isometric embedding of links also at those points. 
Indeed, if $p\in \edgespace$ is such a break point, and $c$ is a geodesic on $\edgespace$ through $p$, then the incoming and outgoing directions are at distance  $D\geq \pi$ in $\lk{p, \edgespace}$. 
Let $c'=\pi_M(c)$. Then $c'$ is a piecewise geodesic in $M$ through $\pi_M(p)$. 
With the notations of Lemma~\ref{lem:edge_spaces_link_projection}, the distance in $\lk{\pi_M(p), M\cap T}$ between the incoming and outgoing directions is the same $D\geq \pi$. The distance in the full $\lk{\pi_M(p),M}$ is not smaller, as $\lk{\pi_M(p),M}$ does not contain geodesic loops shorter than $2\pi$ by convexity. 
So, $c'$ is a local geodesic in $M$ at $\pi_M(p)$. 
Moreover if $c$ is non-constant then $c'$ is non--constant because $\pi_M$ is locally injective.

To conclude, we prove  \eqref{item:edge_space_cat-1}. By \eqref{item:edge_space_npc} we know that $\edgespace$ is locally $\cat k$, so we only need to prove that it is also simply connected. By contradiction, let $\gamma\in \pi_1(\edgespace)$ be a non--trivial homotopy class. 
Since $\edgespace$ is complete and non--positively curved, $\gamma$ is represented by a unique non--constant local geodesic $c_\gamma$. 
By \eqref{item:edge_space_geodesics} $\pi_M(c_\gamma)$ is a non--constant local geodesic on $M$. Since $M$ is complete and non--positively curved, $\pi_M(c_\gamma)$ is not nullhomotopic, 
which contradicts the fact that $M$ is contractible.
\endproof

\begin{remark}\label{rem:fold npc}
Note that if for a mirror $M$ and a tile $\tau$ the intersection $M\cap \tau$ was lower--dimensional, then the equidistant space space would develop to an equidistant hypersurface from a lower--dimensional totally geodesic subspace of $\hh^n$, which has some positive curvature.
So,  Lemma~\ref{lem:mirrors_in_tile} (establishing that if a mirror intersects a tile then the intersection is a codimension-1 cell) is a key tool to prove that edge spaces are non--positively curved.
\end{remark}

\begin{figure}[h]
\centering
\def\svgwidth{\columnwidth}
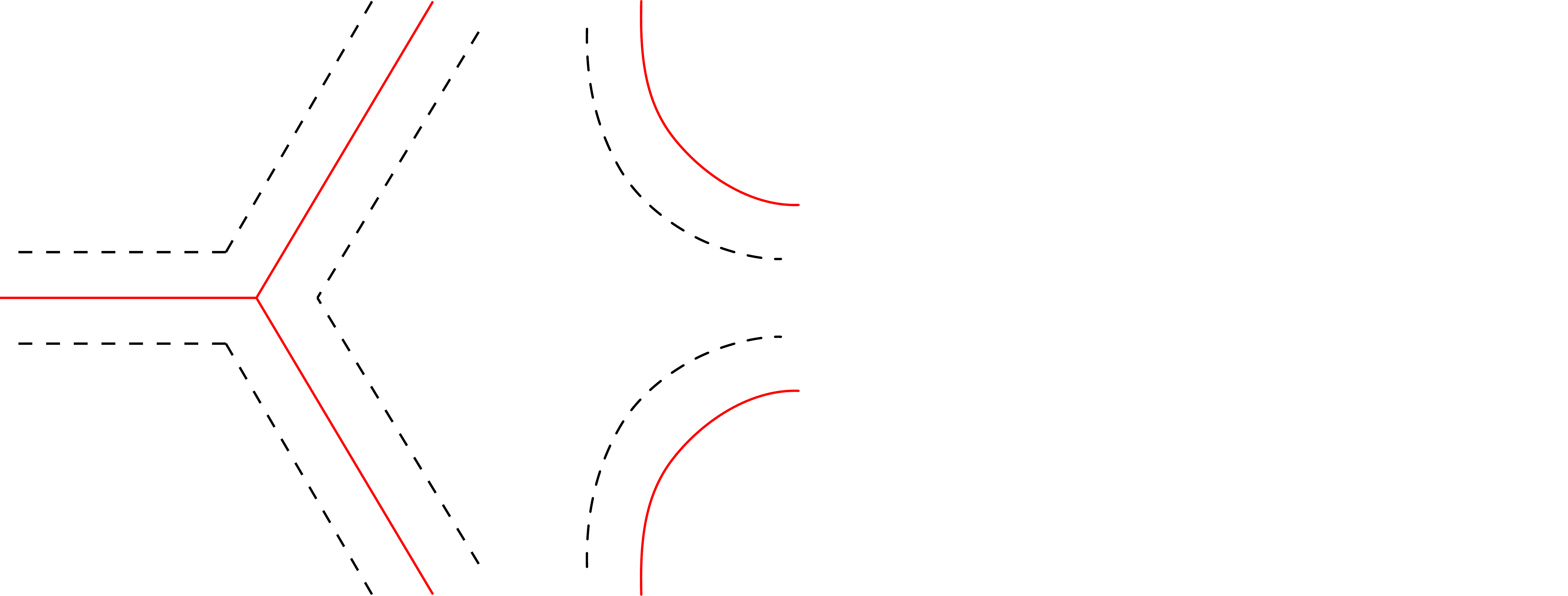
    \caption{The graph of spaces decomposition of $\uchc$}    
    \label{fig:graph_spaces}
\end{figure}

\begin{proposition}\label{prop:graph of spaces}
$\uchc$ admits the structure of a graph of spaces, with underlying graph a connected tree.
\end{proposition}
\proof
We define a graph $T_i$ as follows (see Figure~\ref{fig:graph_spaces}). Vertices come in two different families, namely a vertex $v_M$ for each mirror $M\in \mirrors_i$ and a vertex $v_C$ for each component $C\in \components_i$. Then we place one edge $e_{M,C}$ between $v_M$ and $v_C$ whenever $M$ intersects the closure of $C$.
Vertex and edge spaces are defined as follows:
we associate $M$ to $v_M$, $C$ to $v_c$, and $\edgespace$ to the edge $e_{M,C}$ between them.

The edge maps to the two types of vertices are respectively given by the nearest point projection $\pi_M:\edgespace \to M$ and the natural inclusion $i:\edgespace \hookrightarrow C$. 
Note that $i$ is an embedding, and that by \eqref{item:eqsurf_conformal} in Lemma~\ref{lem:equidistant_surfaces} we know that the restriction of  the projection $\pi_M$ to each tile is an embedding too.
Moreover, edge spaces are contractible by \eqref{item:edge_space_cat-1} in Lemma~\ref{lem:edge_spaces_npc}, so the gluing maps are automatically injective on fundamental groups.

Recall that by construction any two mirrors from $\mirrors_i$ are disjoint, and even have disjoint $\varepsilon$--neighborhoods for $\varepsilon$ sufficiently small (because $\Gamma$ is cocompact).
Similarly, any two components from $\components_i$ are disjoint.
Moreover, $\uchc \setminus \bigsqcup_{M\in \mirrors_i} M = \bigsqcup_{C\in \components_i} C$ and the boundary of a component consists of a disjoint union of closed subspaces, each of which sits inside a different mirror from $\mirrors_i$.
The resulting graph of spaces is homeomorphic to $\uchc$. 

We are left to show that $T_i$ is a connected tree.
Connectedness of $T_i$ follows directly from that of $X$.
There is a natural map $r_i:\uchc \to T_i$ obtained by collapsing all the vertex spaces to points and all the cylinders over edge spaces to edges. Notice that $r_i$ is a retraction and $\uchc$ is contractible, which forces $T_i$ to be simply connected.
\endproof

\begin{remark}\label{rem:edge map not nice}
In this graph of spaces decomposition all the spaces involved are non--positively curved, but the edge maps are not local isometries. Moreover, further pathological behavior can arise depending on the structure of the mirrors, as we now discuss. Note that the following phenomena already arise in the setting of cubical complexes, i.e.\ are not introduced by the hyperbolization procedure.

On one hand, if the mirror $M$ branches (i.e.\ has non--locally Euclidean points) in such a way that different branches meet the closure of different complementary components, then the nearest point projections $\pi_M:\edgespace\to M$ from the individual edge spaces fail to be surjective. 

On the other hand, if the mirror $M$ is such that a complementary component $C$ wraps around $M$ and meets it on different sides, then the map $\pi_M:\edgespace\to M$ fails to be injective.
This would be the case for a mirror that separates locally but not globally, e.g.\ one that is contained in the closure of a single complementary component. In this case the corresponding vertex would be a boundary vertex for the tree $T_i$. We will see in \S~\ref{subsec:separation} that this failure of injectivity does not occur in our setting.

\end{remark}

\begin{remark}[A graph of groups decomposition for $\hg$]
Note that $\hg=\pi_1(\hc)$ acts on $\uchc$ sending mirrors to mirrors and preserving the coloring, i.e.\ each family $\mirrors_i$. In particular it preserves this graph of spaces decomposition, hence it acts on the underlying graph, which has been seen to be a tree. The action is without global fix points and without inversions. This realizes $\hg=\pi_1(\hc)$ as a graph of groups. 
It is worth noticing that combination theorems are available in the literature, which provide a way to construct a cubulation of a group expressed as a graph of cubulated groups, when certain conditions are met (see for instance \cite{HW12,HW19,W21}).
In our context, the vertex groups are given by the fundamental groups of the mirrors from $\mirrors_i$ and the components from $\components_i$. While it is reasonable to expect that the former are cubulated (e.g.\ arguing by induction on dimension), it is not at all clear that the latter should be.
The guiding idea for the rest of the paper is that nevertheless those components can be further decomposed into tiles. The fundamental group of a tile can be shown to be cubulated (see Lemma~\ref{lem:hyperbolizing cube is vcs}), and the results of Groves and Manning from \cite{GM18} then provide a way to combine the cubulation from each tile into a global cubulation.
\end{remark}


\subsection{Mirrors: separation}\label{subsec:separation}
In this section we will prove a strong separation property for mirrors in $\uchc$.
In order to obtain convexity of the mirrors, in the proof of Proposition~\ref{prop:mirrors convex} we have used the fundamental fact that in a $\cat 0$ space local convexity implies global convexity. The same local--to--global property fails for separation, as shown by the following example.

\begin{example}\label{ex:locsep nonsep}
Consider the square complex $Y$  in the center of Figure~\ref{fig:locsepnonsep}. 
Consider the subcomplex $Z$ consisting of the central thick (red) edge.
The subspace $Z$  is locally separating in $Y$, in the sense that for any $z \in Z$ and any  arbitrarily small neighborhood $U_z$ of  $z$ in $Y$, $U_z\setminus Z$ is disconnected. However, $Z$ is not separating, i.e.\ $Y\setminus Z$ is connected. 
Notice that  $Y$ is a $\cat 0$ and foldable cubical complex, but $Z$ is not a full connected component of the preimage of a codimension-$1$ face, i.e.\ not a mirror.

In this example both $Y$ and $Z$ have boundary, but it can be modified to obtain an example without boundary. We start by attaching eight more squares following the pattern in Figure~\ref{fig:locsepnonsep}, and extending $Z$ with two more edges. In the resulting complex, no edge meeting $Z$ is a boundary edge, so we can keep adding squares (and extending $Z$) to get an admissible complex which displays the same pathology as the original one.
\end{example}

\begin{figure}[h]
\centering
\def\svgwidth{\columnwidth}
\begingroup%
  \makeatletter%
  \providecommand\color[2][]{%
    \errmessage{(Inkscape) Color is used for the text in Inkscape, but the package 'color.sty' is not loaded}%
    \renewcommand\color[2][]{}%
  }%
  \providecommand\transparent[1]{%
    \errmessage{(Inkscape) Transparency is used (non-zero) for the text in Inkscape, but the package 'transparent.sty' is not loaded}%
    \renewcommand\transparent[1]{}%
  }%
  \providecommand\rotatebox[2]{#2}%
  \newcommand*\fsize{\dimexpr\f@size pt\relax}%
  \newcommand*\lineheight[1]{\fontsize{\fsize}{#1\fsize}\selectfont}%
  \ifx\svgwidth\undefined%
    \setlength{\unitlength}{1622.83464567bp}%
    \ifx\svgscale\undefined%
      \relax%
    \else%
      \setlength{\unitlength}{\unitlength * \real{\svgscale}}%
    \fi%
  \else%
    \setlength{\unitlength}{\svgwidth}%
  \fi%
  \global\let\svgwidth\undefined%
  \global\let\svgscale\undefined%
  \makeatother%
  \begin{picture}(1,0.22358073)%
    \lineheight{1}%
    \setlength\tabcolsep{0pt}%
    \put(0,0){\includegraphics[width=\unitlength,page=1]{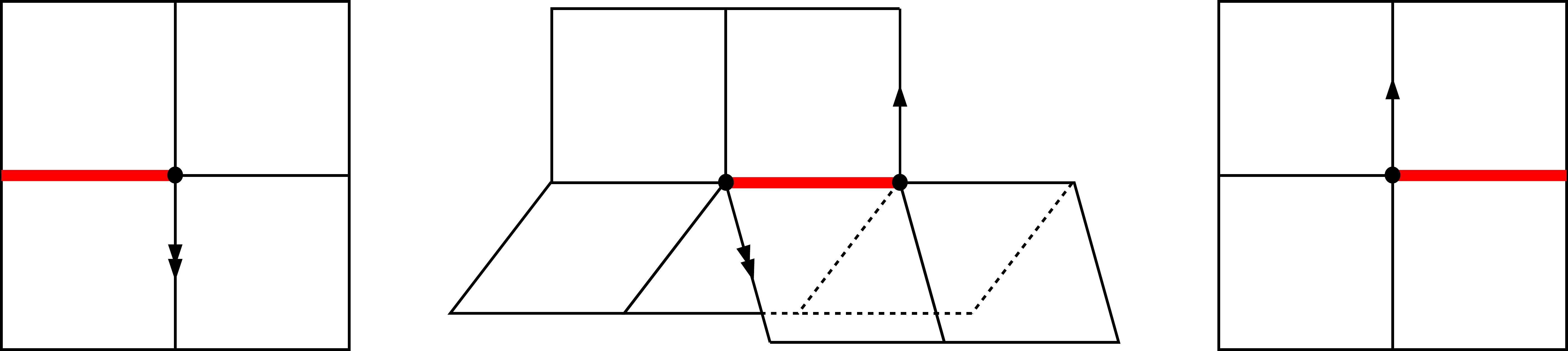}}%
    \put(0.63284704,0.17698106){\color[rgb]{0,0,0}\makebox(0,0)[lt]{\lineheight{1.25}\smash{\begin{tabular}[t]{l}$Y$\end{tabular}}}}%
    \put(0.50092512,0.12548327){\color[rgb]{1,0,0}\makebox(0,0)[lt]{\lineheight{1.25}\smash{\begin{tabular}[t]{l}$Z$\end{tabular}}}}%
  \end{picture}%
\endgroup%

\caption{A locally separating but not separating subcomplex in a $\cat 0$ square complex.}
\label{fig:locsepnonsep}
\end{figure}

When $Y$ is a homogeneous cubical complex of dimension $n$, every $k$--cube $F$ of $Y$ is contained in some $n$-cell. When $Y$ has no boundary, $F$ is contained in at least two distinct $n$-cubes. This motivates the following definition. 
Let $M$ be a mirror of $Y$ and let $F$ be a $k$--cube of $M$. 
A \textit{framing} for $F$ is a choice of two $n$--cubes $\{C_1,C_2\}$ of $Y$ such that $F\subseteq C_1\cap C_2 \subseteq M$.
We note explicitly that this definition is relative to the fixed mirror $M$.
For the next proof, we will make use of some properties of hyperplanes in $\cat 0$ cubical complexes. We refer the reader to \cite[Theorem 4.10]{SA95} or \cite[Example 3.3.(3), Lemma 13.3]{HW08} for details and proofs.

\begin{figure}[h]
\centering
\def\svgwidth{.5\columnwidth}
\begingroup%
  \makeatletter%
  \providecommand\color[2][]{%
    \errmessage{(Inkscape) Color is used for the text in Inkscape, but the package 'color.sty' is not loaded}%
    \renewcommand\color[2][]{}%
  }%
  \providecommand\transparent[1]{%
    \errmessage{(Inkscape) Transparency is used (non-zero) for the text in Inkscape, but the package 'transparent.sty' is not loaded}%
    \renewcommand\transparent[1]{}%
  }%
  \providecommand\rotatebox[2]{#2}%
  \newcommand*\fsize{\dimexpr\f@size pt\relax}%
  \newcommand*\lineheight[1]{\fontsize{\fsize}{#1\fsize}\selectfont}%
  \ifx\svgwidth\undefined%
    \setlength{\unitlength}{822.6417169bp}%
    \ifx\svgscale\undefined%
      \relax%
    \else%
      \setlength{\unitlength}{\unitlength * \real{\svgscale}}%
    \fi%
  \else%
    \setlength{\unitlength}{\svgwidth}%
  \fi%
  \global\let\svgwidth\undefined%
  \global\let\svgscale\undefined%
  \makeatother%
  \begin{picture}(1,0.50183273)%
    \lineheight{1}%
    \setlength\tabcolsep{0pt}%
    \put(0,0){\includegraphics[width=\unitlength,page=1]{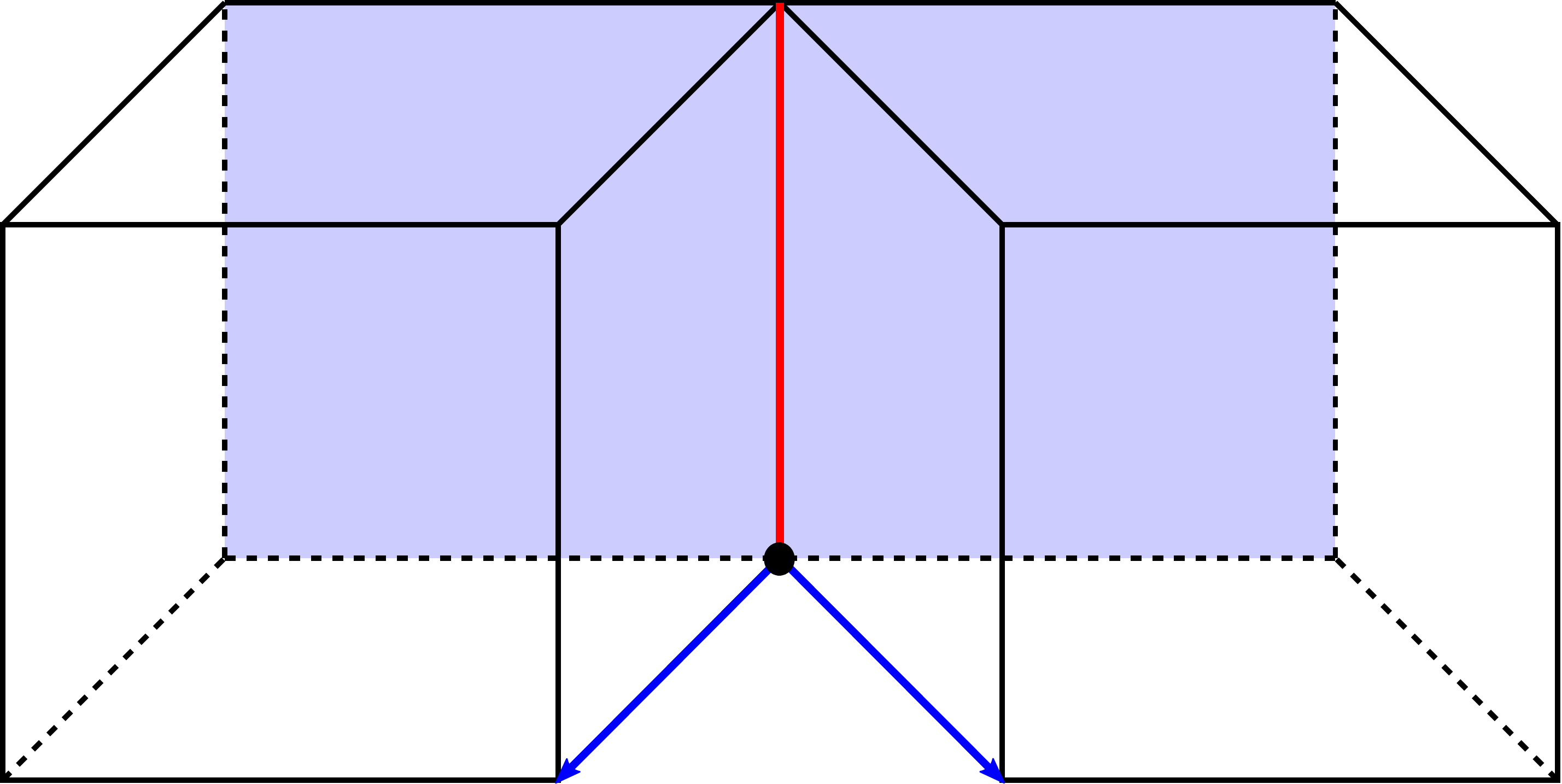}}%
    \put(0.5195822,0.30864719){\color[rgb]{1,0,0}\makebox(0,0)[lt]{\lineheight{1.25}\smash{\begin{tabular}[t]{l}$F$\end{tabular}}}}%
    \put(0.4320592,0.16459895){\color[rgb]{0,0,0}\makebox(0,0)[lt]{\lineheight{1.25}\smash{\begin{tabular}[t]{l}$v$\end{tabular}}}}%
    \put(0.20048837,0.44540174){\color[rgb]{0,0,1}\makebox(0,0)[lt]{\lineheight{1.25}\smash{\begin{tabular}[t]{l}$M$\end{tabular}}}}%
    \put(0.39734897,0.01228358){\color[rgb]{0,0,1}\makebox(0,0)[lt]{\lineheight{1.25}\smash{\begin{tabular}[t]{l}$e_1$\end{tabular}}}}%
    \put(0.54321938,0.01228358){\color[rgb]{0,0,1}\makebox(0,0)[lt]{\lineheight{1.25}\smash{\begin{tabular}[t]{l}$e_2$\end{tabular}}}}%
    \put(0.01450184,0.22659446){\color[rgb]{0,0,0}\makebox(0,0)[lt]{\lineheight{1.25}\smash{\begin{tabular}[t]{l}$C_1$\end{tabular}}}}%
    \put(0.88972423,0.22659446){\color[rgb]{0,0,0}\makebox(0,0)[lt]{\lineheight{1.25}\smash{\begin{tabular}[t]{l}$C_2$\end{tabular}}}}%
  \end{picture}%
\endgroup%

\caption{A framing for a cube $F$ on a mirror $M$.}
\label{fig:framing}
\end{figure}

\begin{lemma}\label{lem:CAT(0)_mirrors_separate}
Let $Y$ be a $\cat 0$ admissible cubical complex.
Then each mirror separates $Y$. More precisely, let $M\subseteq Y$ be a mirror, let $F\subseteq M$ be a $k$--cube, and let $\{C_1,C_2\}$ be a framing for $F$. Then $C_1,C_2$ are contained in the closure of two distinct connected components of $Y\setminus M$.
\end{lemma}
\proof
Let $v$ be a vertex on $F$, let $e_i$ be the edge of $C_i$ with starting point $v$ and endpoint in $C_i\setminus M$ (see Figure~\ref{fig:framing}). 
Note that this edge exists because $C_i$ is $n$--dimensional, while $M$ is $(n-1)$--dimensional and convex, so that $M\cap C_i$ is some $(n-1)$--dimensional face $E_i$ of $C_i$, by an argument similar to that of Lemma~\ref{lem:mirrors_in_tile}.
Also note that by definition of framing, $C_1\cap C_2\subseteq M$, and therefore $e_1\neq e_2$.
Let  $H_i$ be the hyperplane of $Y$ dual to $e_i$. 
In particular this means that $H_i$ meets $C_i$ in the midcube orthogonal to $e_i$.
Since $Y$ is $\cat 0$, $Y$ is special in the sense of \cite{HW08}. 
Since foldability of $Y$ prevents $e_1,e_2$ from being contained in the same square, we then get that $H_1\neq H_2$ (hyperplanes do not self-osculate) and $H_1\cap H_2=\varnothing$ (hyperplanes do not inter-osculate).  
Moreover, $H_k\cap M=\varnothing$ and $Y\setminus H_k$ consists of exactly two components, one containing $M$ and one not containing $M$.


The carrier of a hyperplane $H$ in a $\cat 0$ is isomorphic to $H\times [0,1]$. 
By definition  of mirror, if $M$ contains an $(n-1)$-cube of $H\times \{0\}$  then actually $H\times \{0\}\subseteq M$. 
Since $M$ contains the $(n-1)$--cell $E_i=C_i\cap M$ of $C_i$, and $E_i \subseteq H_i\times \{0\}$  by construction, we can conclude that $M$ contains $H_i\times\{0\}$ for $i=1,2$.
It follows that any path from $H_1$ to $H_2$ must intersect $M$.
In particular, $M$ separates $Y$ in at least two components, one containing $H_1$ and one containing $H_2$. The closures of such components contain $C_1$ and $C_2$ respectively.
\endproof

We want to extend this result to mirrors in $\uchc$. 
To do this, we introduce the following terminology, in analogy with the cubical case.
Let $M$ be a mirror of $\uchc$, and let $\sigma$ be a $k$--cell of $M$. A \textit{framing} for $\sigma$ is the choice of two distinct $n$--cells $\tau_1, \tau_2$ such that $\sigma \subseteq \tau_1\cap \tau_2 \subseteq M$.
We begin by obtaining a weak separation property.
 
 \begin{lemma}\label{lem:mirror almost separates framing}
 Let $M\in \mirrors_i$, let $\sigma\subseteq M$ be a $k$--cell, and let $\{\tau_1,\tau_2\}$ be a framing for $\sigma$.
 Then there exist two different components $C_1,C_2 \in \components_i$ whose closure contain $\tau_1,\tau_2$ respectively.
 \end{lemma}
 \proof
The map $\cdX:\hc\to X$ lifts to a map $\alpha:\uchc\to \widetilde X$ between the universal covers. 
Note that it sends mirrors to mirrors. 
In particular we obtain a mirror $\alpha(M)$ and a $k$--cube $\alpha(\sigma)\subseteq \alpha(M)$ with a framing $\{\alpha(\tau_1),\alpha(\tau_2)\}$. 
By Lemma~\ref{lem:CAT(0)_mirrors_separate} we can conclude that $\alpha(\tau_1)$ and $\alpha(\tau_2)$ are separated by $\alpha(M)$ in $\widetilde X$.
This implies that $\alpha^{-1}(\alpha (\tau_1))$ and $\alpha^{-1}(\alpha (\tau_2))$ are separated in $\uchc$ by $\alpha^{-1}(\alpha (M))$, i.e.\ the full preimage of the mirror $\alpha (M)$ in $\uchc$. 
Note that $\alpha^{-1}(\alpha (M))$ consist of infinitely many mirrors from $\mirrors_i$: indeed, recall from Lemma~\ref{lem:mirrors_in_tile} that disjoint $(n-1)$--cells of a tile belong to different mirrors.
A fortiori, $\tau_1$ and $\tau_2$ are separated by the entire collection $\mirrors_i$. 
In particular, there exists two different components $C_1,C_2 \in \components_i$ whose closure contain $\tau_1,\tau_2$ respectively, as desired.
 \endproof
 
\begin{remark}
Observe that in the proof of Lemma~\ref{lem:mirror almost separates framing}, it is not clear whether the framing is separated by $M$ itself. While the entire collection of mirrors $\mirrors_i$ disconnects $\uchc$ into a collection of complementary components, it is not a priori clear that any single mirror separates $\uchc$. 
\end{remark}

Recall from Proposition~\ref{prop:graph of spaces} that $\uchc$ admits the structure of a graph of spaces over a connected tree $T_i$, and that there is a natural retraction $r_i:\uchc \to T_i$ obtained by collapsing all the vertex spaces to points and all the cylinders over edge spaces to edges.

\begin{lemma}\label{lem:dual_graph_is_boundaryless_tree}
The tree $T_i$ has no boundary.
\end{lemma}
\proof
It is enough to show that each vertex has at least two neighboring vertices.
Vertices of $T_i$ are either associated to mirrors from $\mirrors_i$ or to components from $\components_i$. 
We analyze the two different cases separately.
Let $v_C$ be the vertex associated to a component $C\in \components_i$. Then $v_C$ has infinitely many edges coming into it, because $C$ has infinitely many mirrors from $\mirrors_i$ on its boundary (this is already true for a single tile: by Lemma~\ref{lem:mirrors_in_tile}, disjoint $(n-1)$--cells in the boundary of a tile belong to different mirrors). 

Now let $v_M$ be the vertex associated to a mirror $M\in \mirrors_i$.
Let $\sigma\subseteq M$ be an $(n-1)$-cell on it, and pick a framing $\{\tau_1,\tau_2\}$.
By Lemma~\ref{lem:mirror almost separates framing}, there exist two different components $C_1,C_2 \in \components_i$ whose closure contain $\tau_1,\tau_2$ respectively.
The corresponding vertices $v_{C_1},v_{C_2}$ in $T_i$ are both adjacent to the vertex $v_M$ corresponding to $M$, as desired.
\endproof

The next result is the analogue of Lemma~\ref{lem:CAT(0)_mirrors_separate} from the cubical case.

\begin{proposition}\label{prop:mirrors separate}
Each $M\in \mirrors_i$ separates $\uchc$. 
More precisely, let $M\in \mirrors_i$ be a mirror, let $\sigma \subseteq M$  be a $k$--cell, and let $\{\tau_1,\tau_2\}$ be a framing for $\sigma$. Then $\tau_1,\tau_2$ are contained in the closure of two distinct connected components of $\uchc\setminus M$.
\end{proposition}
\proof
For the first statement, consider the natural retraction $r_i:\uchc \to T_i$. 
Note that for each  mirror $M\in \mirrors_i$ there is a corresponding vertex $v_M\in T_i$, and $M=r_i^{-1}(v_M)$.
By Lemma~\ref{lem:dual_graph_is_boundaryless_tree} we know that $T_i$ is a tree with no boundary, hence any of its vertices disconnects it. Therefore $M=r_i^{-1}(v_M)$ disconnects $\uchc$.

For the second statement, we fix a $k$--cell $\sigma \subseteq M$ and a framing $\{\tau_1,\tau_2\}$. By Lemma~\ref{lem:mirror almost separates framing} we get two components $C_1,C_2 \in \components_i$ containing $\tau_1,\tau_2$ in their closures. Note that these are complementary components of the entire collection of mirrors $\mirrors_i$, not complementary components of the mirror $M$.
The corresponding vertices $v_{C_1},v_{C_2}$ in $T_i$ are both adjacent to the vertex $v_M$ corresponding to $M$, and are separated by $v_M$ in $T_i$, since $T_i$ is a tree (see Proposition~\ref{prop:graph of spaces}). 
Arguing as above via the natural retraction $r_i:\uchc \to T_i$, we can conclude that $\tau_1,\tau_2$ are separated by $M$ in $\uchc$.
\endproof

We conclude this section with some remarks about the construction that we have described.

\begin{remark}[Foldability is key]
Foldability of $X$ has played the role of some sort of \textit{combinatorial completeness}, as it guarantees that if a mirror $M$ intersects a tile $T$, then $M$ goes across  $T$ along a top dimensional subcomplex of the boundary. 
This has  provided both features of  non--positive curvature (see Remark~\ref{rem:fold npc})   and separation properties (as in the proof of Lemma~\ref{lem:CAT(0)_mirrors_separate}).
Example~\ref{ex:locsep nonsep} shows that neither is available if foldability is not taken into account in the definition of mirrors (even on a foldable complex).
\end{remark}

\begin{remark}[Complexes with boundary]
The construction from \S\ref{subsec:graph of spaces} can be generalized to cubical complexes that have enough \textit{good} mirrors (i.e.\ mirrors that admit a cell which locally separates a framing), and
keeping track only of such mirrors in the construction of the tree of spaces.
For instance, one could drop the assumption that $X$ is without boundary, and ignore the mirrors that are entirely contained in the boundary. One still gets a decomposition as a graph of spaces over a tree without boundary. Indeed, vertices associated to good mirrors still have degree at least $2$. One may worry about vertices associated to components. Even if there is a cube of $X$ with only one face $F$ contained in a good mirror, each of the components  $C\in \components_i$ of $\uchc$ arising from it still has infinitely many boundary cells corresponding to $F$. This guarantees that the vertices of the tree which are associated to components in $\components_i$ still have infinite degree.

Alternatively, one can work with a relative version of the Charney-Davis hyperbolization procedure that is designed to hyperbolize complexes with boundary without altering the boundary components (see \cite{CD95,BE07}).
We consider the problem of cubulating the resulting relatively hyperbolic groups in \cite{LRGM23}.
\end{remark}


\section{The dual cubical complex}\label{sec:dual cubical complex}

We define a cubical complex associated to the stratification of $\uchc$ introduced in \S\ref{subsec:stratification}, and prove that it is a $\cat 0$ cubical complex (see Theorem~\ref{thm:dual cubical complex is CAT(0)}). 
Recall that $X$ is assumed to be an admissible cubical complex (as defined at the beginning of \S\ref{sec:strict hyperbolization}). Let $n=\dim(X)$ be its dimension.
The \textit{dual cubical complex} is denoted $\dccx$ and defined as follows.

\begin{itemize}
    \item Vertices are given by all the $k$-cells in $\uchc$ for $k=0,\dots,n$. 
    \item Two vertices corresponding to cells $\sigma$ and $\tau$ are connected by an edge if and only if  $|\dim(\sigma)-\dim(\tau)|=1$, and either $\sigma \subseteq \tau$ or $\tau\subseteq \sigma$.
    \item For $k>1$, we attach one $k$-dimensional cube whenever we see its 1-skeleton. 
\end{itemize}
The resulting cell complex $\dccx$ is a cubical complex (see Figure~\ref{fig:dccx}). 
Moreover, we can label its $0$--skeleton by integers $0\leq k\leq n$: if $v$ is a vertex dual to a $k$--cell $\sigma$, then we define the \textit{height} of $v$ to be $\height v=\dim (\sigma)=k$.

\begin{figure}[h]
\centering
\def\svgwidth{\columnwidth}
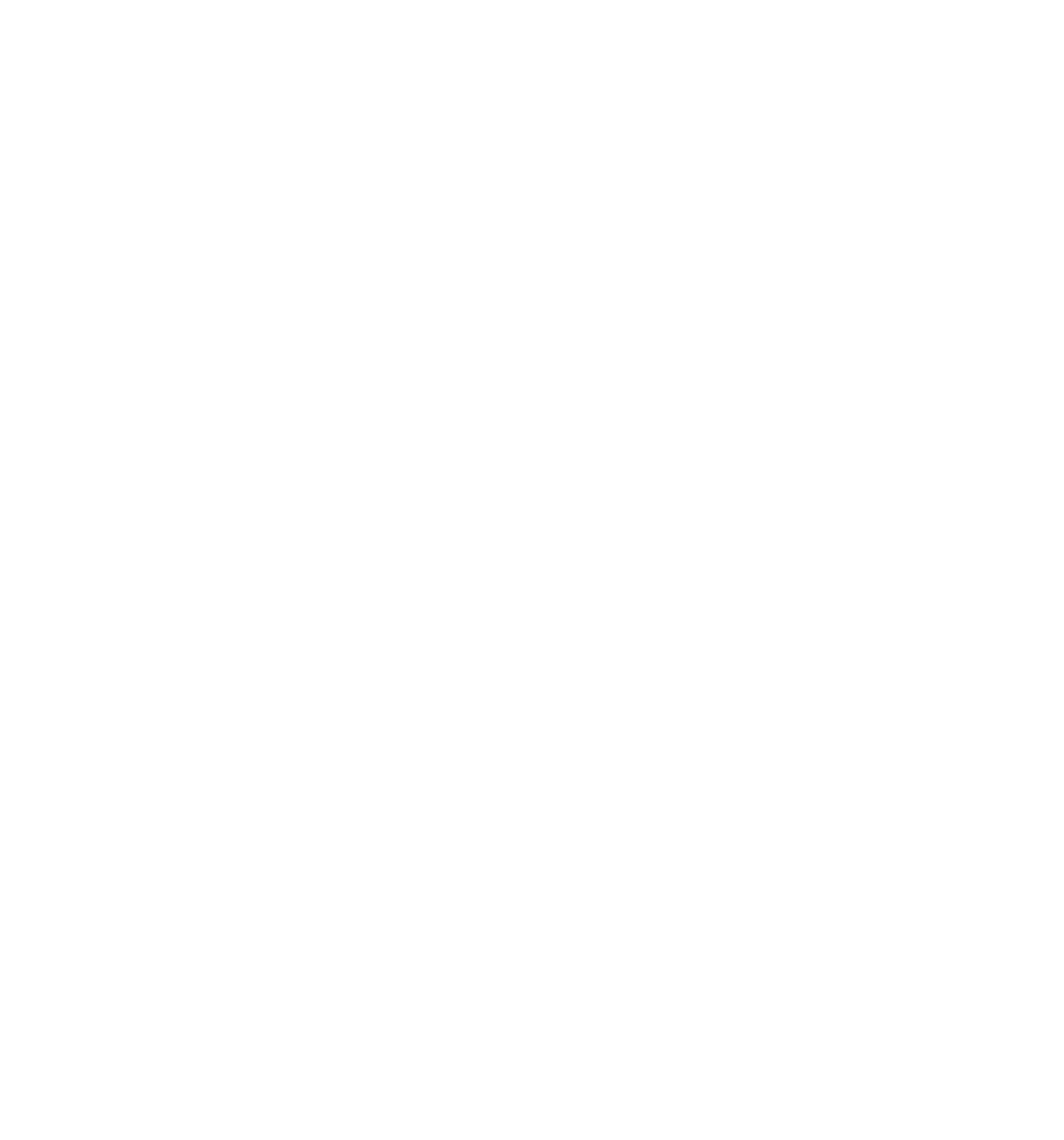
    \caption{The dual cubical complex $\dccx$ superimposed on the stratification of $\uchc$; compare Figure~\ref{fig:tiles}. (In this picture the dimension is $n=2$. Key: $\fullmoon$, $\astrosun$, and $\newmoon$  denote a vertex of height $0,1,2$ respectively.)}    
    \label{fig:dccx}
\end{figure}

Note that if one applies this construction to the standard $n$-simplex, one obtains a natural cubulation of the $n$-simplex by $n+1$ cubes of the same dimension.
However, it is not clear when this construction preserves asphericity. For instance, applying this construction to a solid octahedron results in a cube complex with non-trivial $\pi_2$.

In this section we study the combinatorial geometry of $\dccx$, by analyzing cubes and links in \S\ref{subsec:cubes and links}, some notions of complexity for edge--paths in \S \ref{subsec:efficiency} and \S \ref{subsec:mirror complexity}, and how to use them to prove that $\dccx$ is simply connected in \S \ref{subsec:mirror surgeries}. Before starting, the following two remarks address the relation between $\dccx$ and other natural combinatorial structures associated to $\uchc$ and its collection of mirrors $\mirrors$.

\begin{remark}[The associated graded poset]
The set of cells in $\uchc$ can be partially ordered by inclusion. The result is a graded poset, whose rank function is given by the dimension of the corresponding cell. The height we just defined is induced by this rank function. One could construct the order complex of such a poset, by taking a simplex for every chain. This would result in a simplicial complex, and is not what we are considering here.
\end{remark}

\begin{remark}[The associated wallspace] \label{rem:wallspaces}
Since mirrors are separating subspaces (see  Proposition~\ref{prop:mirrors separate}), the collection of mirrors can be used to define a wallspace structure $(\uchc, \mirrors)$ on $\uchc$, and one could consider the dual $\cat 0$ cubical complex $\dccm$ associated to this wallspace by Sageev's construction. 
We refer the reader to \cite{SA95,HP98,HW14} for details about this construction, and we only review the main ingredients here.
Given a mirror $M$, any partition of the complementary components into two classes is called a wall associated to $M$. An orientation of a wall is a choice of one of the two classes.
A vertex of $\dccm$ can then be described as a consistent choice of orientation for each mirror.

When $X$ and all mirrors are homeomorphic to manifolds, each mirror of $\uchc$ has exactly two complementary components.
In this quite restrictive case, an orientation of a wall is just a choice of one of the two complementary components. 
Therefore vertices of $\dccm$ correspond to tiles (i.e.\ $n$--cells) in the stratification of $\uchc$, and two vertices are connected by an edge when the corresponding tiles are adjacent along a mirror.
In particular, $\dccm$ is an $n$--dimensional cubical complex that can be subdivided to recover  $\dccx$.
However, if there are mirrors which have more than two complementary components (such as in Figures~\ref{fig:mirrors} and \ref{fig:dccx}), then we find vertices in $\dccm$ which do not correspond to tiles  from the stratification of $\uchc$ (they are not canonical vertices, in the terminology of \cite{HW14}).
As a result, the dimension of $\dccm$ is usually higher than that of $X$, and it is more challenging to relate the actions of $\hg$ on $\uchc$ and on $\dccm$. 
\end{remark}


\subsection{Cubes and links}\label{subsec:cubes and links}
In this section we explore basic facts about the cubical geometry of $\dccx$. While this complex is not locally compact (see Remark~\ref{rmk:dcc_non_locally_finite}), its dimension is the same as that of $X$ (see Lemma~\ref{lem:dcc_fin_dim}), and the links of vertices are flag complexes (see Proposition~\ref{prop:flag_links}).

The first two lemmas show that squares and cubes in $\dccx$ admit unique vertices of minimum and maximum height. Recall that the height of a vertex is the dimension of its dual cell, and notice that, by definition of $\dccx$, if $u,v$ are adjacent vertices, then $|\height u - \height v |=1$.

\begin{figure}[h]
\centering
\def\svgwidth{.5\columnwidth}
\begingroup%
  \makeatletter%
  \providecommand\color[2][]{%
    \errmessage{(Inkscape) Color is used for the text in Inkscape, but the package 'color.sty' is not loaded}%
    \renewcommand\color[2][]{}%
  }%
  \providecommand\transparent[1]{%
    \errmessage{(Inkscape) Transparency is used (non-zero) for the text in Inkscape, but the package 'transparent.sty' is not loaded}%
    \renewcommand\transparent[1]{}%
  }%
  \providecommand\rotatebox[2]{#2}%
  \newcommand*\fsize{\dimexpr\f@size pt\relax}%
  \newcommand*\lineheight[1]{\fontsize{\fsize}{#1\fsize}\selectfont}%
  \ifx\svgwidth\undefined%
    \setlength{\unitlength}{353.48156065bp}%
    \ifx\svgscale\undefined%
      \relax%
    \else%
      \setlength{\unitlength}{\unitlength * \real{\svgscale}}%
    \fi%
  \else%
    \setlength{\unitlength}{\svgwidth}%
  \fi%
  \global\let\svgwidth\undefined%
  \global\let\svgscale\undefined%
  \makeatother%
  \begin{picture}(1,0.71079839)%
    \lineheight{1}%
    \setlength\tabcolsep{0pt}%
    \put(0,0){\includegraphics[width=\unitlength,page=1]{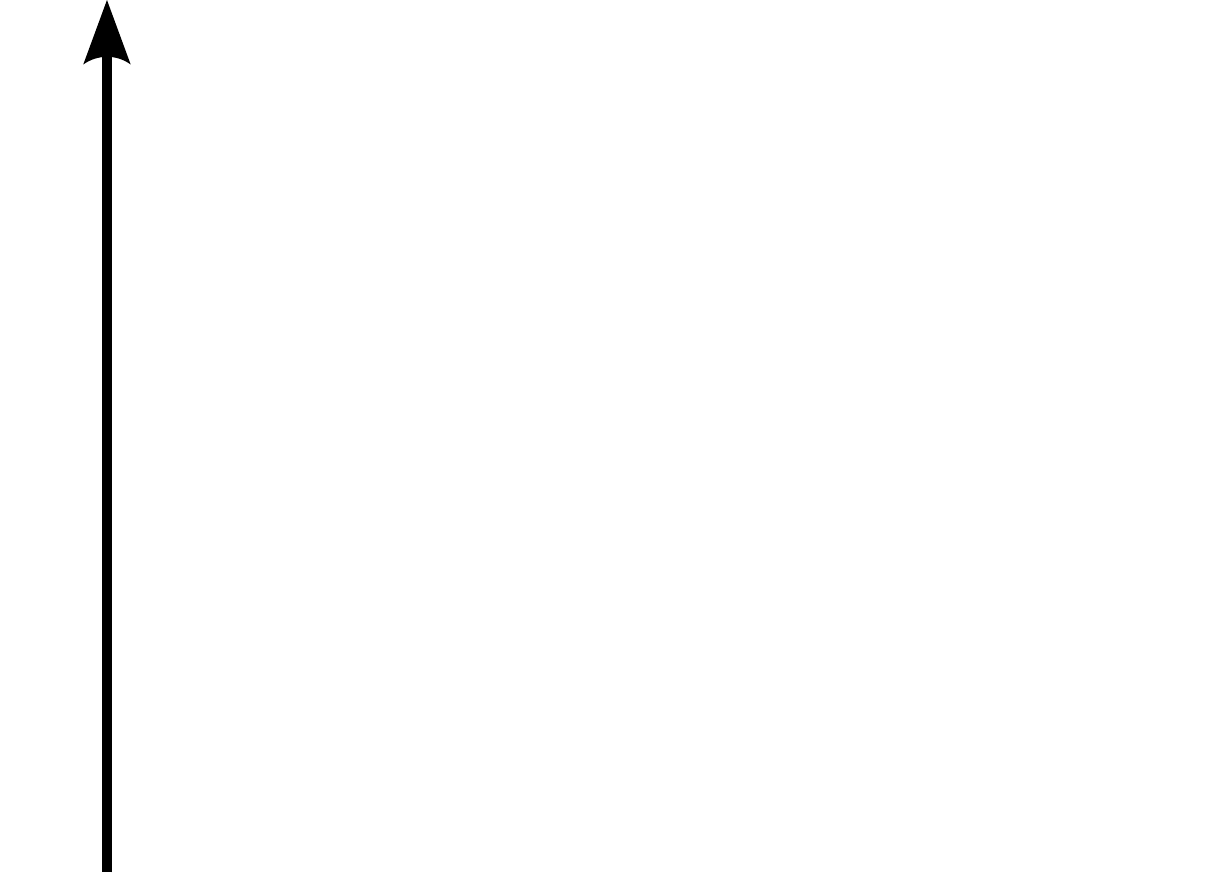}}%
    \put(-0.00642325,0.60761524){\color[rgb]{0,0,0}\makebox(0,0)[lt]{\lineheight{1.25}\smash{\begin{tabular}[t]{l}$\operatorname{h}$\end{tabular}}}}%
    \put(0,0){\includegraphics[width=\unitlength,page=2]{dcc_square.pdf}}%
    \put(0.4904514,0.06978685){\color[rgb]{0,0,0}\makebox(0,0)[lt]{\lineheight{1.25}\smash{\begin{tabular}[t]{l}$v_3$\end{tabular}}}}%
    \put(0.4904514,0.57476356){\color[rgb]{0,0,0}\makebox(0,0)[lt]{\lineheight{1.25}\smash{\begin{tabular}[t]{l}$v_1$\end{tabular}}}}%
    \put(0.23486084,0.32857213){\color[rgb]{0,0,0}\makebox(0,0)[lt]{\lineheight{1.25}\smash{\begin{tabular}[t]{l}$v_2$\end{tabular}}}}%
    \put(0.73983765,0.32857213){\color[rgb]{0,0,0}\makebox(0,0)[lt]{\lineheight{1.25}\smash{\begin{tabular}[t]{l}$v_4$\end{tabular}}}}%
  \end{picture}%
\endgroup%

    \caption{A square in $\dccx$.}
    \label{fig:dcc_square}
\end{figure}

\begin{lemma}\label{lem:height_square}
Let $S$ be a square of $\dccx$. 
Let $v_1,v_2,v_3,v_4$ be its vertices, with $v_2$ and $v_4$ non-adjacent in $S$.
If $\height{v_2} = \height{v_4}$, then $|\height{v_1}-\height{v_3}|=2$. 
In particular there is a unique vertex of maximal (respectively minimal) height, and the cell dual to it contains (respectively is contained in) each of the cells dual to the other vertices.

\end{lemma}
\proof
Let $h=\height{v_2}=\height{v_4}$ be the common value of the height of $v_2$ and $v_4$. Since $v_1$ is adjacent to $v_2$ and $v_4$, we have $\height{v_1}=h\pm 1$, and similarly for $v_3$ (see Figure~\ref{fig:dcc_square}). 
In particular $|\height{v_1}-\height{v_3}|$ is either $0$ or $2$.
By contradiction let us assume that $|\height{v_1}-\height{v_3}|=0$, i.e.\ $\height{v_1}=\height{v_3}=h\pm 1$.
Without loss of generality we can assume that $\height{v_1}=\height{v_3}=h+1$.  (The case $\height{v_1}=\height{v_3}=h-1$ is completely analogous, via a dual argument).
For $j=1,2,3,4$, let $\sigma_j$ be the cell of $\uchc$ dual to the vertex $v_j$.
Since $v_1$ is adjacent to $v_2$ and $v_4$, and has higher height, $\sigma_1$ contains $\sigma_2$ and $\sigma_4$; the same holds for $\sigma_3$. So $\sigma_1 \cap \sigma_3$ contains $\sigma_2\cup \sigma_4$, contradicting Lemma~\ref{lem:strat_cell_intersection}.

To prove the final statement, let us assume without loss of generality that $v_1$ is the vertex of maximal height and $v_3$ is the one of minimal height, i.e.\ $\height{v_1}-1=h=\height{v_3}+1$. Then we have that $\sigma_3\subseteq \sigma_2,\sigma_4 \subseteq \sigma_1$.
\endproof

In the next lemma we extend this result to higher dimensional cubes of $\dccx$. By  an \textit{edge--path} in $\dccx$ we will mean a continuous path which is entirely contained in the 1--skeleton (i.e.\ is a sequence of edges). If an edge--path $p$ goes through vertices $v_0,\dots,v_s$ of $\dccx$, we will write $p=(v_0,\dots,v_s)$; note that the sequence of vertices completely determines the sequence of edges, hence the path. We call $p$ an \textit{edge--loop} if it is a closed loop, i.e.\ $v_0=v_s$.
For an edge--path $p=(v_0,\dots,v_s)$ we define $\length p =s$  to be the \textit{length} of $p$, i.e.\ the number of edges in it. 
We also define the \textit{height} of $p$ to be $\height p=\max\{\height {v_0},\dots,\height {v_s}\}$. Notice that along each edge of $p$ the height must increase or decrease exactly by 1. 

\begin{lemma}\label{lem:cube_minmax}
Let $Q$ be a cube of $\dccx$. Then the following hold.
\begin{enumerate}
    \item There is a unique vertex $v\in Q$ of minimal height. The cell dual to it is contained in each of the cells dual to the vertices of $Q$.
    \item There is a unique vertex $w\in Q$ of maximal height. The cell dual to it  contains each of the cells dual to the vertices of $Q$.
\end{enumerate}
\end{lemma}
\proof
We prove the first statement; the second is obtained by an analogous argument.
Let $k$ be the minimal height of vertices of $Q$, and assume by contradiction that there is at least a pair of vertices of $Q$ of height $k$.
Consider an edge--path $p=(v_0,\dots,v_s)$ in $Q$ such that $\height {v_0}=\height {v_s}=k$, $v_0\neq v_s$, and such that $p$ is an edge--path of minimal height among all edge--paths in $Q$ joining a pair of vertices of height $k$. This is well--defined since the height of such a path can only be an integer between $0$ and $n$. Let $\height p=h$ be the height of $p$.

\begin{figure}[h]
\centering
\def\svgwidth{.75\columnwidth}
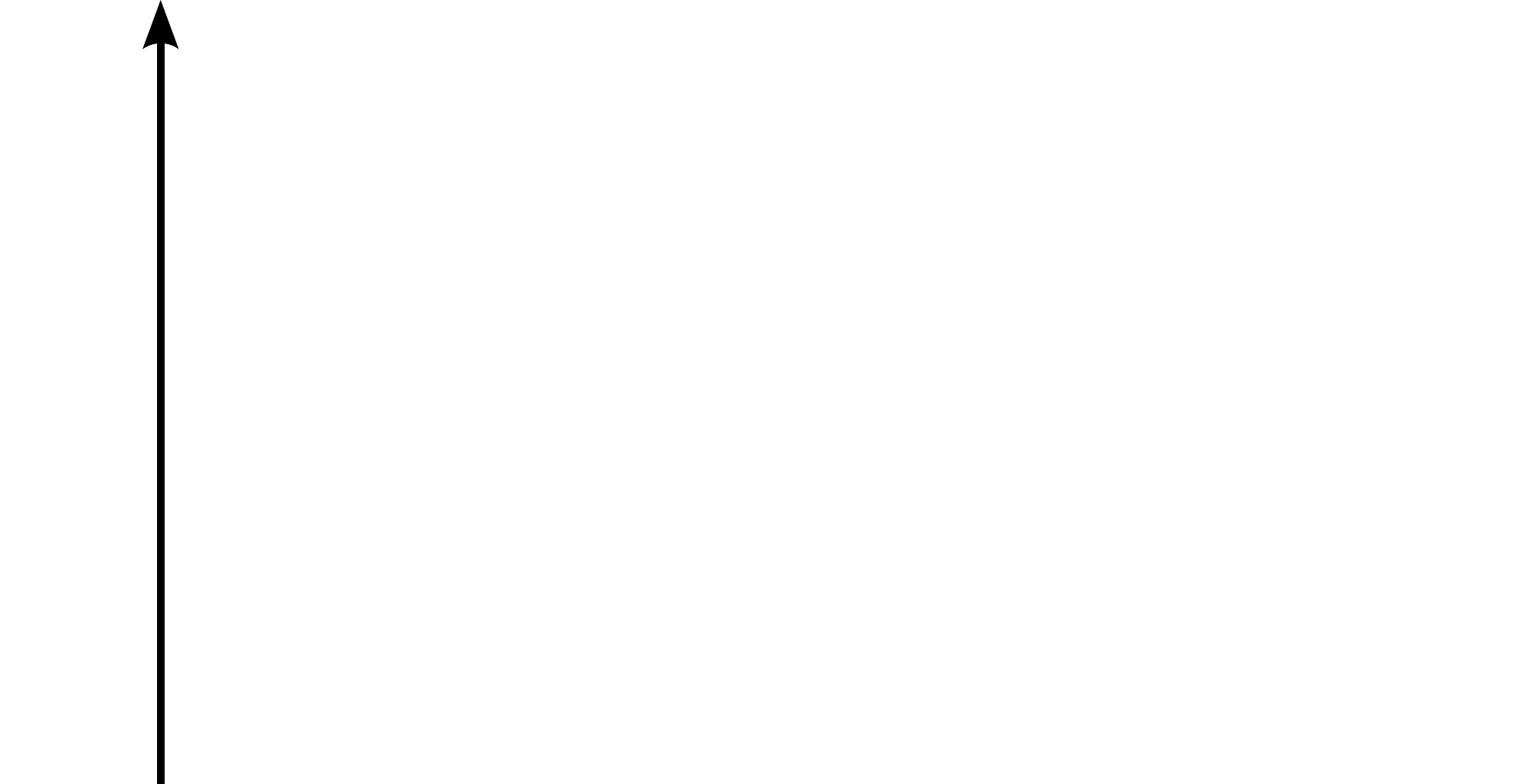
    \caption{Lowering a vertex of maximal height on a path in $\dccx$.}
    \label{fig:dcc_min_cube1}
\end{figure}

Let $v_j$ be a vertex of $p$ of maximal height $\height {v_j}=h=\height p$. Then $\height {v_{j\pm 1}}=h-1$ (notice that $k\geq 0$ and $h\geq k+1\geq 1$). Since $(v_{j-1},v_j,v_{j+1})$ is part of the cube $Q$, it must be contained in a square, i.e.\ there exists a vertex $u_j$ of $Q$ (not necessarily on $p$), such that $\{v_{j-1},v_j,v_{j+1},u_j\}$ span a square in $Q$.
Lemma~\ref{lem:height_square} implies that $\height {u_j}=h-2$. 
We can construct a new path in $Q$ starting from the path $p$ by lowering the vertex of maximal height, i.e.\ by replacing $v_j$ with $u_j$ (see Figure~\ref{fig:dcc_min_cube1}). We repeat the same operation on all vertices of height $h$ along the path, and let $p'$ be the resulting path in $Q$.
We have that $\height {p'}=h-1< h=\height {p}$, contradicting the minimality of the height of $p$. 
This concludes the proof by contradiction, and proves the uniqueness of a vertex $v$ of minimal height $k$ in $Q$.

We are left to show that the cell $\sigma$ dual to $v$ is contained in all the cells dual to the other vertices of $Q$. 
By contradiction suppose there are vertices in $Q$ whose dual cells do not contain $\sigma$; call such vertices exceptional. 
Let $p=(v_0,\dots,v_s)$ be an edge--path in $Q$ with $v_0=v$, $v_s$ an exceptional vertex, and having minimal length among all edge--paths of $Q$ between $v$ and an exceptional vertex. 
We have $\height {v_{s-1}}=\height {v_s}\pm 1$.  
If $\height {v_{s-1}}=\height {v_s}- 1$, then the cell dual to $v_{s-1}$ is contained in the cell dual to $v_s$. By minimality of $p$, we have that $v_{s-1}$ is not exceptional, so the cell dual to $v_{s-1}$ contains $\sigma$, and hence $v_s$ cannot be exceptional. Therefore $\height {v_{s-1}}=\height {v_s}+ 1$ (as in Figure~\ref{fig:dcc_min_cube2}).

\begin{figure}[h]
\centering
\def\svgwidth{.75\columnwidth}
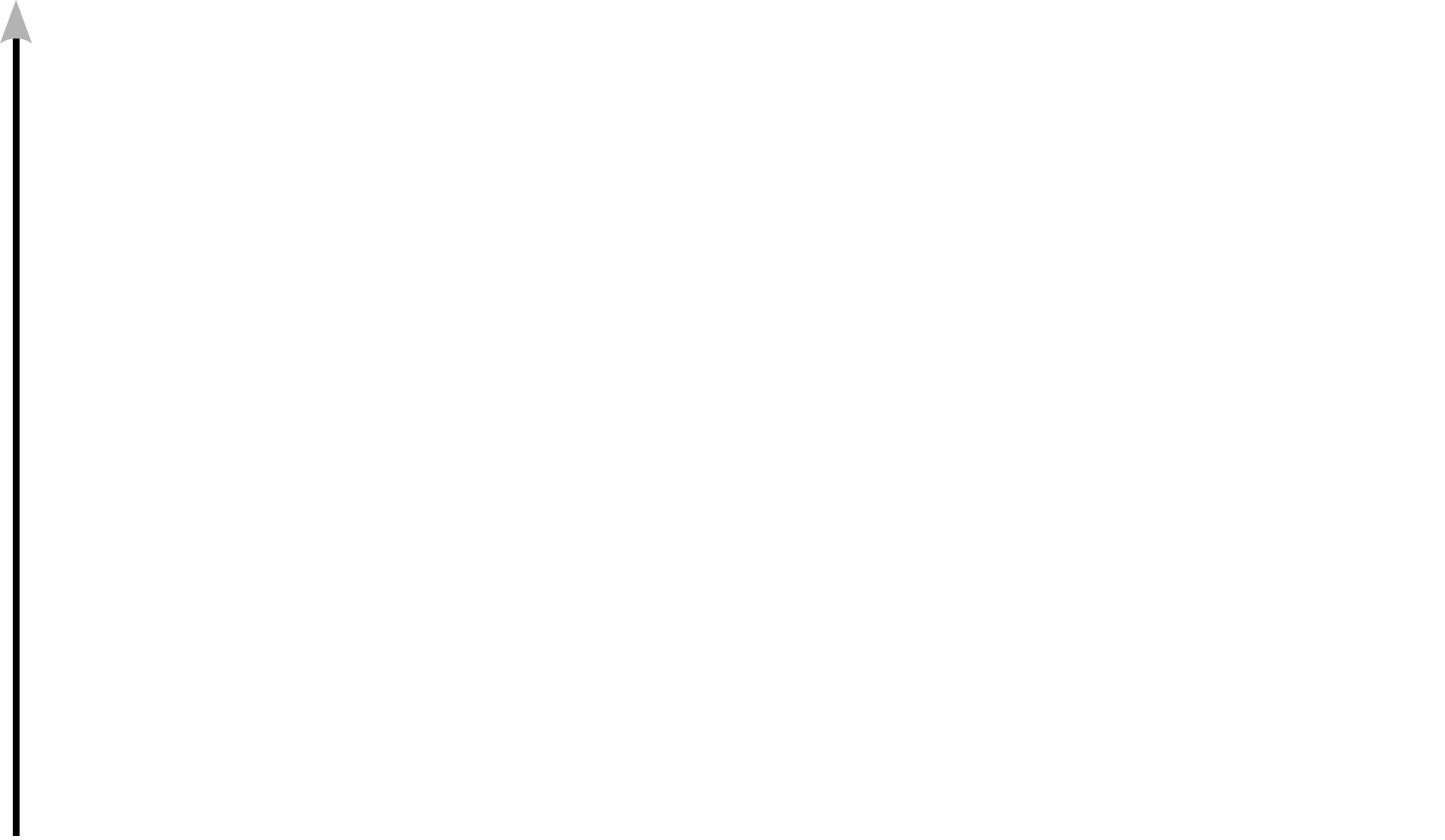
    \caption{Walking backwards along a path in $\dccx$ to find a vertex $v_j$ that can be lowered, and then walking forward to lower vertices until the exceptional endpoint $v_s$.}
    \label{fig:dcc_min_cube2}
\end{figure}

We keep walking backwards along $p$ until we find a triple of vertices $\{v_{j-1},v_j,v_{j+1}\}$ such that $\height {v_j}=\height {v_{j\pm 1}}+1$ (notice $j$ is well defined and positive, since $v$ is the unique vertex of minimal height in the whole $Q$). 
Arguing as before we complete to a square in $Q$ with vertices $\{v_{j-1},v_j,v_{j+1},u_j\}$; again by Lemma~\ref{lem:height_square} we have $\height {u_j}=\height {v_j}-2$ (see Figure~\ref{fig:dcc_min_cube2}). 
By minimality of $p$, $u_j$ must be non--exceptional, and so we can change $p$ by replacing $v_j$ with $u_j$, without changing its length. 
Walking forward along $p$, we can keep changing the path without changing its length, until we are able to complete $\{v_{s-1},v_s\}$ to a square $\{u_{s-2},v_{s-1},v_s,u_{s-1}\}$ in $Q$ with $u_{s-1}$ non-exceptional and with height $\height{u_{s-1}}=\height {v_{s-1}}-2=\height {v_s}-1$ (once again, see Figure~\ref{fig:dcc_min_cube2}). 
In particular, the cell dual to $u_{s-1}$ contains $\sigma$ and is contained in the cell dual to $v_s$, which contradicts the fact that the last vertex  $v_s$ was chosen to be exceptional. 
\endproof

As a consequence, we obtain the following statement.
\begin{lemma}\label{lem:dcc_fin_dim}
The complex $\dccx$ has dimension $\dim {\dccx}=\dim X = n$.
\end{lemma}
\proof
If $\tau $ is a tile of $\uchc$ and $x$ one of its vertices, then the collection of cells containing $x$ and contained in $\tau$ provides a cube of dimension exactly $n$, so $\dim {\dccx} \geq n$, so we focus on the other inequality.

Let $Q$ be a cube of $\dccx$, and let $v_{\operatorname{min}}$  be the vertex of minimal height in $Q$ (see Lemma~\ref{lem:cube_minmax}).
We claim that for each vertex $v\in Q$ we have
$$
\height v = \height{v_{\operatorname{min}}}+d_Q(v_{\operatorname{min}},v)
$$
where $d_Q(v_{\operatorname{min}},v)$ is the distance in $Q$ of $v$ from $v_{\operatorname{min}}$. 
Since the height can take values only  between $0$ and $n=\dim X$, this directly implies that
$$
\dim Q = \max \{ d_Q(v_{\operatorname{min}},v)\} = \max \{ \height v - \height{v_{\operatorname{min}}}\} \leq n.
$$

In order to prove the claim, pick a vertex $v\in Q$, and let $p=(v_0,\dots,v_s)$ be an edge--path of minimal length $s=d_Q(v_{\operatorname{min}},v)$ in $Q$ from $v_0=v_{\operatorname{min}}$ to $v_s=v$. 
Since the height can at most increase by $1$ along each edge of $p$, we have the inequality $\height v \leq \height{v_{\operatorname{min}}}+d_Q(v_{\operatorname{min}},v)$.
Assume by contradiction that the inequality is strict. Then the height is not monotonically increasing along $p$. Let $v_k$ be the first vertex of $p$ which is a local maximum for the height function. 
Arguing as above via Lemma~\ref{lem:height_square}, we look at the triple $v_{k-1},v_k,v_{k+1}$, and complete it to a square with a fourth vertex $u_k$ such that $\height{u_k}=\height{v_k}-2$.
We can even assume without loss of generality that $k=2$ (otherwise we proceed as in the proof of Lemma~\ref{lem:cube_minmax} and change $p$ along squares walking backwards towards $v_{\operatorname{min}}$).
But then $\height{u_2}=\height{v_2}-2=\height{v_{\operatorname{min}}}$. Minimality of $v_{\operatorname{min}}$ implies $u_2=v_{\operatorname{min}}$, and therefore we get that $v_3$ was already adjacent to $v_{\operatorname{min}}$. This provides a path from $v_{\operatorname{min}}$ to $v$ of length at most  $d_Q(v_{\operatorname{min}},v)-2$, which is a contradiction.
\endproof

We now turn to the study of links of vertices of $\dccx$. 
Recall that $\dccx$ is a cubical complex, hence its links are simplicial complexes (see \S\ref{subsec:def links}). 
In particular, if $v\in \dccx$ is a vertex, then vertices in $\lk{v,\dccx}$ correspond to vertices in $\dccx$ which are adjacent to $v$.
We begin with the following combinatorial characterization of simplices in the link of a vertex.

\begin{lemma}\label{lem:simplices_in_link}
Let $\sigma$ be a $k$-cell of $\uchc$, and let $v$ be the dual vertex in $\dccx$. 
Let $v_0,\dots,v_m$ be a collection of vertices of $\dccx$ adjacent to $v$, and let $\tau_0,\dots,\tau_m$ be the dual cells in $\uchc$. Then $v_0,\dots,v_m$ induce a simplex in $\lk {v,\dccx}$ if and only if the following two conditions are satisfied
\begin{itemize}
    \item [$( \downarrow )$] there exists a cell $\lambda$ of $\uchc$ such that $\lambda\subseteq \tau_j,j=0,\dots,m$,
    \item [$( \uparrow )$] there exists a cell $\mu$ of $\uchc$ such that $\tau_j \subseteq \mu,j=0,\dots,m$.
\end{itemize}
\end{lemma}
\proof 
First of all, note that since $\dccx$ is a cubical complex, the vertices $v_0,\dots,v_m$ induce a simplex in $\lk {v,\dccx}$ if and only if there exists a cube $Q$ of $\dccx$ containing $v,v_0,\dots,v_m$.

Assume that they induce a simplex, and let $Q$ be the corresponding cube. 
From Lemma~\ref{lem:cube_minmax} we know that $Q$ has a unique vertex of minimal height, and a unique vertex of maximal height. Let $\lambda,\mu$ be the dual cells. Lemma~\ref{lem:cube_minmax} then implies that $\lambda,\mu$ satisfy the conditions $(\downarrow)$ and $(\uparrow)$ in the statement.

Vice versa suppose that the conditions $(\downarrow)$ and $(\uparrow)$ are satisfied. 
Notice that we have $\lambda \subseteq \tau_j\subseteq \mu$ for all $j=0,\dots,m$.
Let $C_\lambda=\cdXu(\lambda)$ and $C_\mu=\cdXu(\mu)$ be the corresponding cubes of $X$, under the map $\cdXu=\cdX \circ \piuchc :\uchc \to \hc \to X$.
Notice that $\lk{\lambda,\mu}\cong \lk{C_\lambda, C_\mu}$ by  Lemma~\ref{lem:strat_links_cells}.  
In particular, we see that in $\uchc$ there is a collection of cells containing $\lambda$ and contained in $\mu$ (among which we find the cells $\tau_j$) that gives rise to a cube $Q$ in $\dccx$ containing the vertices  $v,v_0,\dots,v_m$.  Therefore $v_0,\dots,v_m$ induce a simplex in $\lk{v,\dccx}$, as desired.
\endproof

\begin{remark}
When the conditions ($\downarrow$) and ($\uparrow$) from Lemma~\ref{lem:simplices_in_link} are satisfied, the cells $\lambda,\mu$  can be chosen to be the lower and upper cell provided by Lemma~\ref{lem:min_max_cells}.
\end{remark}

\begin{remark}\label{rmk:dcc_non_locally_finite}
A cell of dimension at least $2$ in $\uchc$ always admits infinitely many codimension-1 cells (see Figure~\ref{fig:uchq_is_cell}). Lemma~\ref{lem:simplices_in_link} implies that the link of the dual vertex is neither compact nor connected. In particular the cubical complex $\dccx$ is not locally compact.
As a result, even though $\dccx $ is constructed as a sort of dual cubical barycentric subdivision with respect to the combinatorial decomposition of $\uchc$ into cells, $\dccx$ is not homeomorphic to $\uchc$. Namely, $\uchc$ is locally compact, while $\dccx$ is not locally compact.

\end{remark}

As recalled above, if $v\in \dccx$ is a vertex, then the vertices appearing in $\lk{v,\dccx}$ correspond to vertices of $\dccx$ that are adjacent to $v$, and these vertices have height equal to $\height v\pm 1$. 
We find it useful to decompose $\lk {v,\dccx}$ into two subcomplexes: we denote by $\lkd {v,\dccx}$ the full subcomplex of $\lk {v,\dccx}$ generated by vertices of height $\height v-1$, and by $\lku {v,\dccx}$  the full subcomplex of $\lk {v,\dccx}$ generated by vertices of height $\height v+1$. 
As we will see, their geometry is controlled respectively by a certain Helly property for orthogonal hyperplanes in $\hh^n$, and by the non--positive curvature of $X$.
The following statement provides the Helly property. Notice that orthogonality is a key feature here: without the orthogonality requirement, the statement already fails for three geodesics in $\hh^2$.
On the other hand, the interested reader will notice that the argument generalizes to a collection of pairwise orthogonal and totally geodesic hypersurfaces in a simply connected complete manifold of non--positive curvature. We will not need this generality in this paper.

\begin{lemma}[Helly property for orthogonal hyperplanes in $\hh^n$]\label{lem:helly_hyperbolic_space}
Let $\mathcal V$ be a collection of pairwise orthogonal hyperplanes in $\hh^n$. Then $|\mathcal V|\leq n$, and for all $k\in \{2, \dots , n\}$ all the $k$--fold intersections are non-empty.
\end{lemma}
\proof
We begin with some preliminary observation about orthogonal subspaces.
Let $V_1,\dots, V_k \in \mathcal V$ be a collection of hyperplanes from $\mathcal V$, and let $N=\cap_{j=1}^k V_j$ be their intersection.
For $x\in N$, let $\tangent{x}{\hh^n}$  denote the tangent space of $\hh^n$ at $x$, and let $v_j \in \tangent{x}{\hh^n}$ be a unit vector orthogonal to $V_j$ (i.e.\ to all vectors in the tangent space $\tangent{x}{V_j}$). 
The fact that $V_i$ and $V_j$ are orthogonal hyperplanes means that $v_i$ and $v_j$ are orthogonal vectors for all $i\neq j$.
Then a direct computation shows that if $\{n_1,\dots,n_m\}$ is an orthonormal basis for the tangent space of $\tangent{x}{N}$, then $\{n_1,\dots,n_m,v_1,\dots,v_k\}$ is an orthonormal basis for $\tangent{x}{\hh^n}$.
This shows in particular that $k\leq n$.

To prove the statement about non--emptiness of intersections, we notice that the case $k=2$ is exactly the hypothesis that any pair of hyperplanes from $\mathcal V$ intersect. 
For $k\geq 3$, we argue that if all the $h$--fold intersections of elements from $\{V_1,\dots, V_k\}$  are non--empty for all $h<k$, then the $k$--fold intersection is non--empty too.

Let $N_j=\cap_{i\neq j} V_i$. By assumption we have  $N_j\neq \varnothing$. 
Assume by contradiction that $V_1\cap \dots\cap V_k=\varnothing$. 
Then for any choice of indices $j_1\neq j_2$ we have that $N_{j_1}\cap N_{j_2}= \varnothing$. 
In particular, $N_2$ and $N_3$ are non--empty disjoint subspaces of $V_1$ (see Figure~\ref{fig:helly}). 
Let $\gamma_1$ be the common perpendicular between them in $V_1$, and let $x_k\in N_k$ be its endpoint for $k=2,3$. 
Now in the tangent space $\tangent{x_2}{\hh^n}$ we consider an orthonormal basis $\{n_1,\dots,n_m,v_1,v_3,\dots,v_k\}$ constructed as above by adding to an orthonormal basis $\{n_1,\dots,n_m\}$ for $\tangent{x_2}{N_2}$ unit vectors $v_1,v_3,\dots,v_k$ orthogonal to 
$V_1,V_3,\dots,V_k$. 
If $w$ denotes a tangent vector at $x_2$ along $\gamma_1$, then a direct computation shows that $w$ is orthogonal to 
$\{n_1,\dots,n_m\}$, because $\gamma_1$ is orthogonal to $N_2$, and it is also orthogonal to $v_1$, because $\gamma_1 \subseteq V_1$.
Therefore $w$ is in the subspace of $\tangent{x_2}{\hh^n}$ generated by $v_3,\dots,v_k$. 
If we define $W_2=\cap_{j=3}^k V_j$, then this means that $\gamma_1$ is orthogonal to $W_2$ at $x_2$.
Arguing in the same way at the point $x_3$, we find that $\gamma_1$ is orthogonal at $x_3$ to the subspace $W_3=\cap_{j=2,j\neq 3}^k V_j$. 
Note that $W_2\cap W_3=\cap_{j=2}^k V_j=N_1$ is non--empty. 
Moreover, as observed above, it is disjoint from $N_2$ and from $N_3$.
Therefore we can connect $x_2$ (respectively $x_3$) to a point $x_1$ in $N_1$ with a geodesic arc $\gamma_2$ contained in $W_2$ (respectively $\gamma_3$ contained in $W_3$).
Since all the spaces involved are totally geodesic, the arcs $\gamma_1,\gamma_2,\gamma_3$ are geodesic arcs in $\hh^n$, so we have obtained a geodesic triangle with two right angles, which leads to the desired contradiction.
\endproof 

\begin{figure}[h]
\centering
\def\svgwidth{.75\columnwidth}
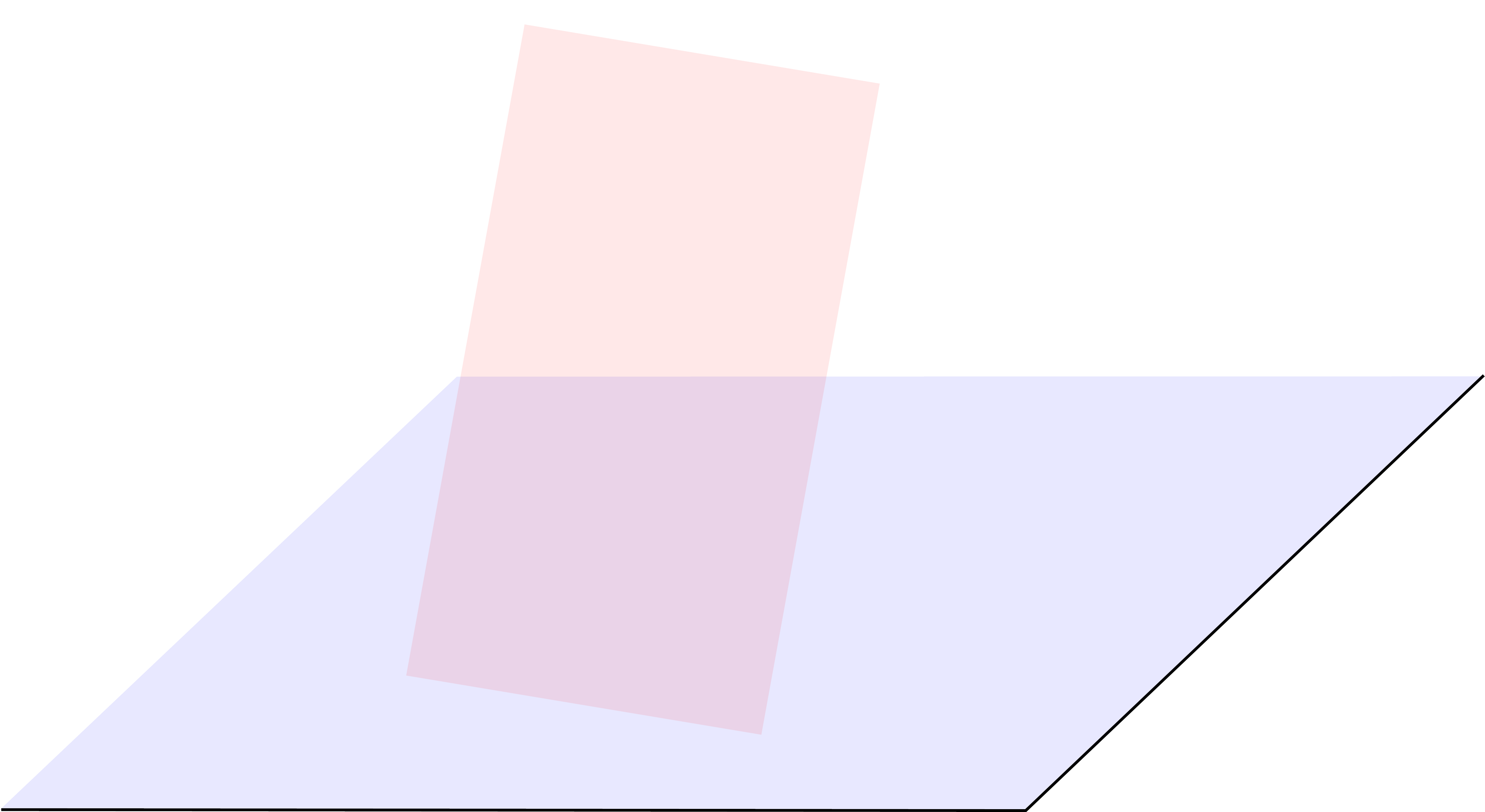
    \caption{The Helly property in Lemma~\ref{lem:helly_hyperbolic_space}.}    
    \label{fig:helly}
\end{figure}

The next statement completes our investigation of the combinatorial geometry of $\dccx$. 
Thanks to Gromov's link condition (see Lemma~\ref{lem:gromov link condition}), it already implies that $\dccx$ is locally $\cat 0$. 
We will show in Theorem~\ref{thm:dual cubical complex is CAT(0)} that it is actually $\cat 0$.

\begin{proposition}\label{prop:flag_links}
Let $\sigma$ be a $k$-cell of $\uchc$, and let $v$ be the dual vertex in $\dccx$. Then the following hold:
\begin{enumerate}
    \item \label{item:downlink} $\lkd {v,\dccx}$ is a flag simplicial complex.
    \item \label{item:uplink} $\lku {v,\dccx}$   is a flag simplicial complex.
    \item \label{item:linkflag} $\lk {v,\dccx}$  is a flag simplicial complex.
\end{enumerate}
\end{proposition}
\proof
Throughout this proof, $w_j$ will denote a vertex in $\lk {v,\dccx}$, $v_j$ the corresponding vertex of $\dccx$ adjacent to $v$, and $\tau_j$ the cell of $\uchc$ dual to $v_j$. 
Notice that two vertices $w_i,w_j$ are adjacent in $\lk {v,\dccx}$ precisely when $v,v_i,v_j$ are contained in a square of $\dccx$.  

We first prove \eqref{item:downlink}. 
Let $w_0,\dots,w_p$ be pairwise adjacent vertices in $\lkd {v,\dccx}$. 
Notice that $\tau_0,\dots,\tau_p$ are all cells of codimension 1 in the boundary of $\sigma$.
For each $i\neq j$, $v,v_i,v_j$ are contained in a square of $\dccx$.  
By Lemma~\ref{lem:height_square}, the fourth vertex of the square is dual to a cell contained in $\tau_i\cap\tau_j$. 
This shows that the cells $\tau_j$ intersect pairwise. 
We claim that actually $\tau_0\cap \dots \cap\tau_p\neq \varnothing$. 
To see this, embed $\sigma$ into a hyperbolic space of   dimension $\dim \sigma$ (as in \S\ref{sec:universal_cover}). 
The family of hyperplanes $V_0,\dots,V_p$ supporting the cells $\tau_0,\dots, \tau_p$ is a collection of pairwise orthogonal hyperplanes, and the boundary of the cell $\tau_j$ in $V_j$ is given by subspaces that are orthogonal to the other $V_i$'s.
Note that by Lemma~\ref{lem:helly_hyperbolic_space} we know that $V_0 \cap \dots \cap V_p\neq \varnothing$.
Arguing by induction on $p$, we can assume that $\tau_0\cap \dots \cap \tau_{p-1}\neq \varnothing$, and then we can leverage the orthogonality structure as in the proof of Lemma~\ref{lem:helly_hyperbolic_space} to obtain that $\varnothing \neq \tau_0\cap \dots \cap \tau_{p-1} \cap V_p \subseteq \tau_0 \cap \dots \cap \tau_p$, which proves the claim.
Since the latter intersection is non-empty, by Lemma~\ref{lem:min_max_cells},  it consists of a single cell $\lambda \subseteq \tau_j$. 
We use Lemma~\ref{lem:simplices_in_link} with this cell $\lambda$ and   $\mu=\sigma$ to conclude that $w_0,\dots,w_k$ span a  simplex.   

We argue via a dual argument to prove \eqref{item:uplink}. Let $w_0,\dots,w_p$ be pairwise adjacent vertices in $\lku {v,\dccx}$. 
Notice that  $\tau_0,\dots,\tau_p$ are cells containing $\sigma$ as a cell of codimension 1 in their boundary.
For each $i\neq j$, $v,v_i,v_j$ are contained in a square of $\dccx$.
By Lemma~\ref{lem:height_square}, the fourth vertex of the square is dual to a cell  containing   $\tau_i\cup\tau_j$. 
So  $\tau_i,\tau_j$ are adjacent in $\lk{\sigma,\uchc}$. 
By \eqref{item:strat_links_cells_cube} in Lemma~\ref{lem:strat_links_cells} this link is isomorphic to the link of the corresponding cube in $X$. 
Since $X$ is non-positively curved, this link is a flag simplicial complex.
Therefore there is a cell $\mu$ containing all the cells $\tau_j$; this can actually be taken to be the upper cell provided by Lemma~\ref{lem:min_max_cells}. 
We use Lemma~\ref{lem:simplices_in_link} with this cell $\mu$ and   $\lambda=\sigma$ to conclude that $w_0,\dots,w_k$ span a  simplex. 

Finally, in order to prove \eqref{item:linkflag}, let $w_0,\dots,w_p$ be pairwise adjacent vertices in $\lk {v,\dccx}$, ordered so that for some $m$ we have $w_0\dots,w_m \in \lkd {v,\dccx}$ and $w_{m+1}\dots,w_p \in \lku {v,\dccx}$.
By \eqref{item:downlink} we know that $w_0\dots,w_m$ span a simplex, hence by Lemma~\ref{lem:simplices_in_link} there exists a cell $\lambda$ in $\cap_{j=0}^m \tau_j$.
Similarly, by \eqref{item:uplink} we know that $w_{m+1}\dots,w_p$ span a simplex, hence by Lemma~\ref{lem:simplices_in_link} there exists a cell $\mu$ containing $\tau_{m+1},\dots,\tau_p$. 
Notice that 
$\lambda \subseteq \tau_i\subseteq \sigma \subseteq \tau_j \subseteq \mu$
for all $i=0,\dots,m$ and  $j=m+1,\dots,p$. In particular  we have
$\lambda \subseteq \tau_j \subseteq \mu$ for all $j=0,\dots,p$.
Using Lemma~\ref{lem:simplices_in_link} again we obtain that $w_0,\dots,w_p$ spans a simplex.
\endproof


\subsection{Efficiency}\label{subsec:efficiency}

In this section we study a notion of complexity for edge--paths in $\dccx$, which is based on the height function, and use it to find suitable representatives of homotopy classes of edge--paths and edge--loops.
Recall that if $p$ is an edge--path in the $1$--skeleton of a cubical complex, an \textit{elementary homotopy} of $p$ is a homotopy which is contained in the $2$--skeleton and is obtained by a finite sequence of the following two moves: 
\begin{itemize}
    \item remove a backtracking subpath, i.e.\ replace $(v_1,v_2,v_1)$ with $v_1$;
    \item slide across a square, i.e.\ replace $(v_1,v_2,v_3)$ with $(v_1,v_4,v_3)$ if $v_1,v_2,v_3,v_4$ appear in this order on the boundary of a square (as in Figure~\ref{fig:dcc_square}).
\end{itemize}

An edge--path $p=(v_0,\dots,v_s)$ is said to be \textit{efficient} if $\exists \ k\in \{0,\dots,s\}$ such that the height strictly increases from $v_0$ to $v_k$ and strictly decreases from $v_k$ to $v_s$, i.e.\ $v_k$ is the unique point of maximum for the height along $p$. 
We allow $k=0$ or $k=s$, i.e.\ that the height is strictly monotone  along $p$. 
In any case, $\height{p}=\height{v_k}$, and the cell dual to  $v_k$ contains the cells dual to all the other vertices of $p$. 
This implies that an efficient edge--path is contained in the union of at most two cubes which share at least a vertex.
In particular, an efficient edge--loop  is entirely contained in a single cube. These observations motivate the following definitions and constructions.

If $\tau$ is a tile of $\uchc$, we define the \textit{dual tile} $\dcc \tau $ to be the full subcomplex of $\dccx$ whose vertices are dual to the cells of $\tau$.
If $v$ is the vertex of $\dccx$ which is dual to $\tau$, then $\dcc \tau$ consists of all the cubes of $\dccx$ that contain $v$, i.e.\ $\dcc \tau$ is the combinatorial $1$--neighborhood of $v$.
Notice that $v$ is the only vertex of height $n$ in $\dcc \tau$ (see Figure~\ref{fig:dccx_long_path_tile}).
We say that an edge--path $p$ in $\dccx$ \textit{stays in a tile} if there exists a tile $\tau$ of $\uchc$ such that $p\subseteq \dcc \tau$.

\begin{lemma}\label{lem:efficient}
Let $p$ be an edge--path in $\dccx$. 
If $p$ stays in a tile, then there is an elementary homotopy relative to endpoints between $p$ and an efficient path.
\end{lemma}
\proof
Let $p=(v_0,\dots,v_s)$.
First of all, notice that if $s=0,1$ then $p$ is already efficient. Moreover, by an elementary homotopy relative to endpoints, we can assume that $p$ has no backtracking subpath. 
Since $p$ stays in a tile, $p$ goes through at most one vertex of height $n$ (possibly several times, possibly at the endpoints $v_0,v_s$). 

For $0<j<s$, we say $v_j$ is a local minimum (with respect to the height function along $p$) if $\height {v_{j\pm 1}}=\height {v_j}+1$, and we consider the following quantity
$$\mathfrak{h}(p)=\min\{\height {v_j} \ | \ v_j \text{ is a local minimum} \}.$$

If there is no local minimum, set $\mathfrak{h}(p)=\infty$; in this case $p$ is already efficient. 
So let us assume that there are some local minima, i.e.\ $\mathfrak{h} (p)<\infty$.
Notice that $\mathfrak{h} (p)$ is in general larger than the minimum of the height along $p$. 
If $\mathfrak{h}(p)=n$ then $p$ is constant, hence efficient. 
If $\mathfrak{h}(p)=n-1$, then $p$ has a backtracking subpath, because $p$ goes through at most one vertex of height $n$. By an elementary homotopy relative to endpoints we can remove this local minimum. Repeating this process, we obtain a path $p'$ with $\mathfrak{h}(p')=n$, and we reduce to the previous case. So let us assume in the following that $\mathfrak{h}(p)\leq n-2$.

We now claim that, by deforming $p$ locally at local minima, we can produce an elementary homotopy relative to endpoints to a path $p'$ such that $\mathfrak{h}(p')\geq \mathfrak{h}(p)+1$. 
To prove the claim, let $v_j$ be a local minimum, and let its height be $\height{p_j}=h_j$ for some $0<j<s$.
Consider the subpath $(v_{j-1},v_j,v_{j+1})$, and note that the cells dual to $v_{j-1},v_{j+1}$ meet along the cell dual to $v_j$. Since $p$ stays in a tile, there is a cell containing all these cells, namely the tile itself. 
By Lemma~\ref{lem:simplices_in_link} we get that $(v_{j-1},v_j,v_{j+1})$ is part of a square in $\dccx$, whose fourth vertex is some $v_j'$, of height $\height{v_j'}=h_j+2$. 
Then we can homotope $(v_{j-1},v_j,v_{j+1})$ to the other side $(v_{j-1},v_j',v_{j+1})$ of the square via an elementary homotopy relative to endpoints (see Lemma~\ref{lem:height_square}).
This process can be applied to all local minima at the same time, since no two local minima can be adjacent along $p$. Then we remove all backtracking subpaths, if needed, keeping endpoints fixed.
The result is an elementary homotopy relative to endpoints between $p$ and an edge--path $p'$ with $\mathfrak{h}(p')\geq \mathfrak{h}(p)+1$. It is even possible that $\mathfrak{h}(p')=\infty$ but in any case this proves the claim.

We repeat this process of elevating local minima, and after a finite number of steps we obtain a path $p''$ with $\mathfrak{h}(p'')\geq n-1$ (again, possibly $\mathfrak{h}(p'')=\infty$). Hence, we reduce to the previously discussed cases to conclude that  $p''$ (hence $p$) admits an elementary homotopy relative to endpoints to an efficient path.
\endproof

In the previous lemma we allow $p$ to be an edge--loop, i.e.\ $v_0=v_s$. All the homotopies in it are relative to the base point $v_0=v_s$. In the following statement we consider free homotopies, i.e.\ homotopies that are not required to fix any point.

\begin{corollary}\label{cor:efficient_loop}
Let $p$ be an edge--loop in $\dccx$. If $p$ stays in a tile, then there is an elementary homotopy  between $p$ and a constant path.
\end{corollary}
\proof
Pick a basepoint $v_0$ on $p$ to be a vertex of maximal height on $p$, and write $p=(v_0,\dots,v_s)$, for $v_0=v_s$. 
Apply the previous argument (from Lemma~\ref{lem:efficient}) to $p$. 
At every iteration we allow ourselves to change the basepoint on $p$ to always be a vertex of maximal height. 
At the end there can be no local minimum, hence the path is constant. 
\endproof

A simple way for an edge-loop to satisfy the condition of Corollary~\ref{cor:efficient_loop} is to be short.
Recall from \S~\ref{subsec:cubes and links} that the length $\length p$ of an edge--path  $p$ is defined to be the number of edges of $p$.

\begin{corollary}\label{cor:short loop}
Let $p$ be an edge--loop in $\dccx$. Then $\length p$ is even. Moreover, if $\length p \leq 4$ then  $p$ stays in a tile, and there is an elementary homotopy between $p$ and a constant path.
\end{corollary}
\proof
The first statement follows from the fact that if an edge $e$ has endpoints $v,w$ then $|h(v)-h(w)|=1$, so if an edge--path has odd length then the endpoints have different height.

Suppose now $\length p\leq 4$. 
If $\ell(p)=2$ then $p=(v,w,v)$ for two adjacent vertices $v,w$. In particular the cell dual to $v$ contains the one dual to $w$, or vice versa.
If $\ell(p)=4$ then $p$ is the boundary path of a square. It follows from Lemma~\ref{lem:height_square} that $p$ contains a unique point of maximal height, and that the cell dual to it contains every other cell.
In either case, there is a cell containing all the cells dual to the vertices of $p$. If $\tau$ is a tile of $\uchc$ containing that cell, then $p$ is entirely contained in $\dcc \tau$ by construction. In particular, $p$ stays in a tile, so the statement follows from Corollary~\ref{cor:efficient_loop}.
\endproof

\begin{figure}[h]
\centering
\def\svgwidth{.5\columnwidth}
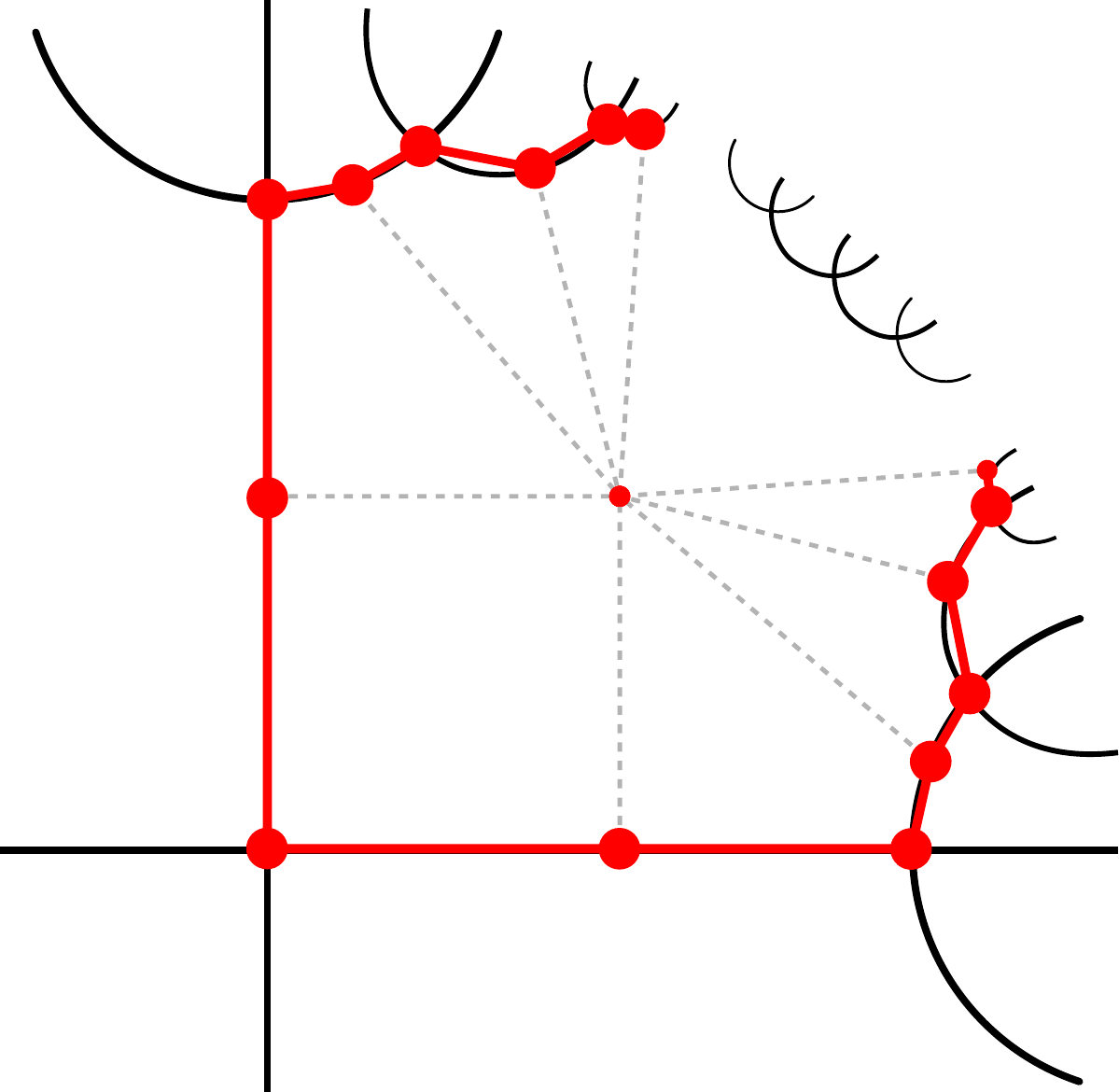
    \caption{A dual tile in $\dccx$, and a long edge--path that stays in a tile.}    
    \label{fig:dccx_long_path_tile}
\end{figure}

\begin{remark}\label{rem:length not enough}
From now on, our main goal in this section will be to show that every edge--loop in $\dccx$ can be written as a product of nullhomotopic edge--loops, i.e.\ $\dccx$ is simply connected.
A naive approach would consist in splitting an edge--loop along mirrors into shorter edge--paths, until they are short enough to be contracted (in the sense of Corollary~\ref{cor:efficient_loop}).
However, there are arbitrarily long edge--paths that stay in a tile (see Figure~\ref{fig:dccx_long_path_tile}). 
Therefore, an inductive argument based on length alone would not suffice, and this idea requires some additional tools which we develop in \S\ref{subsec:mirror complexity}, before returning to the problem of simple connectedness of $\dccx$ in \S\ref{subsec:mirror surgeries}.
\end{remark}

\begin{remark}
Given two cells $\sigma,\sigma'$ contained in the same tile $\tau$, let $\mu=\mu(\sigma,\sigma')$ be their upper cell (i.e.\ the smallest cell that contains both of them, as defined in Lemma~\ref{lem:min_max_cells}).
If $v,v'$ and $w$ are the vertices dual to $\sigma,\sigma'$ and $\mu$ respectively, then an edge--path  of minimal length in $\dccx$ from $v$ to $v'$ can be obtained as an efficient path $p$ in $\dcc \tau$ going through $w$.
Such an efficient edge--path is not unique, but the length of any such path is given by
$$\length p = 2\height w - \height v - \height{v'} = 2 \dim \mu - \dim \sigma - \dim {\sigma '}. $$
It should be noted that if $\mu \subsetneq \tau$ then there are edge--paths from $v$ to $v'$ which are strictly longer than $p$ but still efficient.
\end{remark}


\subsection{Mirror complexity}\label{subsec:mirror complexity}
Here we define an additional notion of complexity for an edge--path, based on the relative position in $\uchc$ between mirrors and the cells dual to the vertices of the edge--path.
We start with the following definition, in analogy to that of a dual tile. If $M$ is a mirror of $\uchc$, we define the \textit{dual mirror} $\dcc M$ to be the full subcomplex of $\dccx$ whose vertices are dual to the cells of $M$.
Since we have not proved yet that $\dccx$ is simply connected, a priori it is not clear that a dual mirror enjoys properties reminiscent of those of a mirror of $\uchc$; for instance, it is not clear yet whether it is convex. 
Nevertheless, we can obtain the following statement about separation (analogous to Proposition~\ref{prop:mirrors separate}).

\begin{lemma}\label{lem:dual mirror separate}
Let $M$ be a mirror of $\uchc$ and let $\dcc M$ be the dual mirror in $\dccx$.
Let $z_1,z_2$ be two points in $\uchc \setminus M$, let $\sigma_1,\sigma_2$ be cells in $\uchc$ such that $z_k\in \sigma_k$, and let $v_k$ be the vertex of $\dccx$ dual to $\sigma_k$.
Then $M$ separates $z_1$ and $z_2$ if and only if $\dcc M$ separates $v_1$ and $v_2$.
In particular, $\dcc M$ separates $\dccx$.
\end{lemma}
\proof
Suppose $M$ separates $z_1$ and $z_2$, and assume by contradiction that there is an edge--path $p$ in $\dccx$ from $v_1$ to $v_2$ avoiding $\dcc M$.
Then the union of the cells dual to the vertices of $p$ contains a path--connected subspace of $\uchc \setminus M$ that contains both $z_1$ and $z_2$. This is in contradiction with the fact that $M$ separates $z_1$ from $z_2$.

Vice versa, suppose $\dcc M$ separates $v_1$ and $v_2$, and assume by contradiction that there is a path $\gamma$ in $\uchc$ from $z_1$ to $z_2$ avoiding $M$.
By a small perturbation, we can assume that $\gamma$ intersects the strata of $\uchc$ in such a way that the sequence of the minimal cells that it visits gives rise to an edge--path in $\dccx$ (i.e.\ their dimension jumps by 1 at a time along $\gamma$). By construction, such an edge--path connects $v_1$ to $v_2$ in the complement of $\dcc M$, which is not possible.

In particular, it follows that $\dcc M$ separates $\dccx$, because $M$ separates $\uchc$ by Proposition~\ref{prop:mirrors separate}.
\endproof

This provides a correspondence between complementary components of a mirror $M$ in $\uchc$ and complementary components of the dual mirror $\dcc M$ in $\dccx$.

\subsubsection{Crossings}
Let $p=(v_0,\dots,v_s)$ be an edge--path (possibly an edge--loop) in $\dccx$, and let $\sigma_0,\dots, \sigma_s$ be the  cells of $\dccx$ dual to its vertices. 
Let $M$ be a mirror in $\uchc$, and let $\dcc M$ be the dual mirror in $\dccx$.
Recall from Lemma~\ref{lem:dual mirror separate} that $\dccx\setminus \dcc M$ is disconnected.
We say that $p$ \textit{crosses}  $M$ if $p \cap \dcc M \neq \varnothing$ and there are at least two connected components $C_1,C_2$ of $\dccx \setminus \dcc M$ such that $p \cap C_k\neq \varnothing$. 
This means that among the cells $\sigma_0,\dots, \sigma_s$, some are contained in $M$, but at least two of them are such that their interiors are contained in different complementary components of $M$. (Recall that in our setting cells are closed and complementary components of mirrors are open.)
Let $q=(v_j,\dots,v_{j+m})$ be a subpath of $p$.
We say that $q$ is a ($p,M$)--\textit{crossing} if $v_j,\dots,v_{j+m} \in \dcc M$, but $v_{j-1}$ and $v_{j+m+1}$  lie in different connected components of $\dccx \setminus \dcc M$.
(See Figure~\ref{fig:crossings} for some examples.)
We denote by $\mirrorcomplexitymirror{p}{M}$  the number of ($p,M$)--crossings. 
The \textit{mirror complexity} of $p$ is defined by taking into account the family  $\mirrors$ of all mirrors of $\uchc$, i.e.\ by the following formula:
$$
\mirrorcomplexity p= \sum_{M\in \mirrors}\mirrorcomplexitymirror{p}{M}.
$$ 

\begin{figure}[h]
\centering
\def\svgwidth{\columnwidth}
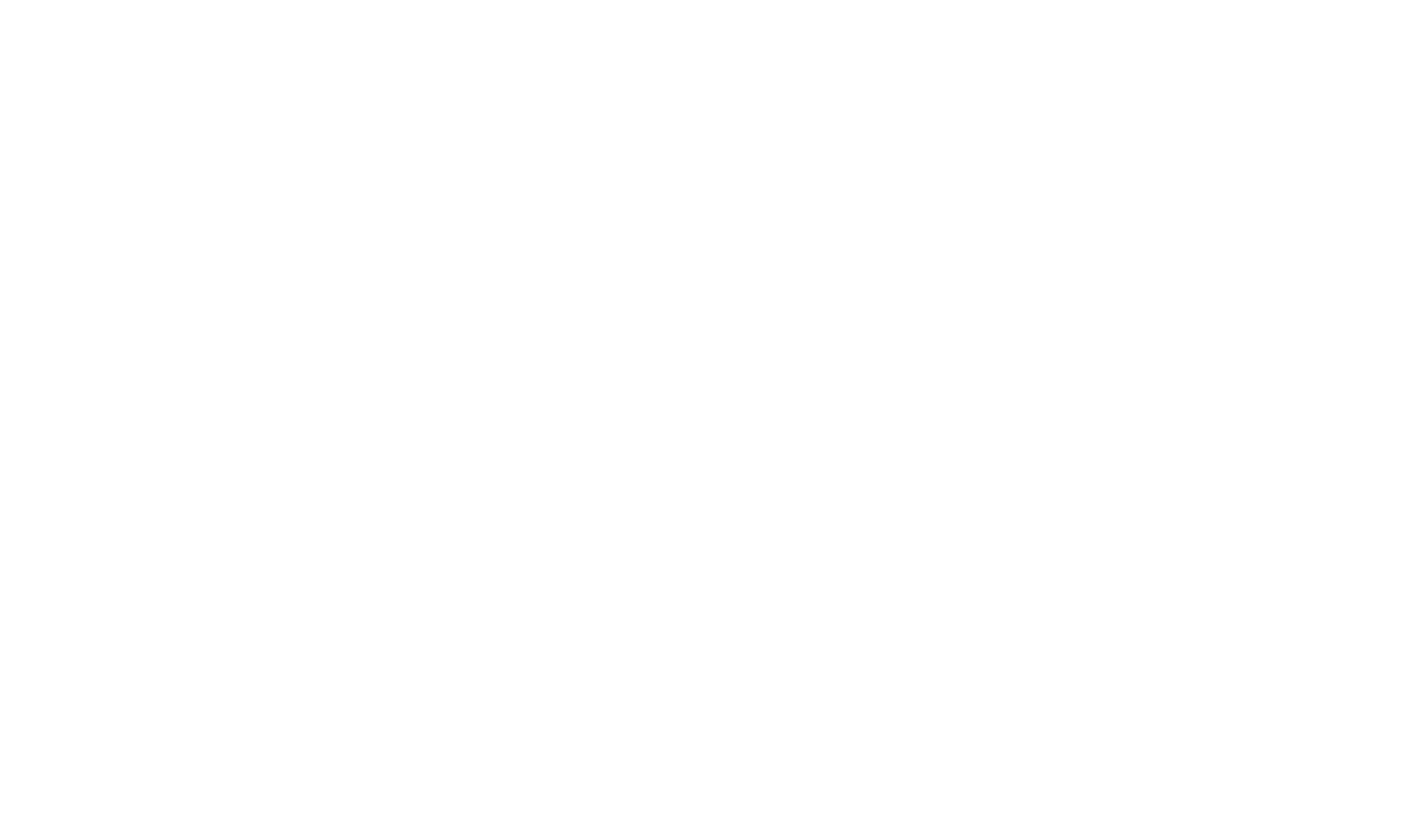
    \caption{An edge--path $p$ in $\dccx$ crossing some mirrors. Mirror crossings are highlighted. We have $\mirrorcomplexitymirror{p}{M}=2$, $\mirrorcomplexitymirror{p}{N_1}=1$, $\mirrorcomplexitymirror{p}{N_2}=3$, and $\mirrorcomplexitymirror{p}{N_3}=3$. In particular, notice that even if $p$ intersects $\dcc{N_1}$ twice, there is only one $(p,N_1)$-crossing.}    
    \label{fig:crossings}
\end{figure}

The relevance of this notion of complexity with respect to Remark~\ref{rem:length not enough} is showcased by the following two lemmas.

\begin{lemma}\label{lem:in tile iff no crossing}
Let $p$ be an edge--path in $\dccx$. Then $p$ stays in a tile if and only if $p$ does not cross any mirror.
\end{lemma}
\proof
Suppose that $p$ stays in a tile, i.e.\ there exists a tile $\tau$ such that $p \subseteq \dcc \tau$. 
Assume by contradiction that $p$ crosses a mirror $M$.
So there are two vertices $v_1, v_2$ of $p$ which are separated by $\dcc M$.
Let $\sigma_k$ be the cell dual to $v_k$.
By Lemma~\ref{lem:dual mirror separate} $M$ separates the interior of $\sigma_1$ from the interior of $\sigma_2$. 
In particular, there is no tile of $\uchc$ that contains both of them,  which contradicts the hypothesis that $p$ stays in a tile.

Vice versa, suppose $p$ does not cross any mirror, and assume by contradiction that there are two vertices $v_1,v_2$ on $p$ such that the dual cells are not contained in the same tile. 
Let  $\tau_1,\tau_2$ be different tiles containing them. 
Up to choosing $v_1,v_2$ closer to each other along $p$, we can assume that the tiles are adjacent, i.e.\ $\tau_1\cap\tau_2\neq \varnothing$.
In particular, $\sigma=\tau_1\cap\tau_2$ is a cell and it is contained in some mirror $M$.
Then $p$ intersects $M$ between $v_1$ and $v_2$.
Moreover the tiles $\tau_1,\tau_2$ provide a framing in the sense of \S\ref{subsec:separation}.  Proposition~\ref{prop:mirrors separate} implies that the interiors of $\tau_1,\tau_2$ are separated by $M$. The same holds for the interiors of the cells dual to $v_1,v_2$.
So by Lemma~\ref{lem:dual mirror separate} we have that $v_1,v_2$ are separated by $\dcc M$, i.e.\ $p$ crosses $M$, a contradiction.
\endproof

\begin{lemma}\label{lem:mirror complexity basics}
Let $p$ be an edge--path in $\dccx$, and let $M$ be a mirror in $\uchc$. Then the following hold.
\begin{enumerate}
    \item \label{item:mirror complexity basics 1} $\mirrorcomplexitymirror{p}{M}=0$ if and only if $p$ does not cross $M$.
    \item \label{item:mirror complexity basics 2} $\mirrorcomplexity p=0$ if and only if $p$ stays in a tile.
    \item \label{item:mirror complexity basics 3} If $p$ is a loop and  $\mirrorcomplexitymirror{p}{M}\geq 1$, then  $\mirrorcomplexitymirror{p}{M}\geq 2$.
    
    \item \label{item:mirror complexity finite} If $p$ has finite length, then $\mirrorcomplexitymirror{p}{M}$ and $\mirrorcomplexity p$ are finite.
\end{enumerate}
\end{lemma}
\proof
For \eqref{item:mirror complexity basics 1}, note that $\mirrorcomplexitymirror{p}{M}$ is by definition the number of ($p,M$)-crossings. For \eqref{item:mirror complexity basics 2}, note that $\mirrorcomplexity p$ is a sum of non-negative numbers, so it is zero if and only if  $\mirrorcomplexitymirror{p}{M}=0$ for every mirror $M$. By \eqref{item:mirror complexity basics 1} this is equivalent to saying that $p$ does not cross any mirror. Then the statement follows from Lemma~\ref{lem:in tile iff no crossing}.
To prove \eqref{item:mirror complexity basics 3}, note that if $p$ is a loop that crosses $M$ at least once, then it must cross it at least twice, because $\dcc M$ separates $\dccx$ by Lemma~\ref{lem:dual mirror separate}. 

Finally, to prove \eqref{item:mirror complexity finite} notice that each $(p,M)$--crossing contributes to at least one vertex of $p$, dual to a cell of $M$. Since $p$ has finite length, there can be only finitely many $(p,M)$--crossings. Then the finiteness of $\mirrorcomplexity p$ follows from the fact that $X$ (hence the collection of mirrors $\mirrors$) is locally finite.
\endproof

\begin{remark}
In this framework, Corollary~\ref{cor:short loop} can be restated by saying that $\length p \leq 4$ implies $\mirrorcomplexity p=0$.
\end{remark}

\subsubsection{Bridges}
Let $p=(v_0,\dots,v_s)$ be an edge--path in $\dccx$, and let $\sigma_0,\dots, \sigma_s$ be the  cells of $\uchc$ dual to its vertices. 
Let $M$ be a mirror in $\uchc$, and let $\dcc M$ be the dual mirror in $\dccx$.
We say that $p$ is a \textit{bridge} if there exists a mirror $M$ of $\uchc$ such that $v_0,v_s\in \dcc M$, but $p\not \subseteq \dcc M$.
In other words, $\sigma_0,\sigma_s\subseteq M$ but some of the other cells $\sigma_1,\dots,\sigma_{s-1}$ are not contained in $M$. 
In this case, we say that $p$ is \textit{supported} by $M$.
We say $p$ is a \textit{minimal bridge} if none of its subpaths is a bridge (see Figure~\ref{fig:minimal_bridge}).

\begin{lemma}\label{lem:minimal bridges exist}
Let $p$ be an edge--path in $\dccx$. If $p$ is a bridge, then there exists a subpath $q \subseteq p$ that is a minimal bridge.
\end{lemma}
\proof
Let us consider the collection of subpaths of $p$ which  are bridges.
Notice that this collection contains $p$ itself, it is partially ordered by inclusion, and it is finite.
Therefore it contains a minimal element.
\endproof

\begin{figure}[h]
\centering
\def\svgwidth{\columnwidth}
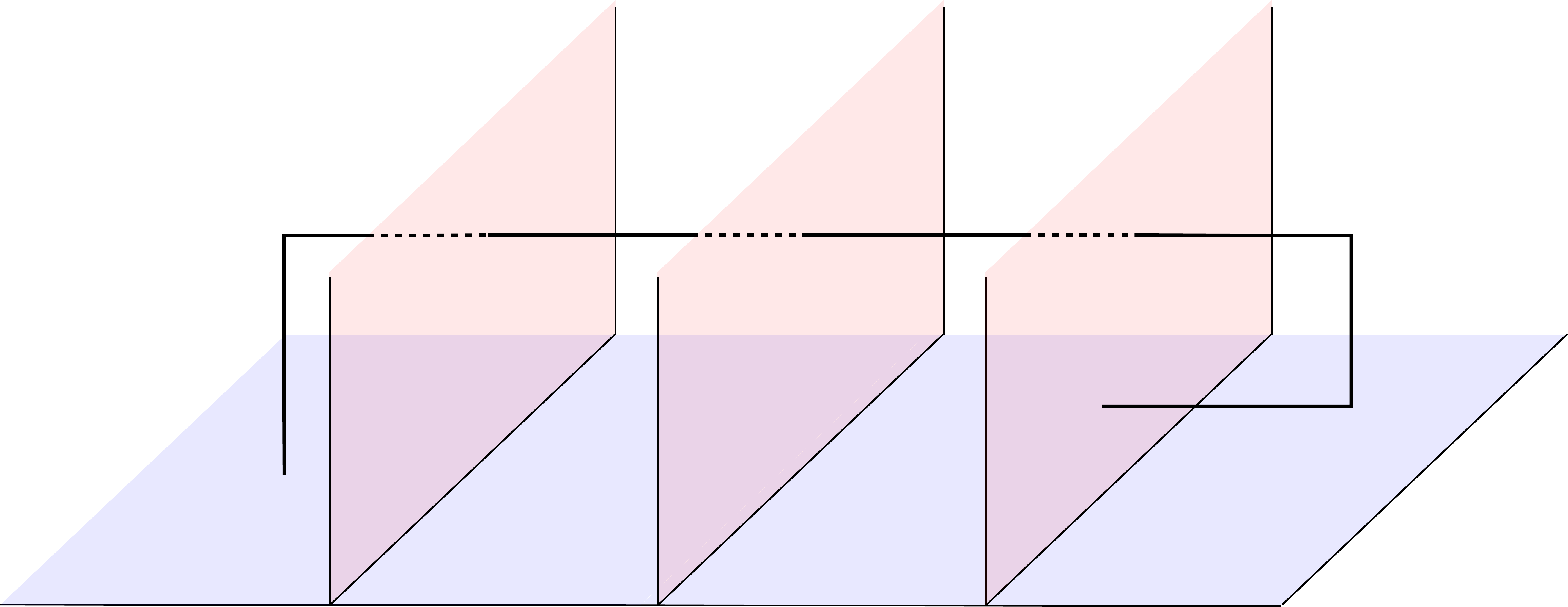
    \caption{A bridge $p$ supported by a mirror $M$, with subpaths $q$ and $q'$ which are bridges supported by the mirrors $N$ and $N'$ respectively. Notice that only $q$ is a minimal bridge.}    
    \label{fig:minimal_bridge}
\end{figure}

\begin{lemma}\label{lem:bridges projection}
Let $p$ be a minimal bridge supported by a mirror $M$, and let $N$ be a mirror such that $\mirrorcomplexitymirror{p}{N}>0$. Then the following hold.
\begin{enumerate}
    \item \label{item:bridge simple crossing} $\mirrorcomplexitymirror{p}{N} =1$.
    \item \label{item:bridge crossing implies intersection} $\dcc M \cap \dcc N \neq \varnothing$ and $M \cap N \neq \varnothing$.
\end{enumerate}
\end{lemma}
\proof
The assumption that $\mirrorcomplexitymirror{p}{N}>0$ means that $p$ crosses $N$ at least once. 
If $p$ crossed $N$ twice, then any subpath between two consecutive $(p,N)$-crossings would be a bridge supported by $N$. But this would contradict minimality, hence $\mirrorcomplexitymirror{p}{N} =1$, which proves \eqref{item:bridge simple crossing}.
In particular the endpoints of $p$ lie in different connected components of $\dccx \setminus \dcc N$. Since they also live on the same dual mirror $\dcc M$, which is connected, and dual mirrors separate by Lemma~\ref{lem:dual mirror separate}, we can conclude that $\dcc M\cap \dcc N\neq \varnothing$. Finally, the cell dual to a vertex in their intersection is contained in $M\cap N$, hence we obtain \eqref{item:bridge crossing implies intersection}.
\endproof

Recall that for a mirror $M$ in $\uchc$ we have a nearest point projection $\pi_M:\uchc \to M$, as discussed in \S\ref{subsec:graph of spaces}. 
If $p$ is a minimal bridge supported on $M$, then we can use $\pi_M$ to induce a projection of $p$ to $\dcc M$, as established by the next results.

\begin{lemma}\label{lem:intersecting mirror projection}
Let $M,N$ be mirrors in $\uchc$ and let $\tau$ be a tile in $\uchc$. 
\begin{enumerate}
    \item  If $M\cap N\neq \varnothing$, then $\pi_M(N)=M\cap N$.
    \item  If $M\cap \tau \neq \varnothing$, then $\pi_M(\tau)=M\cap \tau$.
\end{enumerate}
\end{lemma}
\proof
We start by proving $\pi_M(N) \subseteq M\cap N$.
By contradiction, let $x\in N$ be such that $\pi_M(x)\not \in N$. 
Let $y=\pi_{M\cap N}(x) \in M\cap N$, where $\pi_{M\cap N}$ denotes the nearest point projection to the closed convex subspace $M\cap N$.
Since $\pi_M(x)\not \in N$, we have $y\neq \pi_M(x)$, so we can consider the geodesic triangle with vertices $x,\pi_M(x)$ and $y$. 
By convexity of $M$, the geodesic $[\pi_M(x),y]$ is contained in $M$. By convexity of $N$, the geodesic $[x,y]$ is contained in $N$. Moreover, since $\pi_M$ is the nearest point projection to $M$, the angle between $[x,\pi_M(x)]$ and $[\pi_M(x),y]$ at $\pi_M(x)$ is at least $\frac{\pi}{2}$.  
Analogously, the angle between $[x,y]$ and any geodesic in $M\cap N$ at $y$ is at least $\frac{\pi}{2}$ too, and since $N$ meets $M$ orthogonally, this is enough to ensure that the angle between $[x,y]$ and $[\pi_M(x),y]$ at $y$ is at least $\frac{\pi}{2}$ too.
We obtained a geodesic triangle with two angles at least $\frac \pi 2$, which is impossible in the $\cat 0$ space $\uchc$.
Vice versa, if $x\in M\cap N$, then $x=\pi_M(x)$, so $x\in \pi_M(N)$ already.


The second statement can be obtained via an analogous argument. Indeed, recall that $\tau$ is isometric to a convex subset of $\hh^n$ bounded by orthogonal hyperplanes (see Lemma~\ref{lem:foldingmap_hypcomplex}). In particular, the nearest point projection to a boundary face of $\tau$ is entirely contained in $\tau$.
\endproof

The next lemma is a combinatorial statement about the stratification of $\uchc$ introduced in \S\ref{subsec:stratification}, and will be needed in the following lemma.

\begin{lemma}\label{lem:tiles intersection}
Let $\tau,\tau'$ be non-disjoint tiles of $\uchc$. Let $W_1,\dots, W_q$ be the collection of mirrors of $\uchc$ that separate $\tau$ and $\tau'$. 
Then we have that
\begin{enumerate}
    \item \label{item:tiles sep int} $W_1,\dots, W_q$ coincides with the collection of mirrors of $\uchc$ that contain $\tau \cap \tau'$.
    \item \label{item:tiles intersection cell} $\tau \cap W_1\cap \dots \cap W_q = \tau\cap \tau' = \tau' \cap W_1\cap \dots \cap W_q.$
\end{enumerate}
\end{lemma}
\begin{proof}
First of all, notice that the collection of mirrors is not empty since $\tau$ and $\tau'$ are different tiles.
We start by proving \eqref{item:tiles sep int}.
 Let $W$ be a mirror containing $\tau\cap \tau'$. Then the two tiles provide a framing for the cell $\tau\cap \tau'$. In particular we get from Proposition~\ref{prop:mirrors separate} that $W$ separates the two tiles, hence $W$ is in the collection $\{W_1,\dots,W_q\}$.
Conversely, if $\tau\cap \tau'$ was not inside one $W_i$, then we could connect the two tiles with a path that goes through the intersection but avoids $W_i$, contradicting the fact that $W_i$ separates them. 

To prove \eqref{item:tiles intersection cell} we argue as follows. By \eqref{item:tiles sep int} we know that  $\tau\cap \tau'\subseteq W_1\cap \dots\cap W_q$, so we have that $\tau\cap \tau'\subseteq \tau \cap W_1\cap \dots\cap W_q$. 
Now note that, by definition of the stratification, if $\tau\cap \tau'$ is a $k$-cell, then it must be contained in $n-k$ mirrors, so $q=n-k$. But then the two sides of the inclusion are cells of the same dimension $k$, so they must be equal. Switching the roles of $\tau$ and $\tau'$ proves the second equality in \eqref{item:tiles intersection cell}.
\end{proof}

\begin{lemma}\label{lem:projecting cells to mirrors}
Let $\tau$ be a tile and $M$ be a mirror in $\uchc$, such that $ M \cap \tau\neq \varnothing$.
Let $\sigma$ be a cell of $\tau$, and let $N_1,\dots,N_r\neq M$ be all the mirrors  containing $\sigma$ and such that $M \cap N_j \neq \varnothing$ for $j=1,\dots,r$. (Possibly $r=0$ if there are no such mirrors.)
Then the following hold.
\begin{enumerate}
    \item \label{item:cell projection} $\tau\cap M \cap N_1\cap \dots \cap N_r$ is an $(n-1-r)$--cell  that contains $\pi_M(\sigma)$.
    \item \label{item:cell projection well def} The cell $\tau\cap M \cap N_1\cap \dots \cap N_r$ only depends on $\sigma$ and $M$.
\end{enumerate}
\end{lemma}
\proof
We start by proving \eqref{item:cell projection}.
It follows from Lemma~\ref{lem:intersecting mirror projection} that $\pi_M(\sigma)\subseteq \tau \cap M\cap N_j$ for each $j=1,\dots,r$.
So,  we obtain that $\pi_M(\sigma)\subseteq \tau \cap M \cap N_1\cap \dots \cap N_r$.
To show that this intersection is a cell, note that $\tau$ is an $n$--cell.
So, by Lemma~\ref{lem:mirrors_in_tile} we have that $M\cap \tau$ is an $(n-1)$--cell and then for each $j=1,\dots,r$ we have that $\tau\cap M \cap N_1\cap \dots \cap N_j$ is an $(n-1-j)$--cell.

To prove \eqref{item:cell projection well def} we argue as follows.
Suppose $\tau'$ is another tile as in the statement, i.e.\ $\sigma \subseteq \tau'$ and $M\cap \tau'\neq \varnothing$.
Let $W_1,\dots, W_q$ be the collection of mirrors of $\uchc$ that separate $\tau$ and $\tau'$. (Note that this collection depends on $\tau$ and $\tau'$, while the collection $N_1,\dots, N_r$ only depends on $\sigma$ and $M$.)
Since $\sigma\subseteq \tau \cap \tau'$, we  also have that $\sigma$ is contained in each $W_i$ thanks to \eqref{item:tiles sep int} in Lemma~\ref{lem:tiles intersection}.
We now claim that each $W_i$ meets $M$. This is clear if $\sigma\subseteq M$. On the other hand, if $\sigma$ is not inside $M$, then we can take an efficient edge--path $p$ in $\dccx$ from the vertex dual to $\sigma$ to the vertex dual to $M\cap \tau$ which is contained in $\dcc \tau$ and meets $\dcc M$ only at the endpoint $M\cap \tau$. Take an analogous path $p'$ in $\dcc{\tau '}$, and concatenate $p$ and $p'$ to obtain a minimal bridge $\widehat p$ supported on $M$. Since $W_i$ separates $\tau$ and $\tau'$, we see that $\widehat p$ crosses $W_i$. So by \eqref{item:bridge crossing implies intersection} in Lemma~\ref{lem:bridges projection} we conclude that $M\cap W_i\neq \varnothing$, which proves the claim.

As a result, we have that the collection $\{W_1, \dots , W_q\}$ is a subcollection of $\{M,N_1, \dots , N_r\}$. (Note that $M$ could be one of the mirrors separating $\tau$ and $\tau'$, but  $N_i\neq M$ by definition.) In particular, using \eqref{item:tiles intersection cell} from \ref{lem:tiles intersection}, we obtain that
\begin{equation*}
    \tau\cap M \cap N_1\cap \dots \cap N_r \subseteq \tau  \cap W_1\cap \dots \cap W_r
    \overset{\ref{lem:tiles intersection}}{=} \tau\cap \tau' \subseteq \tau'
\end{equation*}
Therefore it follows that $ \tau\cap M \cap N_1\cap \dots \cap N_r \subseteq \tau'\cap M \cap N_1\cap \dots \cap N_r $. Reversing the roles of $\tau$ and $\tau'$ provides the other inclusion, and shows that the cell defined in \eqref{item:cell projection} does not depend on the choice of the tile.
\endproof

In the notation and setting of Lemma~\ref{lem:projecting cells to mirrors}, if $v\in \dccx$ is the vertex dual to $\sigma$, then we denote by $\pi_M(v)$ the vertex dual to the cell constructed in \eqref{item:cell projection} of the lemma, and call it the \textit{projection of $v$ to $\dcc M$}. This is well defined by \eqref{item:cell projection well def} in the same lemma. Note that in general $0\leq r\leq n-\dim \sigma $, as $\sigma$ could be contained in some mirrors that are disjoint from $M$. 
Nevertheless, this provides the desired projection to $\dcc M$ for vertices of $\dccx$ which are contained in the cubical $2$--neighborhood of $\dcc M$, i.e.\ the union of all the dual tiles corresponding to all the tiles that intersect $M$ in $\uchc$.

The content of the next two lemmas is that a minimal bridge supported by a mirror $M$ is completely contained in such a neighborhood of $\dcc M$ (see Lemma~\ref{lem:carrier minimal bridge}), so we can define a projection of a minimal bridge to $\dcc M$ (see Lemma~\ref{lem:projection control}). We note that the minimality assumption is necessary, see the difference between $q$ and $q'$ in Figure~\ref{fig:minimal_bridge}.

\begin{lemma}\label{lem:carrier minimal bridge}
Let $p$ be a minimal bridge supported on a mirror $M$. 
Then for each vertex $v$ of $p$ there exists a tile $\tau$ such that $v \in \dcc \tau$ and $\tau \cap M\neq \varnothing$.
\end{lemma}
\proof
Let $p=(v_0,\dots,v_s)$, let $\sigma_k$ be the cell of $\uchc$ dual to $v_k$, and assume by contradiction that some vertices do not satisfy the statement. 
Let $v_k$ be the first one. 
Since $p$ is a bridge, its endpoints are on $\dcc M$, so $k\neq 0,s$.
Let $\tau_\pm$ be tiles such that $v_{k- 1}\in \dcc {\tau_-}$ and $v_{k}\in \dcc {\tau_+}$. In particular this means that $\sigma_{k-1}\subseteq \tau_-$ and $\sigma_k\subseteq \tau_+$.
By construction, we can choose $\tau_-$ so that $\tau_-\cap M\neq \varnothing$, while necessarily $\tau_+$ is disjoint from $M$.
Moreover, if we had $\sigma_k \subseteq \sigma_{k-1}$ then we would have $v_k \in \dcc{\tau_-}$, against the choice of $v_k$.
But since  $v_{k-1}$ and $v_k$ are adjacent in $\dccx$, this forces $\sigma_{k-1}\subseteq \sigma_k$.
In particular, the intersection $\tau_-\cap \tau_+$ is non empty: it contains at least the cell $\sigma_{k-1}$.

Consider the cell $\sigma=\tau_-\cap \tau_+$.
For any mirror $N$ containing $\sigma$, we claim that $N$ must  intersect $M$.
Indeed, the tiles $\tau_\pm$ form a framing for $\sigma$ in the sense of \S\ref{subsec:separation}.
By  Proposition~\ref{prop:mirrors separate} we know that $\tau_\pm$ belong to the closure of distinct complementary components of $N$.
In particular, a maximal subpath of $p \cap \dcc N$ whose vertices are dual to cells contained in $\sigma$ gives rise to a $(p,N)$-crossing, hence $\mirrorcomplexitymirror{p}{N}>0$.
By \eqref{item:bridge crossing implies intersection} in Lemma~\ref{lem:bridges projection} we know $M\cap N \neq \varnothing$, as claimed.

Let $N_1,\dots, N_r$ be the collection of all mirrors containing $\sigma$. Since $\sigma$ is a cell of $\tau_-$, we can write $\sigma=\tau_-\cap N_1\cap \dots \cap N_r$.
Using \eqref{item:cell projection} in Lemma~\ref{lem:projecting cells to mirrors}, we have that
$$\pi_M(\sigma)\subseteq \tau_- \cap M  \cap N_1  \cap \dots \cap  N_r =$$
$$=(M \cap \tau_ -) \cap ( \tau_- \cap N_1  \cap \dots \cap  N_r) \subseteq M \cap \sigma \subseteq M \cap \tau_+$$
This contradicts the fact that $\tau_+$ is disjoint from $M$.
\endproof

In the next lemma we finally obtain a projection of a minimal bridge to a supporting mirror. As it might be expected, such a projection is length--decreasing (see Figure~\ref{fig:projection_minimal_bridge}).

\begin{lemma}\label{lem:projection control}
Let $p$ be a minimal bridge supported on a mirror $M$. 
Then there exists an edge--path $p^M\subseteq \dcc M$, such that $p^M$ has the same endpoints as $p$ and $\length{p^M} \leq \length p -2$.
\end{lemma}
\proof
Let $p=(v_0,\dots,v_s)$, and let $\sigma_0,\dots,\sigma_s$ be the cells dual to its vertices. 
Since $p$ is a minimal bridge supported on $M$, by Lemma~\ref{lem:carrier minimal bridge} we know that for each vertex $v_k$ there exists a tile $\tau_k$ of $\uchc$ such that $v_k\in \dcc{\tau_k}$ and $\tau_k\cap M\neq \varnothing$.
Let $w_k=\pi_M(v_k)$ be the projection of $v_k$ to $\dcc M$, constructed in  Lemma~\ref{lem:projecting cells to mirrors}.
We claim that  for each $k$ the vertices $w_{k-1}$ and $w_k$ are either the same vertex or adjacent vertices.

To see this, consider two vertices $v_{k-1}$ and $v_k$ adjacent along $p$. Without loss of generality (i.e., possibly reversing the orientation of $p$) we can assume that $\sigma_{k-1}$ is a cell of codimension $1$ in $\sigma_{k}$.
In particular, we can take $\tau_{k-1}=\tau_k$, and there is exactly one mirror $\widehat N_k$ that contains $\sigma_{k-1}$ but not $\sigma_{k}$.
Let $\{N_1,\dots,N_r\}$ be the collection of all the mirrors that contain $\sigma_k$ and intersect $M$, but are different from $M$.
Then the analogous collection for $\sigma_{k-1}$ consists of the same mirrors, possibly with the addition of $\widehat N_k$.
(Note that since $p$ is a minimal bridge supported on $M$, any mirror containing $\sigma_1,\dots,\sigma_{s-1}$ is guaranteed to be different from $M$, while $\widehat N_k=M$ for $k=1$.)
By \eqref{item:cell projection} in Lemma~\ref{lem:projecting cells to mirrors} we have that  
$\pi_M(\sigma_k)\subseteq \tau_k\cap M \cap N_1\cap \dots \cap N_r$ 
and that either 
$\pi_M(\sigma_{k-1})\subseteq \tau_k\cap M \cap N_1\cap \dots \cap N_r$ 
or 
$\pi_M(\sigma_{k-1})\subseteq \tau_k\cap M \cap N_1\cap \dots \cap N_r \cap \widehat N_k$.
In the first case we have that $\pi_M(\sigma_{k-1})$ and $\pi_M(\sigma_{k})$ are contained in the intersection of the same mirrors, hence $w_{k-1}=w_k$; in the second case $\tau_k\cap M \cap N_1\cap \dots \cap N_r \cap \widehat N_k$ is a codimension-$1$ cell of $\tau_k\cap M \cap N_1\cap \dots \cap N_r$, hence $w_{k-1}$ is adjacent to $w_k$. 
This proves the claim. 

Notice in particular that in the case $k=1$ we have $\widehat N_k=M$, so we have proved that $w_0=w_1$. Analogously, we also have $w_s=w_{s-1}$.
As a result, $(w_0,\dots,w_s)$ is an edge--path in $\dcc M$.
Let $p^M$ be the edge--path obtained from $(w_0,\dots,w_s)$ by removing all backtracking subpaths and repeated vertices.
In particular, since $w_0=w_1$ and $w_s=w_{s-1}$, we have that $\length{p^M}\leq s-2 = \length p -2$.
Moreover,  since $p$ is a bridge supported on $M$, we have that $\sigma_0, \sigma_s\subseteq M$, so that $v_0=w_0,v_s=w_s$, i.e.\ $p$ and $p^M$ have the same endpoints.
\endproof

\begin{figure}[h]
\centering
\def\svgwidth{\columnwidth}
\begingroup%
  \makeatletter%
  \providecommand\color[2][]{%
    \errmessage{(Inkscape) Color is used for the text in Inkscape, but the package 'color.sty' is not loaded}%
    \renewcommand\color[2][]{}%
  }%
  \providecommand\transparent[1]{%
    \errmessage{(Inkscape) Transparency is used (non-zero) for the text in Inkscape, but the package 'transparent.sty' is not loaded}%
    \renewcommand\transparent[1]{}%
  }%
  \providecommand\rotatebox[2]{#2}%
  \newcommand*\fsize{\dimexpr\f@size pt\relax}%
  \newcommand*\lineheight[1]{\fontsize{\fsize}{#1\fsize}\selectfont}%
  \ifx\svgwidth\undefined%
    \setlength{\unitlength}{2544.61889072bp}%
    \ifx\svgscale\undefined%
      \relax%
    \else%
      \setlength{\unitlength}{\unitlength * \real{\svgscale}}%
    \fi%
  \else%
    \setlength{\unitlength}{\svgwidth}%
  \fi%
  \global\let\svgwidth\undefined%
  \global\let\svgscale\undefined%
  \makeatother%
  \begin{picture}(1,0.46359536)%
    \lineheight{1}%
    \setlength\tabcolsep{0pt}%
    \put(0.05159611,0.09543082){\color[rgb]{0,0,1}\makebox(0,0)[lt]{\lineheight{1.25}\smash{\begin{tabular}[t]{l}$\dcc M$\end{tabular}}}}%
    \put(0,0){\includegraphics[width=\unitlength,page=1]{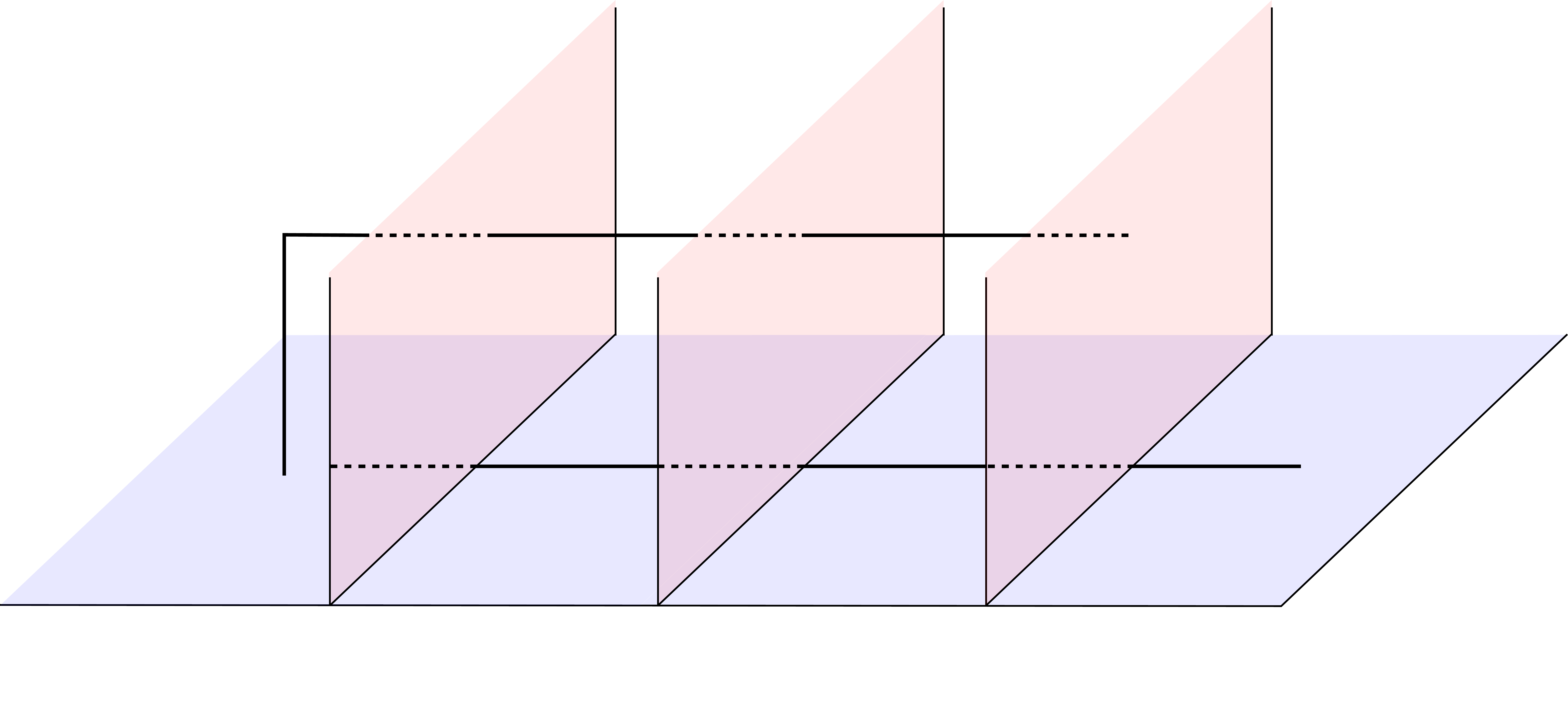}}%
    \put(0.40528107,0.33711708){\color[rgb]{0,0,0}\makebox(0,0)[lt]{\lineheight{1.25}\smash{\begin{tabular}[t]{l}$p$\end{tabular}}}}%
    \put(0.31096457,0.12490456){\color[rgb]{0,0,0}\makebox(0,0)[lt]{\lineheight{1.25}\smash{\begin{tabular}[t]{l}$p^M$\end{tabular}}}}%
    \put(0.88156205,0.34549883){\color[rgb]{0,0,0}\makebox(0,0)[lt]{\lineheight{1.25}\smash{\begin{tabular}[t]{l}$\pi_M$\end{tabular}}}}%
    \put(0,0){\includegraphics[width=\unitlength,page=2]{projection_minimal_bridge.pdf}}%
  \end{picture}%
\endgroup%

    \caption{Projection of the minimal bridge $p$ to $\dcc M$, for a supporting mirror $M$.}    
    \label{fig:projection_minimal_bridge}
\end{figure}


\subsection{Surgeries on edge--loops}\label{subsec:mirror surgeries}
We are now ready to apply the above technology to the study of edge--loops in $\dccx$. 
The goal is to show that $\dccx$ is simply connected. 
The strategy will be to reduce the length and mirror complexity of an edge--loop enough to ensure that it stays in a tile, so that Corollary~\ref{cor:efficient_loop} can be applied.
The following statement is the key surgery step. 
Roughly speaking, we chop an edge--loop $p$ along a mirror $M$ that it crosses, use the projection $p^M$ of $p$ to $M$ to introduce a shortcut along $M$ and obtain two edge--loops $p_1,p_2$ such that $p$ and $p_1p_2$ are elementary homotopic, and finally then check that the lengths have dropped.

\begin{proposition}\label{prop:mirror complexity reduction}
Let $p$ be an edge--loop in $\dccx$. 
If  $\mirrorcomplexity p >0$, then there exist two edge--loops $p_1,p_2$ in $\dccx$ such that $\length{p_1},\length{p_2} < \length p$, and there is an elementary homotopy between $p$ and $p_1p_2$
\end{proposition}
\proof
By assumption, there is a mirror $M_0$ that is crossed by $p$, so  $\mirrorcomplexitymirror{p}{M_0}\geq 1$, hence by \eqref{item:mirror complexity basics 3} in Lemma~\ref{lem:mirror complexity basics} we have that $\mirrorcomplexitymirror{p}{M_0}\geq 2$, i.e.\ $p$ crosses $M_0$ at least twice.
It follows from the definitions that any subarc of $p$ between any two $(p,M_0)$-crossings is a bridge supported by $M_0$.

Choose a $(p,M_0)$--crossing and a bridge $q$ supported by $M_0$ (in general this cannot be chosen to be minimal, see Figure~\ref{fig:surgery}). Let $q'$ be the complement of $q$ in $p$, i.e.\ the edge--path such that $p=qq'$. Of course  we have 
\begin{equation}\label{eq:L+L' extra 1}
    \length q+\length{q'}=\length p
\end{equation}
Moreover, without loss of generality we can assume that 
\begin{equation}\label{eq:L<L'}
    \length q \leq \length{q'}.
\end{equation}

By Lemma~\ref{lem:minimal bridges exist} we can find a  minimal bridge $q_1\subseteq q\subseteq p$. In particular we have 
\begin{equation} \label{eq:bridge<L}
    \length{q_1}\leq \length q.
\end{equation}
Let $q_2$ be the complement of $q_1$ in $p$, i.e.\ the edge--path such that $p=q_1q_2$. 
We can compute that   
\begin{equation} \label{eq:retrace}
\length{q_2} = 
\length p - \length{q_1} \overset{\eqref{eq:bridge<L}}{\geq}
\length p - \length q \overset{\eqref{eq:L+L' extra 1}}{=} 
\length{q'}
\overset{\eqref{eq:L<L'}}{\geq}
\length q
\overset{\eqref{eq:bridge<L}}{\geq}
\length{q_1}.
\end{equation}

Let $M$ be a mirror supporting the minimal bridge $q_1$, and let $q_1^M$ be the projection of $q_1$ to $\dcc M$, i.e.\ the edge--path obtained in Lemma~\ref{lem:projection control}.
In particular we have
\begin{equation}\label{eq:shortcut}
    \length{q_1^M} \leq \length{q_1} -2 < \length{q_1}.
\end{equation}

Define the edge--loops $p_1=q_1\overline{q_1^M}$ and $p_2=q_1^Mq_2$, where $\overline{q_1^M}$ denotes the edge--path $q_1^M$ with the opposite orientation. 
There is an elementary homotopy between the edge--loops $p=q_1q_2$ and $p_1p_2=q_1\overline{q_1^M}q_1^M q_2$, obtained by removing the backtracking subpath $\overline{q_1^M}q_1^M$.
We can compute the desired inequality on the length of $p_1$ and $p_2$ as follows: 
\begin{equation*}
\length{p_1}  = 
\length{q_1} + \length{q_1^M} \overset{\eqref{eq:shortcut}}{<}
\length{q_1} + \length{q_1}
  \overset{\eqref{eq:retrace}}{\leq} 
\length{q_1} + \length{q_2} = \length{p}.
\end{equation*}
\begin{equation*}
    \length{p_2} = \length{q_1^M} + \length{q_2} \overset{\eqref{eq:shortcut}}{<}
    \length{q_1} +\length {q_2} = \length p .
\end{equation*}
\endproof

\begin{figure}[h]
\centering
\def\svgwidth{\columnwidth}
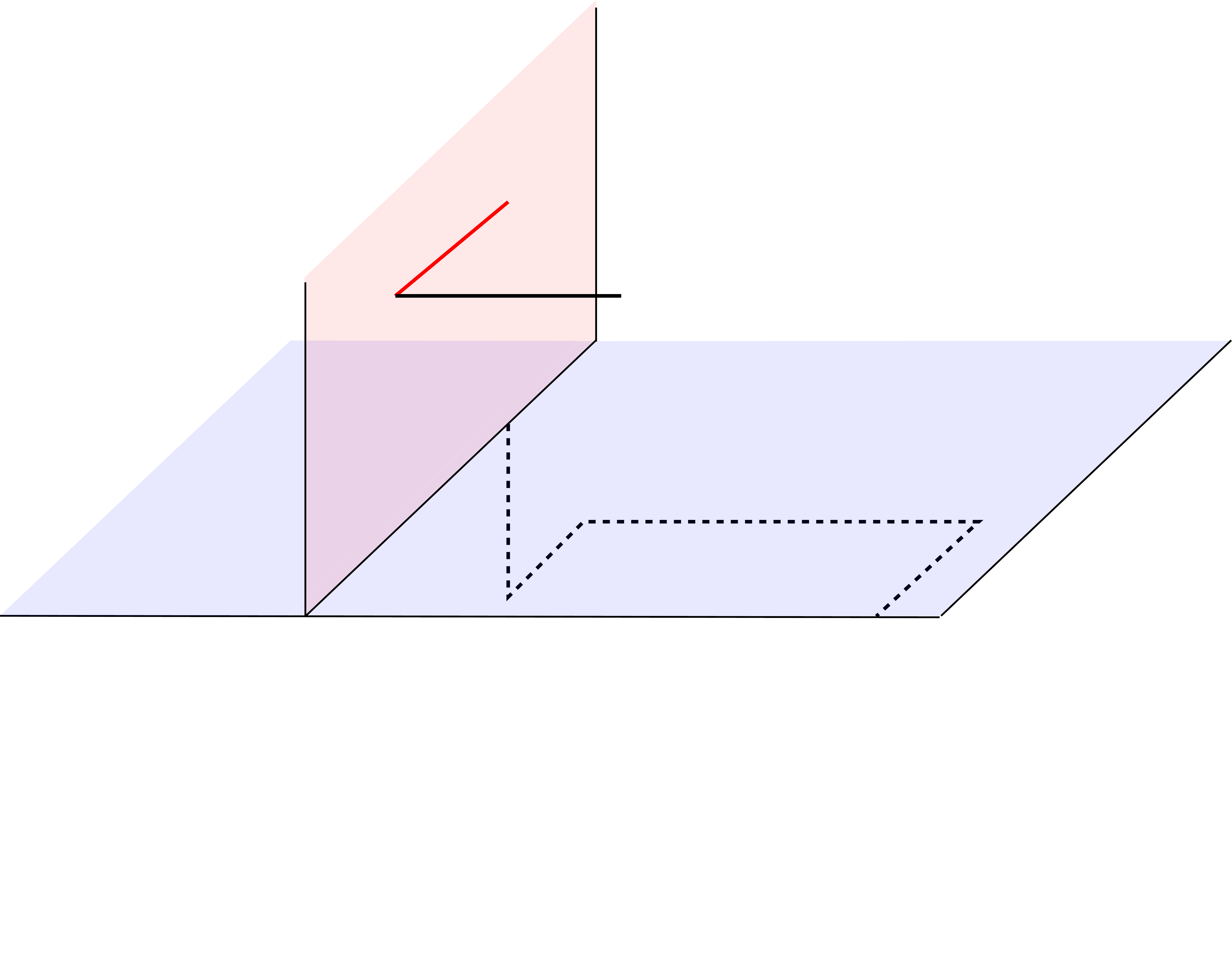
    \caption{The surgery in Proposition~\ref{prop:mirror complexity reduction}. Note that all the minimal bridges in $p$ are supported on $M$ but $p$ does not cross $M$, hence the need to first split $p$ into $q$ and $q'$, and then into $q_1$ and $q_2$.}    
    \label{fig:surgery}
\end{figure}

\begin{remark}
Note that it is possible to have an edge--loop $p$ for which the surgery  from Proposition~\ref{prop:mirror complexity reduction} strictly reduces the mirror complexity for only one subloop. For an example see Figure~\ref{fig:surgery}, where the mirror complexity of the loop $p_2$ is the same as that of the original loop $p$. 
\end{remark}

We are now ready to prove the main result of this section.

\begin{theorem}\label{thm:dual cubical complex is CAT(0)}
The complex $\dccx$ is a connected $\cat 0$ cubical complex.
\end{theorem}
\proof
By construction, the complex $\dccx$ is a cubical complex. Moreover, Gromov's link condition from Lemma~\ref{lem:gromov link condition} implies that $\dccx$ is non--positively curved, since the link of any vertex is a flag simplicial complex by  Proposition~\ref{prop:flag_links}.

Next, $\dccx$ is path--connected because $\uchc$ is path--connected. Indeed, let $v,w$ be vertices in $\dccx$, and let $\sigma_v,\sigma_w$ be the dual cells. Pick any continuous path $\eta$ connecting the two cells in $\uchc$, and keep track of the list of cells that are intersected by $\eta$. By isotoping $\eta$ into lower-dimensional cells, we can ensure that the difference between the dimension of two consecutive cells in this list is exactly $1$. The dual vertices in $\dccx$ give rise to an edge--path from $v$ to $w$.

To conclude, we need to show that $\dccx$ is simply connected.
To do this, we argue that edge--loops are nullhomotopic by induction on their length. 
Let $p$ be an edge--loop in $\dccx$, homotopically non-trivial and of minimal length. 
If $p$ does not cross any mirror, then by Lemma~\ref{lem:in tile iff no crossing} $p$ stays in a tile.
Hence by Corollary~\ref{cor:efficient_loop} there is an elementary homotopy between $p$ and a constant path.
So let us assume that $\mirrorcomplexity p > 0$.
Then by Proposition~\ref{prop:mirror complexity reduction} there exist two edge--loops $p_1,p_2$ in $\dccx$ such that $p$ is homotopic to $p_1p_2$ and  $\length{p_1},\length{p_2} < \length p$.
By minimality, both $p_1$ and $p_2$ are homotopically trivial, and so is $p$.
\endproof


We conclude this section by noting that the action of the hyperbolized group $\hg=\pi_1(\hc)$ on $\uchc$ induces an action on the dual cubical complex $\dccx$.

\begin{lemma}\label{lem:dual action is cocompact}
The group $\hg=\pi_1(\hc)$ acts on $\dccx$ by cubical isometries. Moreover, if $X$ is compact, then the action is cocompact.
\end{lemma}
\proof
The group $\hg$ acts on $\uchc$ preserving the family of mirrors, hence the stratification defined in \S\ref{subsec:stratification}.
The action permutes the cells, so $\hg$ acts on vertices of the dual cubical complex $\dccx$ described in \S\ref{sec:dual cubical complex}. 
Moreover, the action of $\hg$ on $\uchc$ preserves the incidence relations between cells, hence we can extend the action of $\hg$ on vertices to a combinatorial action of $\hg$ on the entire  $\dccx$.
Since $\hg$ acts on $\dccx$ combinatorially, it preserves the standard cubical metric. 

When $X$ is compact, $\hc$ is compact as well,  by \eqref{item:hypcell compact} in Lemma~\ref{lem:cell_is_hyperbolizing_cell}. 
The action of $\hg$ on $\uchc$ has finitely many orbits of cells, so its action on $\dccx$ has finitely many orbits of vertices. Since  $\dccx$ is finite-dimensional (see Lemma \ref{lem:dcc_fin_dim}), the quotient has finitely many cubes, so it is compact.
\endproof


\section{Special cubulation}\label{sec:special cubulation}
The purpose of this section is to study the action of the hyperbolized group $\hg=\pi_1(\hc)$ on the dual cubical complex $\dccx$ (see Lemma~\ref{lem:dual action is cocompact}), and prove that the group $\hg$ is virtually compact special in the sense of \cite{HW08}. 
When $X$ is compact and admissible,  $\hg$ is a Gromov hyperbolic group and $\dccx$ is a $\cat 0$ cubical complex (see \eqref{item:CD hyperbolic} in Proposition~\ref{prop:CD complex}, and Theorem~\ref{thm:dual cubical complex is CAT(0)}). 
Therefore, one could hope to obtain virtual specialness directly from Agol's result from \cite{AG13}.
However, the action of $\hg$ on $\dccx$ is not proper (see Remark~\ref{rem:dual action is not proper}). 
To remedy this, we will use a result of Groves and Manning from \cite[Theorem D]{GM18} that deals with improper actions.
This requires a study of stabilizers of cubes.

In \S \ref{subsec:action} we show that cube stabilizers for the action of $\hg$ on $\dccx$ coincide with cell stabilizers for the action of $\hg$ on $\uchc$. Then in \S\ref{subsec:quasiconvex and vcs} we show that such stabilizers are quasiconvex and virtually compact special.
The complex $X$ is always assumed to be admissible in the sense of \S\ref{sec:strict hyperbolization}. In some statements (such as Theorem~\ref{thm:main_vcs}) we also assume that it is compact.

\begin{remark}[Why we consider the action on $\dccx$ instead of $\widetilde X$]
It is worth noting that when $X$ is admissible, $\widetilde X$ is already a $\cat 0$ cube complex.
Moreover the map $\cdX:\hc\to X$ from Proposition~\ref{prop:CD complex} induces a surjection $\hg \to \pi_1(X)$ that can be used to obtain an action of $\hg$ on $\widetilde X$.
However, this action has a very large kernel. For example, in the case in which $X$ is already simply connected the map 
$\hg \to \pi_1(X)$ is trivial, but $\hg$ is an infinite group; indeed, it retracts to $\hgq$, as discussed in Remark~\ref{rem:what_uchc_isnot}.
\end{remark}

\subsection{Cube stabilizers for the action of  \texorpdfstring{$\hg$}{the hyperbolized group} on \texorpdfstring{$\dccx$}{the dual cubical complex}}\label{subsec:action}

In this section we relate the cube stabilizers for the action of $\hg$ on $\dccx$ to the cell stabilizers for the action of $\hg$ on $\uchc$.

\begin{lemma}\label{lem:vertex stab equals cell stab}
The stabilizer of a vertex in $\dccx$ coincides with the stabilizer of its dual  cell in $\uchc$.
\end{lemma}

\begin{remark}\label{rem:dual action is not proper}
It follows from Lemma~\ref{lem:vertex stab equals cell stab} that the action of $\hg$ on $\dccx$ is in general not proper.
Namely, the stabilizer of a vertex dual to a cell of dimension at least $2$ is infinite (compare Remark~\ref{rmk:dcc_non_locally_finite} and Figure~\ref{fig:uchq_is_cell}).
\end{remark}

We now proceed to the study of stabilizers of higher--dimensional cubes for the action of $\hg$ on $\dccx$.
Recall that the dual cubical complex $\dccx$ is equipped with a $\hg$-invariant height function: the vertex dual to a $k$-cell has height $k$. We proved in Lemma~\ref{lem:cube_minmax} that every cube in  $\dccx$ has a unique vertex of minimal height.

\begin{lemma}\label{lem:cube stab equals min vertex stab}
The stabilizer of a cube in $\dccx$ coincides with the stabilizer of its vertex of minimal height in $\uchc$. 
\end{lemma}
\proof
Let $C$ be a cube in $\dccx$, let $v$ be its vertex of minimal height, and let $\sigma$ be the dual cell in $\uchc$.
Let $g\in \hg$ be an element that stabilizes $C$. 
Since the height function is invariant, $g$ must fix $v$, by uniqueness of the vertex of minimal height.

Conversely let $g$ fix $v$. By Lemma~\ref{lem:vertex stab equals cell stab} we get that $g$ stabilizes $\sigma$, i.e.\ $g.\sigma=\sigma$.
Let $w$ be another vertex of $C$ and let $\tau$ be the dual cell. 
By Lemma~\ref{lem:cube_minmax}, we have that $\sigma \subseteq \tau$. 
It follows that $\sigma =g.\sigma \subseteq g.\tau$, so that both $\tau$ and $g.\tau$ appear in the link of $\sigma$ in the combinatorial structure of $\uchc$ (see \eqref{item:strat_links_cells_simplicial} in Lemma~\ref{lem:strat_links_cells}). 
Since the covering projection $\piuchc:\uchc \to \hc$ induces  isomorphisms on links, if $\tau$ and $g.\tau$ were distinct, then in $\hc$ we would see a tile $\pi(\tau)=\pi(g.\tau)$ isometric to a copy of $\hq$ with some identification along the boundary (namely along the subspace corresponding to $\pi(\sigma)$).
However, by \eqref{item:hypcell homeo} in Lemma~\ref{lem:cell_is_hyperbolizing_cell} we know that tiles of $\hc$ are embedded copies of $\hq$, so we must have $g.\tau = \tau$.
By Lemma~\ref{lem:vertex stab equals cell stab}, this means $g.w=w$. Therefore $g$ fixes $C$ pointwise.
\endproof

\begin{remark}
In the proof of Lemma~\ref{lem:cube stab equals min vertex stab} we  established that the stabilizer of a cell in $\dccx$ actually fixes the cell pointwise.
\end{remark}


\subsection{Cell stabilizers are quasiconvex and virtually compact special}\label{subsec:quasiconvex and vcs}

The goal of this section is to study the stabilizers of cells for the action of $\hg$ on $\uchc$ by covering transformations. 
In particular, note that by Lemma~\ref{lem:foldingmap_hypcomplex} stabilizers of tiles (i.e.\ top--dimensional cells) are isomorphic to the fundamental group of the hyperbolizing cube $\hgq=\pi_1(\hq)$.
More precisely, our goal is to show that cell stabilizers for the action of $\hg$ on $\uchc$ are quasiconvex in $\hg$, and virtually compact special.

\subsubsection{Quasiconvexity}

In the following we say that an action of a group on a metric space is \textit{geometric} if it is proper, cocompact and isometric. 
We will make use of the following standard fact.

\begin{lemma}\label{lem:quasiconvex_stabilizers}
Let $Z$ be a proper Gromov hyperbolic geodesic space, and let $Y\subseteq Z$ be a quasiconvex subset. Let $G$ be a finitely generated group acting geometrically on $Z$, and let $H$ be the stabilizer of $Y$ in $G$. 
If $H$ acts cocompactly on $Y$, then $H$ is quasiconvex in $G$.  
\end{lemma}

We apply this lemma to the cases $G=\hg, H=\hgq$ and $G=\Gamma, H=\hgq$. As noted, $\hg$ is a Gromov hyperbolic group when $X$ is compact.
In both cases, before using the lemma we need to ensure that $H$ is a subgroup of $G$. This is not obvious, because a priori $H$ is just defined as the fundamental group of the hyperbolizing cube $\hq$.

\begin{lemma}\label{lem:hypcube_quasiconvex_in_hyperbolized_group}
Let $X$ be compact. Then $\hgq$ is a quasiconvex subgroup of $\hg$.
\end{lemma}
\proof
By Lemma~\ref{lem:cell_is_hyperbolizing_cell}, we know that a hyperbolized complex   retracts to each of its tiles, each of which is homeomorphic to the hyperbolizing cell. In our setting this means that $\hc$ retracts to $\hq$, so in particular the inclusion $\hq\hookrightarrow \hc$ as a tile induces an injection $\hgq \hookrightarrow \hg$.
Since $X$ is compact, the group $\hg$ acts geometrically on $\uchc$. Moreover, the subgroup $\hgq$ acts geometrically on a tile, which is a closed convex subspace by Lemma~\ref{lem:strat_cells_convex}. Therefore $\hgq$ is quasiconvex by Lemma~\ref{lem:quasiconvex_stabilizers}.
\endproof

\begin{lemma}\label{lem:hypcube_quasiconvex_in_lattice}
The group $\hgq$ is a quasiconvex subgroup of $\Gamma$.
\end{lemma}
\proof
First of all we will prove that $\hgq$ naturally injects in $\Gamma$, by showing that there exists a (normal) cover $\hcalt Y$ of $\hm=\hh^n / \Gamma$ which retracts to $\hq$ (see Figure~\ref{fig:hypRnisThom}).
This would provide the desired injection
$$\hgq=\pi_1(\hq) \hookrightarrow \pi_1(\hcalt Y) \hookrightarrow \pi_1(\hm)=\Gamma .$$

To construct this cover, consider the  cubical complex $Y$ given by the standard cubulation of $\rr^n$ with vertices on $\zz^n$. 
Notice that $Y$ admits a standard folding $f:Y\to \square^n$, and that $Y$ is an admissible cubical complex.
Therefore we can consider the hyperbolized complex $\hcalt Y$.
As in the proof of Lemma~\ref{lem:hypcube_quasiconvex_in_hyperbolized_group},   Lemma~\ref{lem:cell_is_hyperbolizing_cell} implies that $\hcalt Y$ retracts onto any of its tiles, hence $\hgq$ injects in $\pi_1(\hcalt Y)$.

We now claim that $\hcalt Y$ is a (normal) covering space of $\hm$. 
For each $i=1,\dots,n$ consider the mirror of $Y$ given by $M_i=\{y_i=0\}$, and the hyperplane of $Y$ given by $H_i=\{y_i=\frac{1}{2}\}$. Let $m_{i}$ and $h_{i}$ be the reflections in $M_i$ and $H_i$ respectively, i.e.
\[ m_i:Y\to Y, m_i(y_1,\dots,y_i,\dots,y_n)= (y_1,\dots,-y_i,\dots,y_n)\]
\[ h_i:Y\to Y, h_i(y_1,\dots,y_i,\dots,y_n)= (y_1,\dots,1-y_i,\dots,y_n)\]
For each $i=1,\dots,n$, the group $D_i=\langle m_i,h_i \rangle$ is an infinite dihedral group of cubical isometries of $Y$. The group $D=\langle m_1,h_1,\dots m_n,h_n\rangle$ is isomorphic to the direct product $D_1\times \dots \times D_n$, and admits a representation into the group  $\igq$ of Euclidean isometries of the standard cube $\square^n$, in which $m_i$ acts trivially and $h_i$ acts as the standard reflection of $\square^n$ in the $i$--th coordinate.
By Lemma~\ref{lem:CD cell} we have an action of $\igq$ on $\hq$ by isometries, hence we  can induce an action of $D$ on $\hq$ by isometries such that the Charney--Davis map $\cd:\hq\to \square^n$ is equivariant.
Moreover, the standard folding $f:Y\to \square^n$  is clearly $D$--equivariant too, because it can be obtained by  reflecting in the mirrors of $Y$.
Since the two maps in the pullback square defining $\hcalt Y$ are $D$--equivariant (see Figure~\ref{fig:hypRnisThom}), we obtain an action of $D$ on $\hcalt Y$ by isometries.

Note that $t_i=h_im_i$ is the unit integer translation of $Y$ in the $i^{th}$ direction. As a result, $D$ contains a (normal) subgroup $T$ isomorphic to the group of integer translations $\zz^n$.
The action of $D$ on $\hcalt Y$ restricts to a free and properly discontinuous action of $T$ on $\hcalt Y$. A fundamental domain for this action is given by a single tile. Each tile is isometric to a hyperbolizing cube $\hq$, and the induced action identifies corresponding cells on opposite mirrors, recovering $\hm$ (see \S~\ref{sec:construction of CD cell} for more details about the construction of $\hq$.)
In particular $\mu_\Gamma:\hcalt Y \to \hcalt Y/T  \cong  \hm$ realizes the desired covering map, which covers the standard universal covering map $\mu:Y=\rr^n\to (S^1)^n$ (see Figure~\ref{fig:hypRnisThom}).

Finally, let us prove that $\hgq$ is quasiconvex in $\Gamma$. We know $\Gamma$  acts geometrically on $\hh^n$, permuting the stratification induced by the coordinate mirrors and their translates. The subgroup $\hgq$ stabilizes a $\Gamma$--cell, i.e.\ the closure of a connected component of the complement of such collection. This is a closed convex subspace, on which $\hgq$ acts geometrically (see \S\ref{sec:universal_cover} for details). In particular $\hgq$ is quasiconvex in $\Gamma$ by Lemma~\ref{lem:quasiconvex_stabilizers}.
\endproof

\begin{figure}[ht]
    \centering
    \begin{tikzpicture}
\node at (-1,1) (A) {$\hcalt Y$};
\node at (-1,-1) (B) {$Y$};
\node at (1,-1) (C) {$\square^n$};
\node at (1,1) (D) {$\hq$};
\node at (-3,-1) (E) {$(S^1)^n$};
\node at (-3,1) (F) {$\hm$};

\draw [->] (A) edge (B)   (A) edge (D) ;
\draw [->]  (B) edge (C) (D) edge (C);
\draw [->]  (A) edge (F) (F) edge (E);
\draw [->]  (B) edge (E);

\node at (0,1.25) {$\foldc$};
\node at (0,-1.25) {$\fold$};
\node at (1.25,0) {$\cd$};
\node at (1.3-2,0) {$\cd_Y$};
\node at (1.3-4,0) {$\cd_0$};
\node at (-2,1.25) {$\mu_\Gamma$};
\node at (-2,-1.25) {$\mu$};

\end{tikzpicture}
    \caption{$\hcalt Y$ as a covering space of $\hm$ that retracts to $\hq$.}
    \label{fig:hypRnisThom}
\end{figure}

\begin{remark}\label{rem:hyperbolization can be arithmetic}
In Lemma~\ref{lem:hypcube_quasiconvex_in_lattice} we have constructed a normal covering space  of $\hm$ by producing an action of $T=\zz^n$ by deck transformations on the hyperbolization $\hcalt Y$ of the standard integral cubulation of $\rr^n$. 
This covering space can also be defined as the covering space of $\hm$ corresponding to the kernel of the homomorphism $\Gamma=\pi_1(\hm)\to \zz^n$ induced by the  collapse map $\cd_0:\hm\to (S^1)^n$ obtained by applying the Pontryagin--Thom construction to $\hm$ with respect to suitable codimension--$1$ submanifolds (see \S\ref{sec:construction of CD cell} for details, and Figure~\ref{fig:hypRnisThom}).
Compact quotients of $\hcalt Y$, provide examples of closed hyperbolized manifolds which are finite covers of $\hm$. These are all genuine arithmetic hyperbolic manifolds.
\end{remark}

\subsubsection{Virtual specialness}\label{subsec:special}
A cubical complex is \textit{special} if it admits a local isometry into the Salvetti complex of a right-angled Artin group (see \cite{HW08}).
A group $G$ is \textit{virtually compact special} if there exist a finite index subgroup $G'\subseteq G$ and a compact special cubical complex $B$ such that $G'=\pi_1(B)$. 
This property passes from a Gromov hyperbolic group to its quasiconvex subgroups, as established in the following statement. This kind of arguments has appeared in the literature (see for instance \cite[Corollary 7.8]{HW08}). We include a proof for the reader's convenience.

\begin{lemma}\label{lem:quasiconvex_in_special}
Let $G$ be a Gromov hyperbolic group, and let $H$ be a quasiconvex subgroup. If $G$ is virtually compact special, then so is $H$. 
\end{lemma}
\proof

Let $G'$ be a finite index subgroup of $G$ and $B$ a compact special cubical complex such that $G'=\pi_1(B)$. 
By \cite[Remark 3.4, Lemma 3.13]{HW08} we can assume without loss of generality that $B$ is also non--positively curved.
The universal cover $\widetilde B$  is a  $\cat 0$ cubical complex. It is finite dimensional, uniformly locally finite, and Gromov hyperbolic, because $G'$ acts geometrically on it by covering transformations.

Let $H'=H\cap G'$. This is a finite--index subgroup of $H$, and a quasiconvex subgroup of $G'$. 
Since $G'$ is Gromov hyperbolic and acts geometrically on  $\widetilde B$, it follows that $H'$--orbits are quasiconvex. 
By \cite[Theorem H, or Corollary 2.29]{H08} or \cite[Theorem 1.2]{SW15}, there exists a cocompact convex core for $H'$, i.e.\ a convex subcomplex $\widetilde Y\subseteq \widetilde B$ on which $H'$ acts cocompactly.
Moreover, $H'$ acts by deck transformations, and the quotient $Y=\widetilde Y / H'$ is a compact non--positively curved cubical complex with $\pi_1(Y)=H'$. 
The convex embedding $\widetilde Y \hookrightarrow \widetilde B$ descends to a local isometry $Y\to  B$. Since $ B$ is special, by \cite[Corollary 3.9]{HW08} we obtain that $Y$ is special too. 
Therefore $H$ is   virtually compact special, as desired.
\endproof

\begin{remark}
In the previous proof we have the Gromov hyperbolic group $H'$ acting geometrically on the $\cat 0$ cubical complex $\widetilde Y$, so the fact that $H'$ is virtually compact special also follows from the celebrated theorem of Agol in \cite{AG13}. However here everything happens inside the special group $G'$, so one does not need Agol's result. 
\end{remark}

We now apply the previous lemma to the cell stabilizers for the action  of $\hq$ on $\uchc$, starting with the stabilizer of a tile.

\begin{lemma}\label{lem:hyperbolizing cube is vcs}
The group $\hgq$ is virtually compact special.
\end{lemma}
\proof
$\hgq$ is a quasiconvex subgroup of $\Gamma$ by Lemma~\ref{lem:hypcube_quasiconvex_in_lattice} and $\Gamma$ is virtually compact special by \cite[Theorem 1.6]{HW12}. Indeed, it is an arithmetic lattice in $\so$ by construction (see \S\ref{sec:CD hyperbolizing cell} or  \cite{CD95} for details).
The statement then follows from  Lemma~\ref{lem:quasiconvex_in_special}.
\endproof

Finally we prove the same result for all cell stabilizers.

\begin{lemma}\label{lem:cell stabilizers is vcs}
Let $X$ be compact. Then the cell stabilizers for the action of $\hg$ on $\uchc$ are quasiconvex and virtually compact special.
\end{lemma}
\proof
Let $\sigma$ be a cell in $\uchc$ and let $H$ be the stabilizer of $\sigma$ for the action of $\hg$ on $\uchc$.
Since $\sigma$ is a convex subset of $\uchc$ and $H$ acts geometrically on it, we conclude by Lemma~\ref{lem:quasiconvex_stabilizers} that $H$ is quasiconvex in $\hg$.

Arguing as in the proof of Lemma~\ref{lem:cube stab equals min vertex stab}, if $\tau$ is a tile containing $\sigma$, and $K$ is its stabilizer, then $H \subseteq K$.
Note that the folding map $\hc\to \hq$ provides an isomorphism of $K \cong \hgq$, under which $H$ is isomorphic to a quasiconvex subgroup of $\hgq$ (again by Lemma~\ref{lem:quasiconvex_stabilizers}). 
We know that $\hgq$ is Gromov hyperbolic (by Lemma~\ref{lem:hypcube_quasiconvex_in_hyperbolized_group} or Lemma~\ref{lem:hypcube_quasiconvex_in_lattice}) and virtually compact special (by Lemma~\ref{lem:hyperbolizing cube is vcs}).
So it follows from Lemma~\ref{lem:quasiconvex_in_special} that $H$ is virtually compact special too.
\endproof

\subsection{Specialization}
We are now ready to prove that the fundamental group $\hg$ of the hyperbolized complex $\hc$ is virtually compact special, when the original cubical complex $X$ is admissible and compact.
If the action of $\hg$ on $\dccx$ was proper, this would follow from Theorem~\ref{thm:dual cubical complex is CAT(0)}, Lemma~\ref{lem:dual action is cocompact} and Agol's main result from \cite{AG13}.
However, as observed in Remark~\ref{rem:dual action is not proper}, the action on $\dccx$ is {\bf not} proper.
We will use a result by Groves and Manning (see Theorem D in \cite{GM18}), which is  designed to deal with this situation. We report here their statement for the reader's convenience.

\begin{theorem}{\cite[Theorem D]{GM18}}
Suppose that $G$ is a Gromov hyperbolic group acting cocompactly on a $\cat 0$ cubical complex so that cell stabilizers are quasiconvex and virtually compact special. Then $G$ is virtually compact special.
\end{theorem}

Note that when authors of \cite{GM18} say ``virtually special'' they imply that the quotient is compact (see page 3 in \cite{GM18}). Also notice that they explicitly do not assume their complexes to be locally compact (see page 2).

\begin{theorem}\label{thm:main_vcs}
If $X$ is a compact admissible cubical complex and $\Gamma$ is a hyperbolizing lattice, then $\hg$ is virtually compact special Gromov hyperbolic group.
\end{theorem}
\proof
First of all, since $X$ is admissible, by Theorem~\ref{thm:dual cubical complex is CAT(0)} the dual cubical complex $\dccx$ is a $\cat 0$ cubical complex.
Moreover, since $X$ is compact, $\hg$ is a Gromov hyperbolic group by \eqref{item:CD hyperbolic} in Proposition~\ref{prop:CD complex}.
By Lemma~\ref{lem:dual action is cocompact} $\hg$ acts on $\dccx$ cocompactly by isometries. 

Let $C$ be a cube in $\dccx$ and let $H$ be its stabilizer.
By Lemma~\ref{lem:cube stab equals min vertex stab} $H$ coincides with the stabilizer of the vertex of minimal height in $C$. 
By Lemma~\ref{lem:vertex stab equals cell stab} this in turn coincides with the stabilizer of the corresponding dual cell in $\uchc$.
Therefore by Lemma~\ref{lem:cell stabilizers is vcs} $H$ is a quasiconvex subgroup of $\hg$ and it is also virtually compact special.
Finally, by \cite[Theorem D]{GM18} the group $\hg$ is virtually compact special.
\endproof

\printbibliography

\end{document}